\documentclass[10pt]{article}
\usepackage{hyperref}
\usepackage{amsmath}
\usepackage{amsfonts}
\usepackage{amssymb}
\usepackage{amsthm}

\usepackage[shortlabels]{enumitem}
\usepackage{verbatim}
\usepackage{graphicx}
\usepackage{microtype}
\usepackage{fullpage}
\usepackage{xparse}
\usepackage{mathrsfs}
\usepackage[capitalise]{cleveref}
\usepackage{mleftright}

\crefname{thm}{Theorem}{Theorems}
\Crefname{thm}{Theorem}{Theorems}
\crefname{conj}{Conjecture}{Conjectures}
\Crefname{conj}{Conjecture}{Conjectures}
\crefname{prop}{Proposition}{Propositions}
\Crefname{prop}{Proposition}{Propositions}
\crefname{cor}{Corollary}{Corollaries}
\Crefname{cor}{Corollary}{Corollaries}
\crefname{defn}{Definition}{Definitions}
\Crefname{defn}{Definition}{Definitions}
\crefname{rmk}{Remark}{Remarks}
\Crefname{rmk}{Remark}{Remarks}
\crefname{prob}{Problem}{Problems}
\Crefname{prob}{Problem}{Problems}

\crefname{enumi}{}{}
\Crefname{enumi}{}{}

\crefname{figure}{Figure}{Figures}
\Crefname{figure}{Figure}{Figures}

\crefformat{equation}{(#2#1#3)}
\crefrangeformat{equation}{(#3#1#4) to~(#5#2#6)}
\crefmultiformat{equation}{(#2#1#3)}%
{ and~(#2#1#3)}{, (#2#1#3)}{, and~(#2#1#3)}

\Crefformat{equation}{(#2#1#3)}
\Crefrangeformat{equation}{(#3#1#4) to~(#5#2#6)}
\Crefmultiformat{equation}{(#2#1#3)}%
{ and~(#2#1#3)}{, (#2#1#3)}{, and~(#2#1#3)}

\begin{document}

\NewDocumentCommand{\C}{}{{\mathbb{C}}}
\NewDocumentCommand{\R}{}{{\mathbb{R}}}
\NewDocumentCommand{\Q}{}{{\mathbb{Q}}}
\NewDocumentCommand{\Z}{}{{\mathbb{Z}}}
\NewDocumentCommand{\N}{}{{\mathbb{N}}}
\NewDocumentCommand{\M}{}{{\mathbb{M}}}
\NewDocumentCommand{\grad}{}{\nabla}
\NewDocumentCommand{\sA}{}{\mathcal{A}}
\NewDocumentCommand{\sF}{}{\mathcal{F}}
\NewDocumentCommand{\sFh}{}{\widehat{\mathcal{F}}}
\NewDocumentCommand{\sH}{}{\mathcal{H}}
\NewDocumentCommand{\sD}{}{\mathcal{D}}
\NewDocumentCommand{\sB}{}{\mathcal{B}}
\NewDocumentCommand{\sC}{}{\mathcal{C}}
\NewDocumentCommand{\sE}{}{\mathcal{E}}
\NewDocumentCommand{\sL}{}{\mathcal{L}}
\NewDocumentCommand{\sT}{}{\mathcal{T}}
\NewDocumentCommand{\sO}{}{\mathcal{O}}
\NewDocumentCommand{\sP}{}{\mathcal{P}}
\NewDocumentCommand{\sQ}{}{\mathcal{Q}}
\NewDocumentCommand{\sR}{}{\mathcal{R}}
\NewDocumentCommand{\sS}{}{\mathcal{S}}
\NewDocumentCommand{\sSh}{}{\widehat{\mathcal{S}}}
\NewDocumentCommand{\sM}{}{\mathcal{M}}
\NewDocumentCommand{\sI}{}{\mathcal{I}}
\NewDocumentCommand{\sK}{}{\mathcal{K}}
\NewDocumentCommand{\Span}{}{\mathrm{span}}
\NewDocumentCommand{\fM}{}{\mathfrak{M}}
\NewDocumentCommand{\fN}{}{\mathfrak{N}}
\NewDocumentCommand{\fX}{}{\mathfrak{X}}
\NewDocumentCommand{\fY}{}{\mathfrak{Y}}
\NewDocumentCommand{\gammat}{}{\tilde{\gamma}}
\NewDocumentCommand{\ct}{}{\tilde{c}}
\NewDocumentCommand{\bt}{}{\tilde{b}}
\NewDocumentCommand{\ch}{}{\hat{c}}
\NewDocumentCommand{\Ut}{}{\tilde{U}}
\NewDocumentCommand{\Gt}{}{\widetilde{G}}
\NewDocumentCommand{\Ft}{}{\widetilde{F}}
\NewDocumentCommand{\Fh}{}{\widehat{F}}
\NewDocumentCommand{\Vt}{}{\tilde{V}}
\NewDocumentCommand{\ah}{}{\hat{a}}
\NewDocumentCommand{\at}{}{\tilde{a}}
\NewDocumentCommand{\Yh}{}{\widehat{Y}}
\NewDocumentCommand{\Yt}{}{\widetilde{Y}}
\NewDocumentCommand{\gt}{}{\tilde{g}}
\NewDocumentCommand{\Ah}{}{\widehat{A}}
\NewDocumentCommand{\Ch}{}{\widehat{C}}
\NewDocumentCommand{\At}{}{\widetilde{A}}
\NewDocumentCommand{\Bt}{}{\widetilde{B}}
\NewDocumentCommand{\Bh}{}{\widehat{B}}
\NewDocumentCommand{\Vol}{m}{\mathrm{Vol}(#1)}
\NewDocumentCommand{\BVol}{m}{\mathrm{Vol}\left(#1\right)}
\NewDocumentCommand{\fg}{}{\mathfrak{g}}
\NewDocumentCommand{\Div}{}{\mathrm{div}}
\NewDocumentCommand{\etai}{}{\hat{\eta}}
\NewDocumentCommand{\etat}{}{\tilde{\eta}}

\NewDocumentCommand{\Eadmis}{}{\mathrm{E}}

\NewDocumentCommand{\Deriv}{}{\mathscr{D}}
\NewDocumentCommand{\BofA}{}{\mathscr{B}}
\NewDocumentCommand{\ADeriv}{}{\mathscr{A}}

\NewDocumentCommand{\Xa}{}{X^{(\alpha)}}
\NewDocumentCommand{\Va}{}{V^{(\alpha)}}

\NewDocumentCommand{\transpose}{}{\top}

\NewDocumentCommand{\ICond}{}{\sC}

\NewDocumentCommand{\LebDensity}{}{\sigma_{\mathrm{Leb}}}


\NewDocumentCommand{\Lie}{m}{\sL_{#1}}

\NewDocumentCommand{\ZygSymb}{}{\mathscr{C}}

\NewDocumentCommand{\Zyg}{m o}{\IfNoValueTF{#2}{\ZygSymb^{#1}}{\ZygSymb^{#1}(#2) }}
\NewDocumentCommand{\ZygX}{m m o}{\IfNoValueTF{#3}{\ZygSymb^{#2}_{#1}}{\ZygSymb^{#2}_{#1}(#3) }}

\NewDocumentCommand{\CSpace}{m o}{\IfNoValueTF{#2}{C(#1)}{C(#1;#2)}}

\NewDocumentCommand{\CjSpace}{m o o}{\IfNoValueTF{#2}{C^{#1}}{ \IfNoValueTF{#3}{ C^{#1}(#2)}{C^{#1}(#2;#3) } }  }

\NewDocumentCommand{\CjSpaceloc}{m o o}{\IfNoValueTF{#2}{C^{#1}_{\mathrm{loc}}}{ \IfNoValueTF{#3}{ C^{#1}_{\mathrm{loc}}(#2)}{C^{#1}_{\mathrm{loc}}(#2;#3) } }  }

\NewDocumentCommand{\CXjSpace}{m m o}{\IfNoValueTF{#3}{C^{#2}_{#1}}{ C^{#2}_{#1}(#3) } }

\NewDocumentCommand{\ASpace}{m m o}{\mathscr{A}^{#1,#2}\IfNoValueTF{#3}{}{(#3)}}

\NewDocumentCommand{\AXSpace}{m m m o}{\mathscr{A}_{#1}^{#2,#3}\IfNoValueTF{#4}{}{(#4)}}

\NewDocumentCommand{\ComegaSpace}{m o o}{\IfNoValueTF{#2}{\CjSpace{\omega,#1}}{
\IfNoValueTF{#3}
{\CjSpace{\omega,#1}[#2]}
{\CjSpace{\omega,#1}[#2][#3]}
} 
}

\NewDocumentCommand{\CXomegaSpace}{m m o o}{\IfNoValueTF{#3}{\CXjSpace{#1}{\omega,#2}}{
\IfNoValueTF{#4}
{\CXjSpace{#1}{\omega,#2}[#3]}
{\CXjSpace{#1}{\omega,#2}[#3][#4]}
} 
}

\NewDocumentCommand{\ANorm}{m m m o}{\IfNoValueTF{#4}{\Norm{#1}[ \ASpace{#2}{#3} ]}{ \Norm{#1}[ \ASpace{#2}{#3}[#4] ] }}

\NewDocumentCommand{\AXNorm}{m m m m o}{\IfNoValueTF{#5}{\Norm{#1}[ \AXSpace{#2}{#3}{#4} ]}{ \Norm{#1}[ \AXSpace{#2}{#3}{#4}[#5] ] }}
\NewDocumentCommand{\BAXNorm}{m m m m o}{\IfNoValueTF{#5}{\BNorm{#1}[ \AXSpace{#2}{#3}{#4} ]}{ \BNorm{#1}[ \AXSpace{#2}{#3}{#4}[#5] ] }}

\NewDocumentCommand{\ComegaNorm}{m m o o}{\IfNoValueTF{#3}{ \Norm{#1}[\ComegaSpace{#2}] }
{ 
\IfNoValueTF{#4}
{\Norm{#1}[\ComegaSpace{#2}[#3]] }
{\Norm{#1}[\ComegaSpace{#2}[#3][#4]] }
}
}

\NewDocumentCommand{\BComegaNorm}{m m o o}{\IfNoValueTF{#3}{ \BNorm{#1}[\ComegaSpace{#2}] }
{ 
\IfNoValueTF{#4}
{\BNorm{#1}[\ComegaSpace{#2}[#3]] }
{\BNorm{#1}[\ComegaSpace{#2}[#3][#4]] }
}
}

\NewDocumentCommand{\CXomegaNorm}{m m m o o}{\IfNoValueTF{#4}{ \Norm{#1}[\CXomegaSpace{#2}{#3}] }
{ 
\IfNoValueTF{#5}
{\Norm{#1}[\CXomegaSpace{#2}{#3}[#4]] }
{\Norm{#1}[\CXomegaSpace{#2}{#3}[#4][#5]] }
}
}

\NewDocumentCommand{\BCXomegaNorm}{m m m o o}{\IfNoValueTF{#4}{ \BNorm{#1}[\CXomegaSpace{#2}{#3}] }
{ 
\IfNoValueTF{#5}
{\BNorm{#1}[\CXomegaSpace{#2}{#3}[#4]] }
{\BNorm{#1}[\CXomegaSpace{#2}{#3}[#4][#5]] }
}
}

\NewDocumentCommand{\HSpace}{m m o o}{\IfNoValueTF{#3}{C^{#1,#2}}{ \IfNoValueTF{#4} {C^{#1,#2}(#3)} {C^{#1,#2}(#3;#4)} }}

\NewDocumentCommand{\HXSpace}{m m m o}{\IfNoValueTF{#4}{C_{#1}^{#2,#3}}{  {C_{#1}^{#2,#3}(#4)}  }}

\NewDocumentCommand{\ZygSpace}{m o o}{\IfNoValueTF{#2}{\ZygSymb^{#1}}{ \IfNoValueTF{#3} { \ZygSymb^{#1}(#2) }{\ZygSymb^{#1}(#2;#3) } } }

\NewDocumentCommand{\ZygXSpace}{m m o}{\IfNoValueTF{#3}{\ZygSymb^{#2}_{#1}}{\ZygSymb^{#2}_{#1}(#3) }}

\NewDocumentCommand{\Norm}{m o}{\IfNoValueTF{#2}{\| #1\|}{\|#1\|_{#2} }}
\NewDocumentCommand{\BNorm}{m o}{\IfNoValueTF{#2}{\left\| #1\right\|}{\left\|#1\right\|_{#2} }}

\NewDocumentCommand{\CjNorm}{m m o o}{ \IfNoValueTF{#3}{ \Norm{#1}[\CjSpace{#2}]} { \IfNoValueTF{#4}{\Norm{#1}[\CjSpace{#2}[#3]]} {\Norm{#1}[\CjSpace{#2}[#3][#4]]}  }  }

\NewDocumentCommand{\CNorm}{m m o}{\IfNoValueTF{#3}{\Norm{#1}[\CSpace{#2}]}{\Norm{#1}[\CSpace{#2}[#3]]}}

\NewDocumentCommand{\BCNorm}{m m}{\BNorm{#1}[\CSpace{#2}]}

\NewDocumentCommand{\CXjNorm}{m m m o}{\Norm{#1}[ 
\IfNoValueTF{#4}
{\CXjSpace{#2}{#3}}
{\CXjSpace{#2}{#3}[#4]}
]}

\NewDocumentCommand{\BCXjNorm}{m m m o}{\BNorm{#1}[ 
\IfNoValueTF{#4}
{\CXjSpace{#2}{#3}}
{\CXjSpace{#2}{#3}[#4]}
]}

\NewDocumentCommand{\LpNorm}{m m o o}{
\Norm{#2}[L^{#1}
\IfNoValueTF{#3}{}{
(#3
\IfNoValueTF{#4}{}{;#4}
)
}
]
}

\NewDocumentCommand{\BanachSpace}{}{\mathscr{X}}
\NewDocumentCommand{\BanachAlgebra}{}{\mathscr{Y}}
\NewDocumentCommand{\OpsY}{}{\mathcal{O}(\BanachAlgebra)}

\NewDocumentCommand{\Compact}{}{\mathcal{K}}


\NewDocumentCommand{\BCjNorm}{m m o}{ \IfNoValueTF{#3}{ \BNorm{#1}[C^{#2}]} { \BNorm{#1}[C^{#2}(#3)]  }  }

\NewDocumentCommand{\HNorm}{m m m o o}{ \IfNoValueTF{#4}{ \Norm{#1}[\HSpace{#2}{#3}]} {
\IfNoValueTF{#5}
{\Norm{#1}[\HSpace{#2}{#3}[#4]]}
{\Norm{#1}[\HSpace{#2}{#3}[#4][#5]] }
}  }

\NewDocumentCommand{\HXNorm}{m m m m o}{ \IfNoValueTF{#5}{ \Norm{#1}[\HXSpace{#2}{#3}{#4}]} {
{\Norm{#1}[\HXSpace{#2}{#3}{#4}[#5]]}
}  }

\NewDocumentCommand{\BHXNorm}{m m m m o}{ \IfNoValueTF{#5}{ \BNorm{#1}[\HXSpace{#2}{#3}{#4}]} {
{\BNorm{#1}[\HXSpace{#2}{#3}{#4}[#5]]}
}  }

\NewDocumentCommand{\ZygNorm}{m m o o}{ \IfNoValueTF{#3}{ \Norm{#1}[\ZygSpace{#2}]} {
\IfNoValueTF{#4}
{\Norm{#1}[\ZygSpace{#2}[#3]]}  
{\Norm{#1}[\ZygSpace{#2}[#3][#4]]}
}  }

\NewDocumentCommand{\ZygXNorm}{m m m o}{\Norm{#1}[ 
\IfNoValueTF{#4}
{\ZygXSpace{#2}{#3}}
{\ZygXSpace{#2}{#3}[#4]}
]}

\NewDocumentCommand{\BZygXNorm}{m m m o}{\BNorm{#1}[ 
\IfNoValueTF{#4}
{\ZygXSpace{#2}{#3}}
{\ZygXSpace{#2}{#3}[#4]}
]}


\NewDocumentCommand{\diff}{o m}{\IfNoValueTF{#1}{\frac{\partial}{\partial #2}}{\frac{\partial^{#1}}{\partial #2^{#1}} }}

\NewDocumentCommand{\dt}{o}{\IfNoValueTF{#1}{\diff{t}}{\diff[#1]{t} }}

\NewDocumentCommand{\Zygad}{m}{\{ #1\}}

\NewDocumentCommand{\Had}{m}{\langle #1\rangle}

\NewDocumentCommand{\DiffOp}{m}{\Delta_{#1}}


\NewDocumentCommand{\SSFunctionSpacesSection}{}{Section 2}
\NewDocumentCommand{\SSStrangeZygSpace}{}{Remark 2.1}
\NewDocumentCommand{\SSBeyondManifold}{}{Section 2.2.1}
\NewDocumentCommand{\SSNormsAreInv}{}{Proposition 2.3}
\NewDocumentCommand{\SSDefineVectDeriv}{}{Remark 2.4}

\NewDocumentCommand{\SSSectionMoreOnAssumptions}{}{Section 4.1}
\NewDocumentCommand{\SSMainResult}{}{Theorem 4.7}
\NewDocumentCommand{\SSLemmaMoreOnAssump}{}{Proposition 4.14}

\NewDocumentCommand{\SSDivideWedge}{}{Section 5}
\NewDocumentCommand{\SSDerivWedge}{}{Lemma 5.1}

\NewDocumentCommand{\SSDensities}{}{Section 6}
\NewDocumentCommand{\SSDensitiesTheorem}{}{Theorem 6.5}
\NewDocumentCommand{\SSDensityCor}{}{Corollary 6.6}

\NewDocumentCommand{\SSScaling}{}{Section 7}
\NewDocumentCommand{\SSNSW}{}{Section 7.1}
\NewDocumentCommand{\SSHormandersCondition}{}{Section 7.1.1}
\NewDocumentCommand{\SSGenSubR}{}{Section 7.3}
\NewDocumentCommand{\SSGenSubResult}{}{Theorem 7.6}

\NewDocumentCommand{\SSCompareFunctionSpaces}{}{Lemma 8.1}
\NewDocumentCommand{\SSZygIsAlgebra}{}{Proposition 8.3}
\NewDocumentCommand{\SSBiggerNormMap}{}{Proposition 8.6}
\NewDocumentCommand{\SSCompareEuclidNorms}{}{Proposition 8.12}

\NewDocumentCommand{\SSDeriveODE}{}{Proposition 9.1}
\NewDocumentCommand{\SSExistODE}{}{Proposition 9.4}
\NewDocumentCommand{\SSExistXiOne}{}{Lemma 9.23}
\NewDocumentCommand{\SSDifferentOneAdmis}{}{Proposition 9.26}
\NewDocumentCommand{\SSCXjNormWedgeQuotient}{}{Lemma 9.32}
\NewDocumentCommand{\SSExistXiTwo}{}{Lemma 9.35}
\NewDocumentCommand{\SSComputefjzero}{}{Lemma 9.38}
\NewDocumentCommand{\SSSectionDensities}{}{Section 9.4}

\NewDocumentCommand{\SSProofInjectiveImmersion}{}{Appendix A}
\NewDocumentCommand{\SSFinerTopology}{}{Lemma A.1}


\newtheorem{thm}{Theorem}[section]
\newtheorem{cor}[thm]{Corollary}
\newtheorem{prop}[thm]{Proposition}
\newtheorem{lemma}[thm]{Lemma}
\newtheorem{conj}[thm]{Conjecture}
\newtheorem{prob}[thm]{Problem}

\theoremstyle{remark}
\newtheorem{rmk}[thm]{Remark}

\theoremstyle{definition}
\newtheorem{defn}[thm]{Definition}

\theoremstyle{definition}
\newtheorem{assumption}[thm]{Assumption}

\theoremstyle{remark}
\newtheorem{example}[thm]{Example}

\numberwithin{equation}{section}

\title{Coordinates Adapted to Vector Fields III: Real Analyticity}
\author{Brian Street\footnote{This material is partially based upon work supported by the National Science Foundation under Grant No.\ 1440140, while the author was in residence at the Mathematical Sciences Research Institute in Berkeley, California, during the spring semester of 2017.  The author was also partially supported by National Science Foundation Grant Nos.\ 1401671 and 1764265.}}
\date{}

\maketitle

\begin{abstract}
Given a finite collection of $C^1$ vector fields on a $C^2$ manifold which span the tangent space at every point, we consider the question of when there is locally a coordinate system in which these vector fields are real analytic.  We give necessary and sufficient, coordinate-free conditions for the existence of such a coordinate system.  Moreover, we present a quantitative study of these coordinate charts.  This is the third part in a three-part series of papers.  The first part, joint with Stovall, lay the groundwork for the coordinate system we use in this paper and showed how such coordinate charts can be viewed as scaling maps for sub-Riemannian geometry.  The second part dealt with the analogous questions with real analytic replaced by $C^\infty$ and Zygmund spaces.

\end{abstract}


\section{Introduction}
Let $X_1,\ldots, X_q$ be $C^1$ vector fields on a $C^2$ manifold $M$, which span the tangent space at every point of $M$.
In this paper, we investigate the following three closely related questions:
\begin{enumerate}[(i)]
\item\label{Item::Intro::LocalQual} When is there a coordinate system near a fixed point $x_0\in M$ such that the vector fields $X_1,\ldots, X_q$
are real analytic in this coordinate system?
\item\label{Item::Intro::GlobalQual} When is there a real analytic manifold structure on $M$, compatible with its $C^2$ structure, such that $X_1,\ldots, X_q$
are real analytic with respect to this structure?  When such a structure exists, we will see it is unique.
\item\label{Item::Intro::Quant} When there is a coordinate system as in \cref{Item::Intro::LocalQual}, how can we pick it so that $X_1,\ldots, X_q$
are ``normalized'' in this coordinate system in a quantitative way which is useful for applying techniques from analysis?
\end{enumerate}
We present necessary and sufficient conditions for \cref{Item::Intro::LocalQual,Item::Intro::GlobalQual}, and under these conditions give a quantitative
answer to \cref{Item::Intro::Quant}.  This is the third part in a three part series of papers.  In the first two parts \cite{StovallStreet,StreetII}, the same questions
were investigated where ``real analytic'' was replaced by Zygmund spaces.

The first paper in the series \cite{StovallStreet}, joint with Stovall, was based on methods from ODEs, while the second paper \cite{StreetII} sharpened the results from the first paper using
methods from PDEs.  In this paper, we take the results from the first paper as a starting point, and use additional methods from ODEs
to answer the above questions.  Thus, this paper does not use any methods from PDEs.

The coordinate charts from \cref{Item::Intro::Quant} can be viewed as scaling maps in sub-Riemannian geometry.
When viewed in this light, these results can be seen as a continuation of results initiated by Nagel, Stein, and Wainger \cite{NagelSteinWaingerBallsAndMetrics}
and C.\ Fefferman and S\'anchez-Calle \cite{FeffermanSanchezCalleFundamentalSoltuions}, and furthered by Tao and Wright
\cite{TaoWrightLpImproving} and the author \cite{S}.
See \cref{Section::Role::RA,Section::Scaling} for a description of this.

This paper is a continuation of the first part of the series \cite{StovallStreet}.  That paper contains several applications and motivations for the types of results described in this paper.
It also contains a more leisurely introduction to some of the definitions and results in this paper, though we include all the necessary definitions so that the statement of the results is self-contained.

	\subsection{The role of real analyticity}\label{Section::Role::RA}
Two important ways the main results of this paper can be used are:
\begin{itemize}
\item They give necessary and sufficient, coordinate free, conditions on a collection of $C^1$ vector fields, which span the tangent space, for there to exist a coordinate system in which these vector fields are real analytic.
\item They give scaling maps adapted to real analytic sub-Riemannian geometries, which are useful for questions from harmonic analysis.
\end{itemize}
The first way seems to be new.  The second way has a long history and similar results have been used in several areas of harmonic analysis.  We now turn to describing some of this; see also \cref{Section::Scaling::BeyondHormander}.

Real analytic vector fields have important applications in several types of questions from harmonic analysis.
Since the original work of H\"ormander \cite{HormanderHypoellipticSecondOrderDiffEq}, $\CjSpace{\infty}$ vector fields satisfying H\"ormader's condition\footnote{A finite collection of vector fields satisfies ``H\"ormander's condition'' if the Lie algebra generated by the vector fields spans the tangent space at every point.}
have played a central role in several areas.
Nagel, Stein, and Wainger \cite{NagelSteinWaingerBallsAndMetrics} developed a quantitative theory of the sub-Riemannian geometries induced by H\"ormander vector fields.  In particular, they introduced scaling maps
adapted to H\"ormander vector fields which allowed the use of many techniques from harmonic analysis to be generalized to the setting of sub-Riemannian manifolds.  These ideas have been used in many different ways, including applications
to partial differential equations defined by vector fields and singular Radon transforms.  
See the notes at the end of Chapter 2 of \cite{StreetMultiParamSingInt} for a history of some of these ideas.

A finite collection of real analytic vector fields does not necessarily satisfy H\"ormander's condition; however, it does satisfy a generalization of this condition:  the $C^\infty$ module generated by the vector fields and their commutators of all orders is locally finitely generated (as a $C^\infty$ module).
This was first noted by Lobry \cite{LobryControlabiliteDesSystems} and is a simple consequence of the Weierstrass preparation theorem (see \cref{Section::ScalingRevis}).  Because of this, it is possible to generalize the quantitative theory of Nagel, Stein, and Wainger to a setting  which applies to any finite collection
of real analytic vector fields, whether or not they satisfy H\"ormander's condition.  The techniques required for this generalization use ideas of Tao and Wright \cite{TaoWrightLpImproving}
and the author \cite{S}.  In the context of the quantitative theory of sub-Riemannian geometry applied to questions in analysis, this seems to have been first explicitly used by the author and Stein \cite{SteinStreetIII} to study singular Radon transforms.

Thus, real analytic vector fields hold a special place: the quantitative scaling techniques  used to study H\"ormander vector fields can often be applied to real analytic vector fields, whether or not they satisfy H\"ormander's condition.
For many such applications, the scaling maps developed in \cite{S} are sufficient; however, in the context of real analytic vector fields, the theory from that paper has several deficiencies which are fixed in this paper.  One major deficiency is that if one starts with real analytic
vector fields, the scaling theorems from \cite{S} only guarantee quantitive bounds on the $C^m$ norms of the rescaled vector fields and no estimates on their real analyticity.  
Thus, when one applies the results from \cite{S} to a real analytic setting, the real analyticity is destroyed.
In this paper, we show the rescaled vector fields are real analytic and give appropriate quantitative control
of this fact.  This is described in \cref{Section::Scaling::BeyondHormander}.

\begin{rmk}\label{Rmk::Intro::BadAtSingular}
The results from this paper are useful even when considering some very classical settings.  Indeed, suppose $V_1,\ldots, V_r$ are real analytic vector fields on an open set $\Omega\subseteq \R^n$.  The classical
Frobenius theorem applies to foliate $\Omega$ into leaves; each leaf is a real analytic, injectively immersed sub-manifold (see \cite{HermannOnTheAccessibilityProblemInControlTheory,NaganoLienarDifferentialSystemsWithSingularities,LobryControlabiliteDesSystems,SussmanOrbitsOfFamiliesOfVectorFieldsAndIntegrabilityOfDistributions}).
This may be a singular foliation:  the various leaves may have different dimension.  Near a point where the dimension changes (a ``singular point''), classical proofs do not give good quantitative control on the real analytic coordinate systems which define the leaves:  classical proofs ``blow up'' near a singular point.
Our methods give useful, uniform quantitative control near such a singular point and avoid this blow up, in a certain sense.  See \cref{Rmk::BeyondHor::UsefulForSingular}.
\end{rmk}

\noindent \textbf{Outline of the paper:}  In \cref{Section::FuncSpace}, we present the function spaces we use.  There are two types of function spaces: the standard ones on Euclidean space described in \cref{Section::FuncSpace::Euclid},
and analogs on a $C^2$ manifold endowed with a finite collection of $C^1$ vector fields described in \cref{Section::FuncSpace::Mfld}.
In \cref{Section::Results} we describe the main results of this paper, starting with the qualitative results in \cref{Section::Results::Qual}  (i.e., \cref{Item::Intro::LocalQual} and \cref{Item::Intro::GlobalQual}
from the beginning of the introduction), and then turning to the quantitative results in \cref{Section::Results::Quant} (i.e., \cref{Item::Intro::Quant}). 
In \cref{Section::FuncSpaceRev} we further study the function spaces introduced in \cref{Section::FuncSpace}.
In \cref{Section::Scaling} we describe how the quantitative results can be seen as scaling maps in sub-Riemannian type geometries.
In \cref{Section::Part1} we describe the results we use from \cite{StovallStreet}.  In \cref{Section::Proofs} we prove the main results of this paper.
In \cref{Section::DensitiesInEuclidean} we describe the special case of some of our results when working on Euclidean space with Lebesgue measure (which is the most common application).
Finally, in \cref{Section::ScalingRevis} we prove the results concerning scaling from \cref{Section::Scaling}.

\section{Function Spaces}\label{Section::FuncSpace}
Before we can state our main results, we need to introduce the function spaces we use.
As described in \cite{StovallStreet}, we make a distinction between function spaces on subsets of $\R^n$
and function spaces on a $C^2$ manifold $M$.
On $\R^n$ we use the standard coordinate system to define the usual function spaces.
On an abstract $C^2$ manifold $M$, we do not have access to any one natural coordinate system
and so it does not make sense to discuss, for example, real analytic functions on $M$.
However, if $M$ is endowed with $C^1$ vector fields $X_1,\ldots, X_q$, we are able to define
what it means to be real analytic with respect to these vector fields, and that is how we shall proceed.
The notion of a function being real analytic with respect to a finite collection of vector fields is a special case
of a general notion due to Nelson \cite{NelsonAnalyticVectors}. 
Throughout the paper, $B^n(\delta)$ denotes the ball of radius $\delta>0$, centered at $0$, in $\R^n$.

\subsection{Function Spaces on Euclidean Space}\label{Section::FuncSpace::Euclid}
Let $\Omega\subseteq \R^n$ be an open set.  
We have the usual Banach space of bounded, continuous functions on $\Omega$:
\begin{equation*}
\CSpace{\Omega}
:=\{f:\Omega\rightarrow \C \: |\: f\text{ is continuous and bounded}\}, \quad \CNorm{f}{\Omega}
:=\sup_{x\in \Omega}|f(x)|.
\end{equation*}
%
We next define two closely related spaces of real analytic functions on $\R^n$.  
For $r>0$ let $B^n(r)$ be the ball of radius $r$ in $\R^n$, centered at $0$.
We define $\ASpace{n}{r}$ to be the space of those $f\in \CSpace{B^n(r)}$ such that
$f(t)=\sum_{\alpha\in \N^n} \frac{c_\alpha}{\alpha!} t^{\alpha}$, $\forall t\in B^n(r)$, where
\begin{equation*}
\ANorm{f}{n}{r}:=\sum_{\alpha\in \N^n} \frac{|c_{\alpha}|}{\alpha!} r^{|\alpha|}<\infty.
\end{equation*}
For $\Omega\subseteq \R^n$ open, we let $f\in \ComegaSpace{r}[\Omega]$ consist of those $f\in \CjSpace{\infty}[\Omega]$
such that 
\begin{equation}\label{Eqn::FuncEuclid::ComegaNorm}
\ComegaNorm{f}{r}[\Omega]:=\sum_{\alpha\in \N^n} \frac{ \CNorm{\partial_x^{\alpha}f}{\Omega} }{\alpha!} r^{|\alpha|}<\infty.
\end{equation}
For the relationship between $\ASpace{n}{r}$ and $\ComegaSpace{r}[\Omega]$ see \cref{Lemma::FiuncSpaceRev::Euclid::Compare}.
We set 
$$\CjSpace{\omega}[\Omega]:=\bigcup_{r>0} \ComegaSpace{r}[\Omega].$$
We say $f\in \CjSpaceloc{\omega}[\Omega]$ if for all $x\in \Omega$, there exists an open neighborhood $U\subseteq \Omega$ of $x$, with $f\big|_U\in \CjSpace{\omega}[U]$.
It is easy to see that $\CjSpaceloc{\omega}[\Omega]$ is the usual space of real analytic functions on $\Omega$.

If $\BanachSpace$ is a Banach space, we define the same spaces taking values in $\BanachSpace$ in the obvious way, and denote these
spaces by $\CSpace{\Omega}[\BanachSpace]$,
 $\ASpace{n}{r}[\BanachSpace]$,  $\ComegaSpace{r}[\Omega][\BanachSpace]$,
and $\CjSpace{\omega}[\Omega][\BanachSpace]$.
When we have a vector field $X$ on $\Omega$, we identify $X=\sum_{j=1}^n a_j(x)\diff{x_j}$ with the function $(a_1,\ldots, a_n):\Omega\rightarrow \R^n$.
It therefore makes sense to consider quantities like $\ComegaNorm{X}{r}[\Omega][\R^n]$.

\subsection{Function Spaces on Manifolds}\label{Section::FuncSpace::Mfld}
Let $X_1,\ldots, X_q$ be $C^1$ vector fields on a connected $C^2$ manifold $M$.
Define the Carnot-Carath\'eodory ball
associated to $X_1,\ldots, X_q$, centered at $x\in M$, of radius $\delta>0$ by
\begin{equation}\label{Eqn::FuncManifold::DefnBX}
\begin{split}
B_X(x,\delta):=
\Bigg\{ y\in M \: \bigg| \: &\exists \gamma:[0,1]\rightarrow M, \gamma(0)=x, \gamma(1)=y, \gamma'(t)=\sum_{j=1}^q a_j(t) \delta X_j(\gamma(t)),
\\& a_j\in L^\infty([0,1]), \BNorm{\sum_{j=1}^q |a_j|^2}[L^\infty]<1\Bigg\},
\end{split}
\end{equation}
and for $y\in M$ set
\begin{equation}\label{Eqn::FuncManifold::Defnrho}
\rho(x,y)=\inf\{\delta>0 : y\in B_X(x,\delta)\}.
\end{equation}
$\rho$ is called a sub-Riemannian distance.  See \cref{Rmk::FuncMfld::DefineDeriv} for the precise definition of $\gamma'(t)$ used in \cref{Eqn::FuncManifold::DefnBX}.

We use ordered multi-index notation $X^{\alpha}$.  Here $\alpha$ denotes a list of elements of $\{1,\ldots q\}$ and $|\alpha|$
denotes the length of the list.  For example, $X^{(2,1,3,1)}=X_2X_1X_3X_1$ and $|(2,1,3,1)|=4$.

We have the usual Banach space of bounded continuous functions on $M$:
\begin{equation*}
\CSpace{M}
:=\{f:M\rightarrow \C\: |\: f\text{ is continuous and bounded}\}, \quad
\CNorm{f}{M}
:=\sup_{x\in M}|f(x)|.
\end{equation*}
%
Next, we introduce what it means to be real analytic with respect to $X_1,\ldots, X_q$.  Following the setting in $\R^n$, we introduce two versions of this.
Given $x_0\in M$ and $r>0$ we define $\AXSpace{X}{x_0}{r}$ to be the space of those $f\in \CSpace{M}$ such that
\begin{equation*}
h(t_1,\ldots, t_q):=g(e^{t_1 X_1+\cdots +t_q X_q}x_0)\in \ASpace{q}{r},
\end{equation*}
here we are assuming $e^{t_1 X_1+\cdots +t_qX_q}x_0$ exists for $(t_1,\ldots, t_q)\in B^q(r)$ (see \cref{Defn::QuantRes::sC}).
We define $\AXNorm{f}{X}{x_0}{r}:=\ANorm{h}{q}{r}$.  Note that $\AXNorm{f}{X}{x_0}{r}$ depends only on the values of $f(y)$
where $y=e^{t_1X_1+\cdots+t_qX_q}x_0$ and $(t_1,\ldots, t_q)\in B^q(r)$; thus this is merely a semi-norm.

For $r>0$ we define $\CXomegaSpace{X}{r}[M]$ to be the space of those 
$f\in \CSpace{M}$ such that $X^{\alpha} f$ exists
and is continuous for all ordered multi-indices $\alpha$ and
such that
\begin{equation}\label{Eqn::FuncMfld::CXomegaNorm}
\CXomegaNorm{f}{X}{r}[M]
:=\sum_{m=0}^\infty \frac{r^m}{m!} \sum_{|\alpha|=m} \CNorm{X^{\alpha} f}{M}<\infty.
\end{equation}
We set
\begin{equation*}
	\CXjSpace{X}{\omega}[M]:=\bigcup_{r>0} \CXomegaSpace{X}{r}[M].
\end{equation*}
The norm $\CXomegaNorm{f}{X}{r}[M]$ was originally introduced by Nelson \cite{NelsonAnalyticVectors} in greater generality.

\begin{rmk}
When we write $Vf$ for a $C^1$ vector field $V$ and $f:M\rightarrow \R$, we define this as $Vf(x):=\frac{d}{dt}\big|_{t=0} f(e^{tX} x)$.
When we say $Vf$ exists, it mean that this derivative exists in the classical sense, $\forall x$.  If we have several $C^1$ vector fields $V_1,\ldots, V_L$,
we define $V_1V_2\cdots V_L f:= V_1(V_2(\cdots V_L (f)))$ and to say that this exists means that at each stage the derivatives exist.
\end{rmk}

Note that if $\grad$ denotes the list of vector fields $\grad=\left(\diff{x_1},\ldots, \diff{x_n}\right)$ and $\Omega\subseteq \R^n$ is open, we have
\begin{equation*}
\AXSpace{\grad}{0}{r}=\ASpace{n}{r}\text{ and } \CXomegaSpace{\grad}{r}[\Omega]=\ComegaSpace{r}[\Omega],
\end{equation*}
with equality of norms.\footnote{Notice that \cref{Eqn::FuncEuclid::ComegaNorm} uses regular multi-indicies, while \cref{Eqn::FuncMfld::CXomegaNorm} uses ordered
multi-indicies.  Once this is taken into account, proving the two norms are equal when $X=\grad$ is straightforward.}
For more details on these spaces, see \cref{Section::FuncSpaceRev}.

Throughout the paper if we claim $\CXomegaNorm{f}{X}{r}[M]<\infty$ it means $f\in \CXomegaSpace{X}{r}[M]$, and similarly for any other function space.
We refer the reader to \cite{StovallStreet} for a more detailed discussion of the above definitions.

An important property of the above spaces and norms is that they are invariant under diffeomorphisms.

\begin{prop}\label{Prop::FuncMfld::DiffeoInv}
Let $N$ be another $C^2$ manifold,  let $\Phi:M\rightarrow N$ be a $C^2$ diffeomorphism, and let $\Phi_{*}X$ denote the list
of vector fields $\Phi_{*}X_1,\ldots, \Phi_{*} X_q$.  Then the map $f\mapsto f\circ\Phi$ is an isometric isomorphism between
the following spaces:
$\CSpace{N}\rightarrow \CSpace{M}$,
$\CXomegaSpace{\Phi_{*}X}{r}[N]\rightarrow \CXomegaSpace{X}{r}[M]$,
and $\AXSpace{\Phi_{*}X}{\Phi(x_0)}{r}\rightarrow \AXSpace{X}{x_0}{r}$.
\end{prop}
\begin{proof}
This is immediate from the definitions.
\end{proof}

\begin{rmk}\label{Rmk::FuncMfld::DefineDeriv}
In \cref{Eqn::FuncManifold::DefnBX} (and in the rest of the paper), $\gamma'(t)$ is defined as follows.  In the case that $M$ is an open subset $\Omega\subseteq \R^n$
and $\gamma:[a,b]\rightarrow \Omega$, $\gamma'(t)=\sum_{j=1}^q a_j(t) X_j(\gamma(t))$ is defined to mean
$\gamma(t)= \gamma(a)+\int_a^t \sum_j a_j(s) X_j(\gamma(s))\: ds$; note that this definition is local in $t$ (equivalently, we are requiring that $\gamma$ be absolutely continuous
and have the desired derivative almost everywhere).  For an abstract $C^2$ manifold, this is interpreted locally.
I.e., if $\gamma:[a,b]\rightarrow M$, we say $\gamma'(t)= \sum_{j=1}^q a_j(t) X_j(\gamma(t))$ if $\forall t_0\in [a,b]$, there is an open neighborhood $N$ of $\gamma(t_0)$
and a $C^2$ diffeomorphism $\Psi:N\rightarrow \Omega$, where $\Omega\subseteq \R^n$ is open, such that $(\Psi\circ \gamma)'(t) = \sum_{j=1}^q a_j(t) (\Psi_{*}X_j)(\Psi\circ \gamma(t))$
for $t$ near $t_0$ ($t\in [a,b]$). 
\end{rmk}

\section{Results}\label{Section::Results}
We present the main results of the paper.  We separate the results into the qualitative results (i.e., \cref{Item::Intro::LocalQual,Item::Intro::GlobalQual} from the introduction)
and quantitative results (i.e., \cref{Item::Intro::Quant}).  The qualitative results are a simple consequence of the quantitative results, and the quantitative results are useful for proving results in analysis (see \cref{Section::Scaling}).
	
	\subsection{Qualitative Results}\label{Section::Results::Qual}
Let $X_1,\ldots, X_q$ be $C^1$ vector fields on a $C^2$ manifold $\fM$.  For $x,y\in \fM$, let $\rho(x,y)$ denote the sub-Riemannian distance
associated to $X_1,\ldots, X_q$ on $\fM$ defined by \cref{Eqn::FuncManifold::Defnrho}.  Fix $x_0\in M$ and let 
$Z:=\{y\in \fM : \rho(x_0,y)<\infty\}$.  $\rho$ is a metric on $Z$, and we give $Z$ the topology induced by $\rho$ (this is finer\footnote{See \cite[\SSFinerTopology]{StovallStreet} for a proof that this topology is finer than the subspace topology.}
than the topology as a subspace of $\fM$ and may be strictly finer).  Let $M\subseteq Z$ be a connected open subset of $Z$
containing $x_0$.  We give $M$ the topology of a subspace of $Z$.  We begin with a classical result to set the stage.

\begin{prop}\label{Prop::QualRes::InjectiveImmresion}
Suppose $[X_i,X_j]=\sum_{k=1}^q c_{i,j}^k X_k$, where $c_{i,j}^k:M\rightarrow \R$ are locally bounded.
Then, there is a $C^2$ manifold structure on 
$M$ (compatible with its topology) such that:
\begin{itemize}
\item The inclusion $M\hookrightarrow \fM$ is a $C^2$ injective immersion.
\item $X_1,\ldots, X_q$ are $C^1$ vector fields tangent to $M$.
\item $X_1,\ldots, X_q$ span the tangent space at every point of $M$.
\end{itemize}
Furthermore, this $C^2$ structure is unique in the sense that if $M$ is given another $C^2$ structure (compatible with its topology)
such that the inclusion map $M\hookrightarrow \fM$ is a $C^2$ injective immersion, then the identity map $M\rightarrow M$
is a $C^2$ diffeomorphism between these two structures.
\end{prop}

For a proof of \cref{Prop::QualRes::InjectiveImmresion} (which is standard), see \cite[\SSProofInjectiveImmersion]{StovallStreet}.  Henceforth, we assume the conditions of
\cref{Prop::QualRes::InjectiveImmresion} so that $M$ is a $C^2$ manifold and $X_1,\ldots, X_q$ are $C^1$ vector fields on $M$
which span the tangent space at every point.
We write $n:=\dim\Span\{X_1(x_0),\ldots, X_q(x_0)\}$ so that $\dim M=n$.

\begin{rmk}
If $X_1(x_0),\ldots, X_q(x_0)$ span $T_{x_0} \fM$, then $M$ is an open submanifold of $\fM$.  If $X_1,\ldots, X_q$
span the tangent space
at every point of $\fM$ and $\fM$ is connected, one may take $M=\fM$.
\end{rmk}

\begin{thm}[The Local Theorem]\label{Thm::QualRes::LocalThm}
The following three conditions are equivalent:
\begin{enumerate}[(i)]
%
%
\item\label{Item::QualRes::Local::Coord} There exists an open neighborhood $V\subseteq M$ of $x_0$ and a $C^2$ diffeomorphism $\Phi:U\rightarrow V$
where $U\subseteq \R^n$ is open, such that $\Phi^{*}X_1,\ldots, \Phi^{*}X_q\in \CjSpace{\omega}[U][\R^n]$.

\item\label{Item::QualRes::Local::Basis} Reorder the vector fields so that $X_1(x_0),\ldots, X_n(x_0)$ are linearly independent.
There exists an open neighborhood $V\subseteq M$ of $x_0$ such that:
\begin{itemize}
\item $[X_i,X_j]=\sum_{k=1}^n \ch_{i,j}^k X_k$, $1\leq i,j\leq n$, where $\ch_{i,j}^k\in \CXjSpace{X}{\omega}[V]$.
\item For $n+1\leq j\leq q$, $X_j=\sum_{k=1}^n b_j^k X_k$, where $b_j^k\in \CXjSpace{X}{\omega}[V]$.
\end{itemize}

\item\label{Item::QualRes::Local::Commute} There exists an open neighborhood $V\subseteq M$ of $x_0$ such that
$[X_i,X_j]=\sum_{k=1}^q c_{i,j}^k X_k$, $1\leq i,j\leq q$, where $c_{i,j}^k\in \CXjSpace{X}{\omega}[V]$.
\end{enumerate}

\end{thm}

\begin{thm}[The Global Theorem]\label{Thm::QualRes::GlobalThm}
The following two conditions are equivalent:
\begin{enumerate}[(i)]
\item\label{Item::QualRes::Global::Atlas} There is a real analytic atlas on $M$, compatible with its $C^2$ structure, such that $X_1,\ldots, X_q$ are real analytic with respect to this atlas.
\item\label{Item::QualRes::Global::Conds} For each $x_0\in M$, any of the three equivalent conditions from \cref{Thm::QualRes::LocalThm} hold for this choice of $x_0$.
\end{enumerate}
Furthermore, under these conditions, the real analytic manifold structure on $M$ induced by the atlas in \cref{Item::QualRes::Global::Atlas} is unique, in the sense that if there is another
real analytic atlas on $M$, compatible with its $C^2$ structure and such that $X_1,\ldots, X_q$ are real analytic with respect to this second atlas, then the identity map $M\rightarrow M$
is a real analytic diffeomorphism between these two real analytic structures on $M$.
\end{thm}
	
	\subsection{Quantitative Results}\label{Section::Results::Quant}
\Cref{Thm::QualRes::LocalThm} gives necessary and sufficient conditions for a certain type of coordinate chart to exist.
For applications in analysis, it is essential to have quantitative control of this coordinate chart.  By using this quantitative control,
these charts can be seen as generalized scaling maps in sub-Riemannian geometry--see \cref{Section::Scaling} and \cite[\SSScaling]{StovallStreet} for more details.  We now turn to these quantitative results, which are the heart of this paper.

Let $X_1,\ldots, X_q$ be $C^1$ vector fields on a $C^2$ manifold $\fM$.

\begin{defn}\label{Defn::QuantRes::sC}
For $x\in \fM$, $\eta>0$, and $U\subseteq\fM$, we say the list $X=X_1,\ldots, X_q$ satisfies $\sC(x_0,\eta,U)$ if for every $a\in B^q(\eta)$ the expression
\begin{equation*}
e^{a_1 X_1+\cdots + a_q X_q} x_0
\end{equation*}
exists in $U$.  More precisely, consider the differential equation
\begin{equation*}
\diff{r} E(r)=a_1 X_1(E(r))+\cdots+a_q X_q(E(r)), \quad E(0)=x_0.
\end{equation*}
We assume that a solution to this differential equation exists up to $r=1$, $E:[0,1]\rightarrow U$.  We have
$E(r) = e^{ra_1 X_1+\cdots+ ra_q X_q}x_0$.
\end{defn}

For $1\leq n\leq q$, we let
\begin{equation*}
\sI(n,q):=\{(i_1,i_2,\ldots, i_n) : i_j\in \{1,\ldots, q\}\}=\{1,\ldots, q\}^n.
\end{equation*}
For $J=(j_1,\ldots, j_n)\in \sI(n,q)$ we write $X_J$ for the list of vector fields $X_{j_1},\ldots, X_{j_n}$.  We write
$\bigwedge X_J:=X_{j_1}\wedge X_{j_2}\wedge \cdots \wedge X_{j_n}$.

Fix $x_0\in \fM$, $\xi>0$, $\zeta\in (0,1]$, and set $n=\dim\Span\{X_1(x_0),\ldots, X_q(x_0)\}$.  We assume for $1\leq j,k\leq q$,
\begin{equation*}
[X_j,X_k]=\sum_{l=1}^q c_{j,k}^l X_l, \quad c_{j,k}^l\in \CSpace{B_X(x_0,\xi)},
\end{equation*}
where $B_X(x_0,\xi)$ is defined via \cref{Eqn::FuncManifold::DefnBX} and is given the metric topology induced by $\rho$ from \cref{Eqn::FuncManifold::Defnrho}.
\Cref{Prop::QualRes::InjectiveImmresion} applies to show that $B_X(x_0,\xi)$ is an $n$-dimensional, $C^2$, injectively immersed
submanifold of $\fM$.  $X_1,\ldots, X_q$ are $C^1$ vector fields on $B_X(x_0,\xi)$ and span the tangent space at every point.
Henceforth, we treat $X_1,\ldots, X_q$ as vector fields on $B_X(x_0,\xi)$.

Let $J_0\in \sI(n,q)$ be such that $\bigwedge X_{J_0}(x_0)\ne 0$ and moreover
\begin{equation}\label{Eqn::QuantRes::DefnJ0}
\max_{J\in \sI(n,q)} \left|\frac{\bigwedge X_J(x_0)}{\bigwedge X_{J_0}(x_0)}\right|\leq \zeta^{-1},
\end{equation}
where $\frac{\bigwedge X_J(x_0)}{\bigwedge X_{J_0}(x_0)}$ is defined as follows.  Let $\lambda:\bigwedge^n T_{x_0} B_X(x_0,\xi)\rightarrow \R$ be any nonzero linear functional;
then
\begin{equation}\label{Eqn::QuantRes::QuotientWedge}
\frac{\bigwedge X_J(x_0)}{\bigwedge X_{J_0}(x_0)}:=\frac{\lambda(\bigwedge X_J(x_0))}{\lambda(\bigwedge X_{J_0}(x_0))}.
\end{equation}
Because $\bigwedge^n T_{x_0} B_X(x_0,\xi)$ is one dimensional, \cref{Eqn::QuantRes::QuotientWedge} is independent of the choice of $\lambda$;
see \cite[\SSDivideWedge]{StovallStreet} for more details.
Notice that a $J_0\in \sI(n,q)$ satisfying \cref{Eqn::QuantRes::DefnJ0} always exists--one can pick $J_0$ so that \cref{Eqn::QuantRes::DefnJ0} holds with $\zeta=1$; however
it is important for some applications\footnote{For example, it will be essential that we may take $\zeta<1$ in an upcoming work on similar questions with complex vector fields \cite{StreetNN}.} that we have the flexibility to choose $\zeta<1$.
Without loss of generality, reorder $X_1,\ldots, X_q$ so that $J_0=(1,\ldots, n)$.

\begin{itemize}
\item Let $\eta>0$ be such that $X_{J_0}$ satisfies $\sC(x_0,\eta,\fM)$.
\item Let $\delta_0>0$ be such that for $\delta\in (0,\delta_0]$ the following holds:  if $z\in B_{X_{J_0}}(x_0,\xi)$ is such that $X_{J_0}$
satisfies $\sC(z,\delta, B_{X_{J_0}}(x_0,\xi))$ and if $t\in B^n(\delta)$ is such that $e^{t_1 X_1+\cdots+t_n X_n}z=z$ and if $X_1(z),\ldots, X_n(z)$
are linearly independent, then $t=0$.
\end{itemize}

\begin{rmk}
Because $X_1,\ldots, X_n$ are $C^1$, such an $\eta>0$ and $\delta_0>0$ always exist; see \cref{Lemma::PfQual::Existsetadelta,Rmk::PfQual::Existsetadelta}.  However, in general one can
only guarantee that $\eta$, $\delta_0$ are bounded below in terms of the $C^1$ norms of $X_1,\ldots, X_n$ in some coordinate system--and this is not a diffeomorphic
invariant quantity.  Thus, we state our results in terms of $\delta_0$ and $\eta$ to preserve the diffeomorphic invariance.
See \cite[\SSSectionMoreOnAssumptions]{StovallStreet} for a further discussion on $\eta$ and $\delta_0$.
\end{rmk}

\noindent\textbf{Key Assumption:}  We assume $c_{j,k}^l\in \AXSpace{X_{J_0}}{x_0}{\eta}$, $1\leq j,k,l\leq q$.

\begin{rmk}\label{Rmk::QuantRest::CanUseComegaNormsInstead}
The assumption $c_{j,k}^l\in \AXSpace{X_{J_0}}{x_0}{\eta}$ can be replaced with the stronger assumption\footnote{A priori,
$B_{X_{J_0}}(x_0,\xi)$ is not necessarily a manifold.  Nevertheless, $\CXomegaSpace{X_{J_0}}{\eta}[B_{X_{J_0}}(x_0,\xi)]$ can be defined with the same formulas.  
For further details on this, see
\cite[\SSBeyondManifold]{StovallStreet}.} $c_{j,k}^l\in \CXomegaSpace{X_{J_0}}{\eta}[B_{X_{J_0}}(x_0,\xi)]$.
Indeed, \cref{Lemma::FuncSpaceRev::Mfld::ABigger} shows $\CXomegaSpace{X_{J_0}}{\eta}[B_{X_{J_0}}(x_0,\xi)]\subseteq \AXSpace{X_{J_0}}{x_0}{\min\{\xi,\eta\}}$.
\end{rmk}

\begin{defn}
We say $C$ is a $0$-admissible constant if $C$ can be chosen to depend only on upper bounds for $q$, $\zeta^{-1}$, $\xi^{-1}$,
and $\CNorm{c_{j,k}^l}{B_{X_{J_0}}(x_0,\xi)}$, $1\leq j,k,l\leq q$.
\end{defn}

\begin{defn}\label{Defn::QuantRes::AdmissibleConst}
We say $C$ is an admissible constant if $C$ can be chosen to depend on anything a $0$-admissible constant can depend on,
and can also depend on upper bounds for $\eta^{-1}$, $\delta_0^{-1}$, 
and $\AXNorm{c_{j,k}^l}{X_{J_0}}{x_0}{\eta}$ ($1\leq j,k,l\leq q$).
\end{defn}

We write $A\lesssim_0 B$ for $A\leq C B$ where $C$ is a positive $0$-admissible constant, and write $A\approx_0 B$ for $A\lesssim_0 B$ and $B\lesssim_0 A$.
We write $A\lesssim B$ for $A\leq C B$ where $C$ is a positive admissible constant, and write $A\approx B$ for $A\lesssim B$ and $B\lesssim A$.

For $t\in B^n(\eta)$ set
\begin{equation}\label{Eqn::QuantRes::DefnPhi}
\Phi(t) = e^{t_1 X_1+\cdots +t_n X_n} x_0.
\end{equation}
Let $\eta_0:=\min\{\eta,\xi\}$ so that $\Phi:B^n(\eta_0)\rightarrow B_{X_{J_0}}(x_0,\xi)\subseteq B_X(x_0,\xi)$.

\begin{thm}[The Quantitative Theorem]\label{Thm::QuantRes::MainThm}
Fix $x_0\in \fM$ and let $\xi$, $\zeta$, $n$, $J_0$, $\eta$, and $\delta_0$ be as above, and suppose the Key Assumption is satisfied. 
Then, there exists a $0$-admissible constant $\chi\in (0,\xi]$ such that:
\begin{enumerate}[label=(\alph*),series=maintheoremenumeration]
\item\label{Item::QuantRes::LI} $\forall y\in B_{X_{J_0}}(x_0,\chi)$, $\bigwedge X_{J_0}(y)\ne 0$.
\item\label{Item::QuantRes::BestBasis} $\forall y\in B_{X_{J_0}}(x_0,\chi)$,
\begin{equation*}
\sup_{J\in \sI(n,q)} \left| \frac{\bigwedge X_J(y)}{\bigwedge X_{J_0}(y)} \right|\approx_0 1.
\end{equation*}
\item\label{Item::QuantRes::chiSubMfld} $\forall \chi'\in (0,\chi]$, $B_{X_{J_0}}(x_0,\chi')$ is an open subset of $B_X(x_0,\xi)$ and is therefore a submanifold.
\end{enumerate}
There exist admissible constants $\eta_1,\xi_1,\xi_2>0$ such that:
\begin{enumerate}[resume*=maintheoremenumeration]
\item\label{Item::QuantRes::PhiOpen} $\Phi(B^n(\eta_1))$ is an open subset of $B_{X_{J_0}}(x_0,\chi)$ and is therefore a submanifold of $B_X(x_0,\xi)$.
\item\label{Item::QuantRes::PhiDiffeo} $\Phi:B^n(\eta_1)\rightarrow \Phi(B^n(\eta_1))$ is a $C^2$ diffeomorphism.
\item\label{Item::QuantRes::xi1xi2} $B_X(x_0,\xi_2)\subseteq B_{X_{J_0}}(x_0,\xi_1)\subseteq \Phi(B^n(\eta_1))\subseteq B_{X_{J_0}}(x_0,\chi)\subseteq B_X(x_0,\xi)$.
\end{enumerate}
Let $Y_j =\Phi^{*} X_j$ and write $Y_{J_0}=(I+A) \grad$, where $Y_{J_0}$ denotes the column vector of vector fields $Y_{J_0}=[Y_1,Y_2,\ldots, Y_n]^{\transpose}$,
$\grad$ denotes the gradient in $\R^n$ thought of as a column vector, and $A\in \CSpace{B^n(\eta_1)}[\M^{n\times n}]$.\footnote{Here, and in the rest of the paper,
$\M^{n\times n}$ denotes the Banach space of $n\times n$ real matrices endowed with the usual operator norm.}
\begin{enumerate}[resume*=maintheoremenumeration]
\item\label{Item::QuantRes::BoundA} $A(0)=0$ and $A\in \ASpace{n}{\eta_1}[\M^{n\times n}]$ with $\ANorm{A}{n}{\eta_1}[\M^{n\times n}]\leq \frac{1}{2}$.
\item\label{Item::QuantRes::BoundY} For $1\leq j\leq q$, $Y_j \in \ASpace{n}{\eta_1}[\R^n]$ and $\ANorm{Y_j}{n}{\eta_1}[\R^n]\lesssim 1$.
\end{enumerate}
\end{thm}

\begin{rmk}
The main results of this paper (including \cref{Thm::QuantRes::MainThm}) are invariant under arbitrary $C^2$ diffeomorphisms.  
This is true \textit{quantitatively}--all of the estimates are unchanged when pushed forward under an arbitrary $C^2$ diffeomorphism; this is a consequence of \cref{Prop::FuncMfld::DiffeoInv}.
See \cite{StovallStreet} for more details.
\end{rmk}

		\subsubsection{Densities}\label{Section::Results::Desnities}\label{Section::Results::Densitites}
As in \cite{StovallStreet,StreetII}, we describe how to study densities in the coordinate system given by \cref{Thm::QuantRes::MainThm}.
We refer the reader to \cref{Section::Scaling} and \cite[\SSScaling]{StovallStreet} for a further discussion of how these estimates can be used.

We take the same setting as in \cref{Thm::QuantRes::MainThm}.
Let $\chi\in (0,\xi]$ be as in that theorem and let $\nu$ be a $C^1$ density on $B_{X_{J_0}}(x_0,\chi)$.  Suppose
\begin{equation*}
\Lie{X_j} \nu =f_j \nu, \quad 1\leq j\leq n, \quad f_j\in \CSpace{B_{X_{J_0}}(x_0,\chi)},
\end{equation*}
where $\Lie{X_j}$ denotes the Lie derivative with respect to $X_j$.  We refer the reader to \cite{GuilleminNotes} for a quick and easy to read introduction on the basics
of densities (see also \cite{NicolaescuLecturesOnTheGeometryOfManifolds} where densities are called $1$-densities).
We assume that there exists $r>0$ such that $f_j\in \AXSpace{X_{J_0}}{x_0}{r}$.

\begin{defn}
We say $C$ is a $0;\nu$-admissible constant if $C$ is a $0$-admissible constant which is also allowed to depend on upper bounds for 
$\CNorm{f_j}{B_{X_{J_0}}(x_0,\chi)}$, $1\leq j\leq n$.
\end{defn}

\begin{defn}
We say $C$ is a  $\nu$-admissible constant, if $C$ is an admissible constant which is also allowed to depend on upper bounds for $r^{-1}$ and
$\AXNorm{f_j}{X_{J_0}}{x_0}{r}$, $1\leq j\leq n$.
\end{defn}

We write $A\lesssim_{0;\nu} B$ for $A\leq CB$ where $C$ is a $0;\nu$-admissible constant, and write $A\approx_{0;\nu}B$ for
$A\lesssim_{0;\nu} B$ and $B\lesssim_{0;\nu} A$.  We similarly define $\lesssim_\nu$ and $\approx_\nu$.

\begin{thm}\label{Thm::Density::MainThm}
Define $h\in \CjSpace{1}[B^n(\eta_1)]$ by $\Phi^{*}\nu=h\LebDensity$, where $\LebDensity$ denotes the Lebesgue density on $\R^n$.
Then,
\begin{enumerate}[(a)]
\item\label{Item::Density::HConst} $h(t)\approx_{0;\nu} \nu(X_1,\ldots, X_n)(x_0)$, $\forall t\in B^n(\eta_1)$.  In particular, $h(t)$ always has the same sign, and is either never zero, or always zero.
\item\label{Item::Density::HRA} Set $s:=\min\{\eta_1,r\}$, where $\eta_1$ is as in  \cref{Thm::QuantRes::MainThm}.  Then, $h\in \ASpace{n}{s}$ and $\ANorm{h}{n}{s}\lesssim_\nu |\nu(X_1,\ldots, X_n)(x_0)|$.
\end{enumerate}
\end{thm}

\begin{cor}\label{Cor::Density::MainCor}
Let $\xi_2$ be as in \cref{Thm::QuantRes::MainThm}.  Then,
\begin{equation}\label{Eqn::Density::MainCor::NoAbs}
\nu(B_{X_{J_0}}(x_0,\xi_2)) \approx_{\nu} \nu(B_X(x_0,\xi_2))\approx_\nu \nu(X_1,\ldots, X_n)(x_0),
\end{equation}
and therefore,
\begin{equation}\label{Eqn::Density::MainCor::Abs}
|\nu(B_{X_{J_0}}(x_0,\xi_2))| \approx_\nu |\nu(B_X(x_0,\xi_2))| \approx_\nu |\nu(X_1,\ldots, X_n)(x_0)|\approx_0 \max_{j_1,\ldots, j_n\in \{1,\ldots, q\}} |\nu(X_{j_1},\ldots, X_{j_n})(x_0)|.
\end{equation}
\end{cor}
		
\section{Function Spaces Revisited}\label{Section::FuncSpaceRev}
In this section, we present the basic results we need concerning the function spaces defined in \cref{Section::FuncSpace}.
Let $M$ be a $C^2$ manifold and let $X_1,\ldots, X_q$ be $C^1$ vector fields on $M$
and let $X$ denote the list $X=X_1,\ldots, X_q$.  Let $\Omega\subseteq \R^n$ be an open set.

\begin{lemma}\label{Lemma::FuncSpaceRev::Algebra}
The spaces $\ComegaSpace{r}[\Omega]$, $\ASpace{n}{r}$, $\CXomegaSpace{X}{r}[M]$, and $\AXSpace{X}{x_0}{r}$ are Banach algebras.
In particular, if $\BanachAlgebra$ denotes any one of these spaces and if $x,y\in \BanachAlgebra$, then $\Norm{xy}[\BanachAlgebra]\leq \Norm{x}[\BanachAlgebra]\Norm{y}[\BanachAlgebra]$.
More generally, the same holds for the analogous spaces of functions taking values in a Banach algebra.
\end{lemma}
\begin{proof}
We prove only the result for functions taking values in $\C$; the same proof proves the more general result for functions taking values in a Banach algebra.

We begin with proof for $\ASpace{n}{r}$.  Suppose $f,g\in \ASpace{n}{r}$.  Then if $f(t)=\sum_{\alpha\in \N^n} \frac{c_\alpha}{\alpha!} t^{\alpha}$
and $g(t)=\sum_{\alpha\in \N^n} \frac{d_\alpha}{\alpha!} t^{\alpha}$,
we have
$f(t)g(t)= \sum_{\alpha,\beta\in \N^n} \frac{c_\alpha d_\beta}{\alpha! \beta!} t^{\alpha+\beta}$, and therefore,
\begin{equation*}
\ANorm{fg}{n}{r} \leq \sum_{\alpha,\beta\in \N^n} \frac{|c_\alpha d_\beta|}{\alpha!\beta!} r^{|\alpha|+|\beta|} \leq \left(\sum_{\alpha\in \N^n} \frac{|c_\alpha|}{\alpha!} r^{|\alpha|}\right)\left(\sum_{\beta\in \N^n} \frac{|d_\beta|}{\beta!} r^{|\beta|}\right)=\ANorm{f}{n}{r}\ANorm{g}{n}{r},
\end{equation*}
completing the proof for $\ASpace{n}{r}$.  The result for $\AXSpace{X}{x_0}{r}$ follows immediately from the result for $\ASpace{q}{r}$.

Next we consider $\CXomegaSpace{X}{r}[M]$.  For this, we need some notation from \cite{NelsonAnalyticVectors}--we refer the reader
to that reference for more detailed information on these definitions.
Let $\BanachAlgebra$ be a Banach space and let $\OpsY$ denote the set of all (bounded or unbounded) operators on $\BanachAlgebra$.
For $A\in \OpsY$, we write $|A|$ for the set consisting of $A$ alone.
Let $|\OpsY|$ be the free abelian semigroup with the set of all $|A|$, $A\in \OpsY$, as generators.  Let
$\alpha,\beta\in |\OpsY|$ so that $\alpha=|A_1|+\cdots+|A_l|$, $\beta=|B_1|+\cdots+|B_m|$ (where these are formal sums).
We define $\alpha\beta=\sum_{i=1}^l \sum_{j=1}^m |A_iB_j|\in |\OpsY|$.  

Let $A_1,\ldots, A_l\in \OpsY$, and set $\alpha=|A_1|+\cdots+|A_l|$.
For $y\in \BanachAlgebra$ in the domains of $A_1,\ldots, A_l$, we define
$\Norm{\alpha y}:=\Norm{A_1 y}[\BanachAlgebra]+\cdots+\Norm{A_l y}[\BanachAlgebra]$.
For $f=(f_1,\ldots, f_m)\in \BanachAlgebra^m$, with each $f_j$ in the domains of $A_1,\ldots, A_l$, define
$\sA f:= (A_j f_k)_{1\leq j\leq l, 1\leq k\leq m}\in \BanachAlgebra^{ml}$.  Note that
$\Norm{\alpha x} = \Norm{\sA x}[\BanachAlgebra^l]$, and more generally, $\Norm{\alpha^m x} =\Norm{\sA^m x}[\BanachAlgebra^{l^m}]$.
Here we are giving $\BanachAlgebra^m$ the norm $\Norm{f}[\BanachAlgebra^m]:=\sum_{l=1}^m \Norm{f_j}[\BanachAlgebra]$.

Now suppose $\BanachAlgebra$ is a Banach algebra, and suppose $A_1,\ldots, A_l\in \OpsY$ satisfy
$A_j (xy) = (A_j(x)) y + x(A_j(y))$ (and the domains of $A_1,\ldots, A_l$ are algebras).  For $f=(f_1,\ldots, f_{m_1})\in \BanachAlgebra^{m_1}$ and $g=(g_1,\ldots, g_{m_2})\in \BanachAlgebra^{m_2}$
set $fg=(f_j g_k)_{1\leq j\leq m_1, 1\leq k\leq m_2}\in \BanachAlgebra^{m_1m_2}$.
We then have, for $x,y\in \BanachAlgebra$ (and $x$ and $y$ in the domains of the appropriate operators),
\begin{equation*}
\sA^k (xy) = \sum_{k_1+k_2=k} \binom{k}{k_1} (\sA^{k_1} x) (\sA^{k_2} y),
\end{equation*}
and in particular
\begin{equation}\label{Eqn::FuncSpaceRev::ProductAbstract}
\Norm{\alpha^k (xy)} = \Norm{\sA^k (xy)}[\BanachAlgebra^{l^k}] \leq \sum_{k_1+k_2=k} \binom{k}{k_1}\Norm{\sA^{k_1} x}[\BanachAlgebra^{l^{k_1}}] \Norm{\sA^{k_2} y}[\BanachAlgebra^{l^{k_2}}]
=\sum_{k_1+k_2=k} \binom{k}{k_1} \Norm{\alpha^{k_1} x} \Norm{\alpha^{k_2} y}.
\end{equation}

We now specialize to the case $\BanachAlgebra=\CSpace{M}$ and $\alpha=|X_1|+|X_2|+\cdots+|X_q|$.  We have, using \cref{Eqn::FuncSpaceRev::ProductAbstract} and
the definition of $\CXomegaNorm{\cdot}{X}{r}$,
\begin{equation*}
\begin{split}
&\CXomegaNorm{fg}{X}{r}[M] = \sum_{k=0}^\infty \frac{1}{k!} \Norm{\alpha^k (fg)} r^k
\leq \sum_{k=0}^\infty \frac{1}{k!} \sum_{k_1+k_2=k} \binom{k}{k_1} \Norm{\alpha^{k_1} f} \Norm{\alpha^{k_2} g}r^{k_1} r^{k_2}
\\&=\left(\sum_{k_1=0}^{\infty} \frac{1}{k_1!} \Norm{\alpha^{k_1}f} r^{k_1}\right) \left(\sum_{k_2=0}^\infty \frac{1}{k_2!} \Norm{\alpha^{k_2} g}r^{k_2}\right)
=\CXomegaNorm{f}{X}{r}[M] \CXomegaNorm{g}{X}{r}[M].
\end{split}
\end{equation*}
This completes the proof for $\CXomegaSpace{X}{r}[M]$.  Since $\ComegaSpace{r}[\Omega]=\CXomegaSpace{\grad}{r}[\Omega]$ (with equality of norms), the result
for $\ComegaSpace{r}[\Omega]$ follows as well.
\end{proof}

The spaces $\CXomegaSpace{X}{r}[M]$, and $\AXSpace{X}{x_0}{r}$ are closely related as the next three results show.

\begin{lemma}\label{Lemma::FiuncSpaceRev::Euclid::Compare}
\begin{enumerate}[(i)]
\item\label{Item::FuncSpaceRev::Euclid::ABigger} $\ComegaSpace{r}[B^n(r)]\subseteq \ASpace{n}{r}$ and $\ANorm{f}{n}{r}\leq \ComegaNorm{f}{r}[B^n(r)]$.
\item\label{Item::FuncSpaceRev::Euclid::CBigger} $\ASpace{n}{r}\subseteq \ComegaSpace{r/2}[B^n(r/2)]$ and $\ComegaNorm{f}{r/2}[B^n(r/2)]\leq \ANorm{f}{n}{r}$.
\end{enumerate}
\end{lemma}
\begin{proof}
\Cref{Item::FuncSpaceRev::Euclid::ABigger} is a special case of \cref{Lemma::FuncSpaceRev::Mfld::ABigger}, below, so we only prove \cref{Item::FuncSpaceRev::Euclid::CBigger}.
We use the identity, for multi-indices $\alpha\in \N^n$,
\begin{equation*}
\sum_{\beta\leq \alpha} \binom{\alpha}{\beta} = 2^{|\alpha|},
\end{equation*}
where the sum is taken over all $\beta\in \N^n$ with $\beta_j\leq \alpha_j$ for all $j$.

Suppose $f\in \ASpace{n}{r}$.  Then, $f(t)=\sum_{\alpha\in \N^n}\frac{c_\alpha}{\alpha!} t^{\alpha}$ with $\ANorm{f}{n}{r}=\sum_{\alpha\in \N^n} \frac{|c_\alpha|}{\alpha!} r^{|\alpha|}<\infty$.
Set $r_1=r/2$.  We have
\begin{equation*}
\begin{split}
&\ComegaNorm{f}{r_1}[B^n(r_1)] = \sum_{\beta\in \N^n} \frac{r_1^{|\beta|}}{\beta!} \CNorm{\partial_t^{\beta} f}{B^n(r_1)}
= \sum_{\beta\in \N^n} \frac{r_1^{|\beta|}}{\beta!} \sup_{t\in B^n(r_1)} \left| \sum_{\alpha\geq \beta} \frac{c_\alpha}{(\alpha-\beta)!} t^{\alpha-\beta}\right|
\\&\leq \sum_{\beta\in \N^n} \sum_{\alpha\geq \beta} \frac{r_1^{|\alpha|}}{\beta!(\alpha-\beta)!} |c_\alpha|
=\sum_{\alpha\in \N^n} \frac{r_1^{|\alpha|}}{\alpha!} |c_\alpha| \sum_{\beta\leq \alpha} \binom{\alpha}{\beta}
=\sum_{\alpha\in \N^n} \frac{2^{|\alpha|} r_1^{|\alpha|}}{\alpha!} |c_\alpha| =\ANorm{f}{n}{r},
\end{split}
\end{equation*}
completing the proof.
\end{proof}

\begin{lemma}\label{Lemma::FuncSpaceRev::Mfld::ABigger}
Suppose $X=X_1,\ldots, X_q$ satisfies $\sC(x_0,r,M)$.  Then, $\CXomegaSpace{X}{r}[M]\subseteq \AXSpace{X}{x_0}{r}$ and
$\AXNorm{f}{X}{x_0}{r}\leq \CXomegaNorm{f}{X}{r}[M]$.
\end{lemma}
\begin{proof}
We will show, for $f\in \CXomegaSpace{X}{r}[M]$ that
\begin{equation}\label{Eqn::FuncSpaceRev::ToShowEqualh}
f(e^{t_1 X_1+\cdots + t_q X_q}x_0)= \sum_{m=0}^\infty \frac{((t_1X_1+\cdots +t_qX_q )^m f)(x_0)}{m!}=:h(t), \quad t\in B^q(r).
\end{equation}
The result will follow since the hypothesis $f\in \CXomegaSpace{X}{r}[M]$ implies that the sum in \cref{Eqn::FuncSpaceRev::ToShowEqualh} converges absolutely
for $|t|\leq r$, and $\AXNorm{f}{X}{x_0}{r}=\ANorm{h}{q}{r}\leq \CXomegaNorm{f}{X}{r}[M]$.

Fix $t\in B^q(r)$. 
For $\delta>0$ small (depending on $t$) and for $s_1,s_2\in (-1-\delta,1+\delta)$ define
\begin{equation}\label{Eqn::FuncSpaceRev::Makeg}
g(s_1,s_2)=\sum_{m=0}^\infty \frac{s_1^m}{m!} \left( (t_1X_1+\cdots+t_qX_q)^m f\right)\left( e^{s_2 (t_1X_1+\cdots+t_qX_q)}x_0 \right).
\end{equation}
Since $X$ satisfies $\sC(x_0,r,M)$, for $s_2\in (-1-\delta,1+\delta)$ we have $e^{s_2(t_1X_1+\cdots+t_qX_q)}x_0\in M$, and therefore the sum
in \cref{Eqn::FuncSpaceRev::Makeg} converges absolutely by the hypothesis that  $f\in \CXomegaSpace{X}{r}[M]$ (here we are taking $\delta$ small, depending on $t$).  Hence
$g(s_1,s_2)$ is defined for $s_1,s_2\in (-1-\delta,1+\delta)$.

We have,
\begin{equation*}
\diff{s_1} g(s_1,s_2)= \sum_{m=0}^\infty \frac{s_1^m}{m!} \left( (t_1X_1+\cdots+t_qX_q)^{m+1} f \right)\left( e^{s_2 (t_1X_1+\cdots+t_qX_q)}x_0 \right)
= \diff{s_2} g(s_1,s_2).
\end{equation*}
We conclude $g(s,0)=g(0,s)$ for $s\in (-1-\delta,1+\delta)$.  In particular, $g(1,0)=g(0,1)$, which is exactly
\cref{Eqn::FuncSpaceRev::ToShowEqualh}, completing the proof.
\end{proof}

Unlike the Euclidean case in \cref{Lemma::FiuncSpaceRev::Euclid::Compare}, the reverse containment to \cref{Lemma::FuncSpaceRev::Mfld::ABigger}
is a more involved and requires more hypotheses.  In fact, we see it as a corollary of \cref{Thm::QuantRes::MainThm}.

\begin{cor}\label{Cor::FuncSpaceRev::HardContainment}
We take all the same hypotheses and notation as in \cref{Thm::QuantRes::MainThm} and define admissible constants as in that theorem.
Fix $r\in (0,\eta_1]$ (where $\eta_1$ is as in \cref{Thm::QuantRes::MainThm}).
Then, there is an admissible constant $s=s(r)>0$ such that $\AXSpace{X_{J_0}}{x_0}{r}\subseteq \CXomegaSpace{X}{s}[\Phi(B^n(r/2)]$.
Moreover, there is an admissible constant $C=C(r)$ such that
\begin{equation}\label{Eqn::FuncSpaceRev::HardIneq}
\CXomegaNorm{f}{X}{s}[\Phi(B^n(r/2))]\leq C \AXNorm{f}{X_{J_0}}{x_0}{r}.
\end{equation}
\end{cor}

\begin{rmk}
Notice that $\AXNorm{f}{X_{J_0}}{x_0}{r}\leq \AXNorm{f}{X}{x_0}{r}$ and so one may replace $X_{J_0}$ with $X$
throughout \cref{Cor::FuncSpaceRev::HardContainment}.
\end{rmk}

\begin{proof}[Proof of \cref{Cor::FuncSpaceRev::HardContainment}]
Let $\Phi(t)=e^{t_1X_1+\cdots+t_nX_n}x_0$ be as in \cref{Thm::QuantRes::MainThm}.
Suppose $f\in \AXSpace{X_{J_0}}{x_0}{r}$; so that, by the definition of $\AXSpace{X_{J_0}}{x_0}{r}$, $\Phi^{*}f\in \ASpace{n}{r}$
with $\ANorm{\Phi^{*} f}{n}{r}= \AXNorm{f}{X_{J_0}}{x_0}{r}$.
By \cref{Lemma::FiuncSpaceRev::Euclid::Compare} \cref{Item::FuncSpaceRev::Euclid::CBigger},
$\Phi^{*} f\in \ComegaSpace{r/2}[B^n(r/2)]$ with $\ComegaNorm{\Phi^{*} f}{r/2}[B^n(r/2)]\leq \AXNorm{f}{X_{J_0}}{x_0}{r}$.  
Letting $Y_j=\Phi^{*}X_j$ as in \cref{Thm::QuantRes::MainThm},
we have that $Y_1,\ldots, Y_q\in \ASpace{n}{\eta_1}[\R^n]\subseteq \ComegaSpace{\eta_1/2}[B^n(\eta_1/2)][\R^n]$ 
with $\ComegaNorm{Y_j}{\eta_1/2}[B^n(\eta_1/2)][\R^n] \leq \ANorm{Y_j}{n}{\eta_1}[\R^n]\lesssim 1$
(where we have again used \cref{Lemma::FiuncSpaceRev::Euclid::Compare} \cref{Item::FuncSpaceRev::Euclid::CBigger}).  

\Cref{Prop::NelsonTheorem2} (below)
shows that
there exists an admissible $s=s(r)>0$ such that $\Phi^{*} f\in \ComegaSpace{r/2}[B^n(r/2)]= \CXomegaSpace{Y}{s}[B^n(r/2)]$ and
\begin{equation*}
\CXomegaNorm{\Phi^{*} f}{Y}{s}[B^n(r/2)]\leq C \ComegaNorm{\Phi^{*} f}{r/2}[B^n(r/2)]\leq C \AXNorm{f}{X_{J_0}}{x_0}{r},
\end{equation*}
where $C$ is as in the statement of the corollary.
Because $\Phi_{*}Y_j=X_j$, 
\cref{Prop::FuncMfld::DiffeoInv} implies $f\in \CXomegaSpace{X}{s}[\Phi(B^n(r/2)]$ with $\CXomegaNorm{f}{X}{s}[\Phi(B^n(r/2))]=\CXomegaNorm{\Phi^{*} f}{Y}{s}[B^n(r/2)]\leq C\AXNorm{f}{X_{J_0}}{x_0}{r}$ , completing the proof.
\end{proof}

\begin{lemma}\label{Lemma::FuncSpaceRev::DerivComega}
For any $s\in (0,r)$,
$X_j:\CXomegaSpace{X}{r}[M]\rightarrow \CXomegaSpace{X}{s}[M]$.
Furthermore, for $f\in \CXomegaSpace{X}{r}[M]$, $\CXomegaNorm{X_j f}{X}{s}[M] \leq \mleft(\sup_{m\in \N} (s/r)^m \mleft(\frac{m+1}{r}\mright)  \mright) \CXomegaNorm{f}{X}{r}[M]$.
\end{lemma}
\begin{proof}
Let $f\in \CXomegaSpace{X}{r}[M]$, and consider 
\begin{equation*}
\begin{split}
&\CXomegaNorm{X_j f}{X}{s}[M] = \sum_{m=0}^\infty \frac{s^m}{m!} \sum_{|\alpha|=m} \CNorm{X^{\alpha} X_j f}{M}
\leq \sum_{m=0}^\infty (s/r)^m \frac{m+1}{r} \frac{r^{m+1}}{(m+1)!} \sum_{|\alpha|=m+1} \CNorm{X^{\alpha} f}{M}
\\&\leq \mleft(\sup_{m\in \N} (s/r)^m \frac{m+1}{r}  \mright) \CXomegaNorm{f}{X}{r}[M]<\infty,
\end{split}
\end{equation*}
where in the last inequality, we have used $s<r$.  The result follows.
\end{proof}

	\subsection{Comparison with Euclidean function spaces and a result of Nelson}
Let $\Omega\subseteq \R^n$ be an open set.  If $Y_1,\ldots, Y_q$ are real analytic vector fields on $\Omega$ which span the tangent space at every point,
it is a result of Nelson \cite[Theorem 2]{NelsonAnalyticVectors} that being real analytic with respect to $Y_1,\ldots, Y_q$ is the same as being real analytic in the classical sense.
We state a quantitative version of this.



\begin{prop}\label{Prop::NelsonTheorem2}
Fix $r>0$, and let $Y_1,\ldots, Y_q\in \ComegaSpace{r}[\Omega][\R^n]$.
\begin{enumerate}[(i)]
\item\label{Item::NelsonThm::Used} There exists $s>0$ such that 
$\ComegaSpace{r}[\Omega]\subseteq \CXomegaSpace{Y}{s}[\Omega]$ and
$\CXomegaNorm{f}{Y}{s}[\Omega]\leq C \ComegaNorm{f}{r}[\Omega]$, $\forall f\in \ComegaSpace{r}[\Omega]$,
where $s$ and $C$ can be chosen to depend only on upper bounds for $q$, $n$, $r^{-1}$, and $\ComegaNorm{Y_j}{r}[\Omega][\R^n]$ ($1\leq j\leq q$).

\item Suppose, in addition,  that for $1\leq j\leq n$, $1\leq k\leq q$, there exists $b_j^k\in \ComegaSpace{r}[\Omega]$
such that $\diff{x_j}=\sum_{k=1}^q b_j^k Y_k$.
For all $r_1>0$, there exists $s'>0$ such that 
$\CXomegaSpace{Y}{r_1}[\Omega]\subseteq \ComegaSpace{s'}[\Omega]$ and
$\ComegaNorm{f}{s'}[\Omega]\leq C \CXomegaNorm{f}{Y}{r_1}[\Omega]$, $\forall f\in  \CXomegaSpace{Y}{r_1}[\Omega]$,
where $s'$ and $C$ can be chosen to depend only on upper bounds for $q$, $r^{-1}$, $r_1^{-1}$, $\ComegaNorm{Y_j}{r}[\Omega][\R^n]$ ($1\leq j\leq q$),
and $\ComegaNorm{b_j^k}{r}[\Omega]$ ($1\leq j\leq n$, $1\leq k\leq q$).
\end{enumerate}
\end{prop}
\begin{proof}
This follows from a straightforward modification of the proof of  \cite[Theorem 2]{NelsonAnalyticVectors} and we leave the details to the reader.
\end{proof}

\begin{rmk}
In the sequel, we only use \cref{Item::NelsonThm::Used} of \cref{Prop::NelsonTheorem2}.
\end{rmk}
	
\section{Sub-Riemannian Geometry and Scaling}\label{Section::Scaling}
One of the main applications of results like \cref{Thm::QuantRes::MainThm} is as scaling maps for sub-Riemannian geometries.  Such scaling maps were first introduced by
Nagel, Stein, and Wainger \cite{NagelSteinWaingerBallsAndMetrics}, and were further studied by many other authors including Tao and Wright \cite{TaoWrightLpImproving}, the author \cite{S}, Montanari and Morbidelli \cite{MontanariMorbidelliNonsmoothHormanderVectorFields},
and most recently in the first two parts of this series \cite{StovallStreet,StreetII}.  Since Nagel, Stein, and Wainger's results, these ideas have been used in a wide variety of problems.
For a description of some of these applications, see the notes at the end of Chapter 2 of \cite{StreetMultiParamSingInt}.

Nagel, Stein, and Wainger's results worked in the smooth category\footnote{More precisely, the quantitative estimates they proved involved $C^m$ norms, for various $m\in \N$.}.  The later papers either worked in the smooth category,
or with a finite level of smoothness.  Thus, if one starts with a sub-Riemannian geometry based on real analytic vector fields, the results in these works do not yield appropriate quantitative control in the real analytic setting
and therefore these results destroy the real analytic nature of the problem under consideration.
\Cref{Thm::QuantRes::MainThm} fixes this issue.

Furthermore, when dealing with real analytic vector fields, one does not need to assume H\"ormander's condition on the vector fields.  See \cref{Section::Scaling::BeyondHormander}.

We present such results, here.

	\subsection{Classical Real Analytic Sub-Riemannian Geometries}\label{Section::NSW}
In this section, we describe a real analytic version of the foundational work of Nagel, Stein, and Wainger \cite{NagelSteinWaingerBallsAndMetrics}, and see how it is a special case
of \cref{Thm::QuantRes::MainThm}.  This is the simplest non-trivial setting where the results in this paper can be seen as providing scaling maps adapted to a sub-Riemannian geometry.

Let $X_1,\ldots, X_q$ be real analytic vector fields on an open set $\Omega\subseteq \R^n$; we assume $X_1,\ldots, X_q$ span the tangent space at every point of $\Omega$.  To each $X_j$ assign a formal degree $d_j\in [1,\infty)$.
We assume $\forall x\in \Omega$ there exists an open neighborhood $U_x\subseteq \Omega$ of $x$ such that:
\begin{equation}\label{Eqn::NSW::MainAssump}
[X_j,X_k] =\sum_{d_l\leq d_j+d_k} c_{j,k}^{l,x} X_l, \quad c_{j,k}^{l,x}\in \CjSpace{\omega}[U_x].
\end{equation}

We write $(X,d)$ for the list $(X_1,d_1),\ldots, (X_q,d_q)$ and for $\delta>0$ write $\delta^d X$ for the list of vector fields $\delta^{d_1} X_1,\ldots, \delta^{d_q} X_q$.  The sub-Riemannian ball
associated to $(X,d)$ centered at $x_0\in \Omega$, of radius $\delta>0$ is defined by
\begin{equation*}
B_{(X,d)}(x_0,\delta):= B_{\delta^d X}(x_0,1),
\end{equation*}
where the later ball is defined by \cref{Eqn::FuncManifold::DefnBX}.  $B_{(X,d)}(x_0,\delta)$ is an open subset of $\Omega$.  It is easy to see that the balls $B_{(X,d)}(x,\delta)$ are metric balls.

Define, for $x\in \Omega$, $\delta\in (0,1]$,
\begin{equation*}
	\Lambda(x,\delta):=\max_{j_1,\ldots, j_n\in \{1,\ldots, q\}} \mleft|\det \mleft(\delta^{d_{k_1}} X_{k_1}(x)| \cdots | \delta^{d_{k_n}} X_{k_n}(x) \mright) \mright|.
\end{equation*}
For each $x\in \Omega$, $\delta\in (0,1]$, pick $j_1=j_1(x,\delta),\ldots, j_n=j_n(x,\delta)$ so that
\begin{equation*}
	\mleft|\det \mleft(\delta^{d_{j_1}} X_{j_1}(x)| \cdots | \delta^{d_{j_n}} X_{j_n}(x) \mright) \mright| =\Lambda(x,\delta).
\end{equation*}
For this choice of $j_1=j_1(x,\delta),\ldots, j_n=j_n(x,\delta)$, define
\begin{equation*}
\Phi_{x,\delta}(t_1,\ldots, t_n):=\exp(t_1 \delta^{d_{j_1}} X_{j_1} + \cdots+t_n \delta^{d_{j_n}} X_{j_n})x.
\end{equation*}
We let $\LebDensity$ denote the usual Lebesgue density on $\Omega$.

\begin{thm}\label{Thm::NSW::RA}
Fix a compact set $\Compact\Subset \Omega$.\footnote{We write $A\Subset B$ to mean $A$ is a relatively compact subset of $B$.}  In what follows, we write $A\lesssim B$ for $A\leq C B$ where
$C$ is a positive constant which may depend on $\Compact$, but does not depend on the particular points $x\in \Compact$ or $u\in \R^n$, or the scale $\delta\in (0,1]$; we write
$A\approx B$ for $A\lesssim B$ and $B\lesssim A$.  Under the above described hypotheses, there exists $\eta_1,\xi_0\approx 1$,
such that $\forall x\in \Compact$,
\begin{enumerate}[(i)]
\item\label{NSW::1} $\LebDensity(B_{(X,d)}(x,\delta)) \approx \Lambda(x,\delta)$, $\forall \delta\in (0,\xi_0]$.
\item\label{NSW::2} $\LebDensity(B_{(X,d)}(x,2\delta))\lesssim \LebDensity(B_{(X,d)}(x,\delta))$, $\forall \delta\in (0,\xi_0/2]$.
\item\label{NSW::3} $\forall \delta\in (0,1]$, $\Phi_{x,\delta}(B^n(\eta_1))\subseteq \Omega$ is open and $\Phi_{x,\delta}:B^n(\eta_1)\rightarrow \Phi_{x,\delta}(B^n(\eta_1))$ is a real analytic diffeomoprhism.
\item\label{NSW::4} Define $h_{x,\delta}(t)$ by $h_{x,\delta}\LebDensity=\Phi_{x,\delta}^{*}\LebDensity$.  Then, $h_{x,\delta}(t)\approx \Lambda(x,\delta)$, $\forall t\in B^n(\eta_1)$, and there exists $s\approx 1$
with $\ANorm{h_{x,\delta}}{n}{s}\lesssim \Lambda(x,\delta)$.
\item\label{NSW::5} $B_{(X,d)}(x,\xi_0\delta)\subseteq \Phi_{x,\delta}(B^n(\eta_1))\subseteq B_{(X,d)}(x,\delta)$, $\forall \delta\in (0,1]$.
\item\label{NSW::6} Let $Y_j^{x,\delta}:=\Phi^{*}_{x,\delta} \delta^{d_j} X_j$, $1\leq j\leq q$, so that $Y_j^{x,\delta}$ is a real analytic vector field on $B^n(\eta_1)$.  We have
\begin{equation}\label{Eqn::NSW::EstYs}
\ANorm{Y_j^{x,\delta}}{n}{\eta_1}[\R^n]\lesssim 1,\quad 1\leq j\leq q.
\end{equation}
Finally, $Y_1^{x,\delta}(u),\ldots, Y_q^{x,\delta}(u)$ span $T_uB^n(\eta_1)$, uniformly in $x$, $\delta$, and $u$, in the sense that
\begin{equation}\label{Eqn::NSW::Span}
\max_{k_1,\ldots, k_n\in \{1,\ldots, q\}} \inf_{u\in B^n(\eta_1)} \mleft| \det\mleft( Y_{k_1}^{x,\delta}(u)| \cdots | Y_{k_n}^{x,\delta}(u)  \mright) \mright|\approx 1.
\end{equation}
\end{enumerate}
\end{thm}

\begin{rmk}
If real analytic was replaced in this entire section with $C^\infty$ (and the estimates in \cref{Eqn::NSW::EstYs} were replaced with appropriate estimates of $C^m$ norms), then \cref{Thm::NSW::RA} is the main result, in a slightly different language, of Nagel, Stein, and Wainger's work \cite{NagelSteinWaingerBallsAndMetrics}--see \cite[\SSNSW]{StovallStreet} for a further discussion.  A main consequence of the results of this paper is that one can obtain good estimates
on the real analyticity of the vector fields $Y_1^{x,\delta},\ldots, Y_q^{x,\delta}$--see \cref{Eqn::NSW::EstYs}.
\end{rmk}

\begin{proof}
By a simple partition of unity argument, we may write
\begin{equation*}
[X_j,X_k] =\sum_{d_l\leq d_j+d_k} c_{j,k}^{l} X_l, \quad c_{j,k}^{l}\in \CjSpaceloc{\infty}[\Omega].
\end{equation*}
Using this, most of \cref{Thm::NSW::RA} is contained in \cite[\SSNSW]{StovallStreet}; the only parts which are not are those which relate to real analyticity.  In particular, \cref{NSW::1}, \cref{NSW::2}, \cref{NSW::3} (with real analytic replaced by $C^\infty$), \cref{NSW::5}, and \cref{Eqn::NSW::Span} are all explicitly stated in \cite[\SSNSW]{StovallStreet}.  Furthermore, since $X_1,\ldots, X_q$ are real analytic, $\Phi_{x,\delta}$ is real analytic (by classical theorems),
so \cref{NSW::3} follows.  Thus, the new parts are \cref{NSW::4} and \cref{Eqn::NSW::EstYs}.  These are simple consequences of \cref{Thm::Density::MainThm} and \cref{Thm::QuantRes::MainThm}, respectively,
though we will see them as part of a more general theorem:  \cref{Thm::GenSubR::MainThm}, below.

Hence, to complete the proof, we show how \cref{Thm::GenSubR::MainThm} applies to this setting.
Without loss of generality, we may shrink each $U_x$ so that $U_x\Subset \Omega$.
Set
\begin{equation*}
X_j^{\delta}:=\delta^{d_j} X_j,\quad \ch_{j,k}^{l,x,\delta}:=
\begin{cases}
\delta^{d_j+d_k-d_l} c_{j,k}^{l,x}, & d_l\leq d_j+d_k,\\
0, &\text{otherwise,}
\end{cases}
\end{equation*}
so that
\begin{equation*}
[X_j^{\delta}, X_k^{\delta}]=\sum_{l=1}^q \ch_{j,k}^{l,x,\delta} X_l^{\delta},\text{ on }U_x.
\end{equation*}
We let $X^\delta$ denote the list $X_1^{\delta},\ldots, X_q^{\delta}$.
Fix an open set $\Omega'$ with $\Compact\Subset \Omega'\Subset \Omega$ and set $\Compact_1:=\overline{\Omega'}$, so that $\Compact_1\Subset \Omega$ is compact.
Take $s_1>0$ so that $X_1,\ldots, X_q\in \ComegaSpace{s_1}[\Omega'][\R^n]$.
By the Phragmen-Lindel\"of principle, we may take $\xi'>0$ so small that $B_{X}(x,\xi')\subseteq \Omega'\subset \Compact_1$, $\forall x\in \Compact$.
$\{U_x : x\in \Compact_1\}$ is an open cover for $\Compact_1$ and we extract a finite sub-cover $U_{x_1},\ldots, U_{x_R}$.  The balls $B_X(x,\delta)$ are metric balls and the topology
induced by these balls is the same as the usual topology on $\Omega$.  Let $\xi\in (0,\min\{\xi', 1\}]$ be less than or equal to the Lebesgue number for the cover $U_{x_1},\ldots, U_{x_R}$
of $\Compact_1$, with respect to the metric associated to the balls $B_X(x,\delta)$.  Thus, since $\xi\leq \xi'$, $\forall x\in \Compact$, $\exists r\in \{1,\ldots, R\}$ with $B_X(x,\xi)=B_X(x,\xi)\cap \Compact_1\subseteq U_{x_r}$.  For this choice of $r$,
set $c_{j,k}^{l,x,\delta}:= \ch_{j,k}^{l, x_r,\delta}$.
Take $s_2>0$ so that $\forall r\in \{1,\ldots, R\}$, $c_{j,k}^{l, x_r}\in \ComegaSpace{s_2}[U_{x_r}]$.
By \cref{Prop::NelsonTheorem2} (using $X_1,\ldots, X_q\in \ComegaSpace{s_1}[\Omega'][\R^n]$), there exists $s\approx 1$ with $c_{j,k}^{l,x_r}\in \CXomegaSpace{X}{s}[U_{x_r}]$, $1\leq r\leq R$, $1\leq j,k,l\leq q$, and
\begin{equation*}
    \CXomegaNorm{c_{j,k}^{l,x_r}}{X}{s}[U_{x_r}]\lesssim \ComegaNorm{c_{j,k}^{l,x_r}}{s_1}[U_{x_r}] \lesssim 1, \quad 1\leq r\leq R, 1\leq j,k,l\leq q.
\end{equation*}
Thus, by tracing through the definitions we have, for $x\in \Compact$, $\delta\in (0,1]$, $c_{j,k}^{l,x,\delta}\in \CXomegaSpace{X^{\delta}}{s}[B_{X^{\delta}}(x,\xi)]$ and
\begin{equation*}
    \sup_{\substack{x\in \Compact \\ \delta\in (0,1]}} \CXomegaNorm{c_{j,k}^{l,x,\delta}}{X^{\delta}}{s}[B_{X^{\delta}}(x,\xi)]
    \leq \sup_{1\leq r\leq R} \CXomegaNorm{c_{j,k}^{l,x_r}}{X}{s}[U_{x_r}]
    \lesssim 1.
\end{equation*}

Define $f_j^{\delta}$ by $f_j^{\delta}\LebDensity=\Lie{X_j^{\delta}} \LebDensity$; i.e., $f_j^{\delta}=\delta^{d_j} f_j$, where $f_j:= \grad \cdot X_j$.
By our hypotheses, we have $f_j\in \CjSpaceloc{\omega}[\Omega]$, $1\leq j\leq q$.
It follows that there exists $s_3>0$ with $f_j\in \ComegaSpace{s_3}[\Omega']$, $1\leq j\leq q$.  By \cref{Prop::NelsonTheorem2} (using that $X_1,\ldots, X_q\in \ComegaSpace{s_1}[\Omega'][\R^n]$), there exists $r\approx 1$
with $f_j\in \CXomegaSpace{X}{r}[\Omega']$.  Directly from the definitions, we now have
\begin{equation*}
 \sup_{\delta\in (0,1]} \CXomegaNorm{f_{j}^{\delta}}{X^\delta}{r}[\Omega']\leq \CXomegaNorm{f_{j}}{X}{r}[\Omega']\lesssim 1.
\end{equation*}

Using the above remarks, the result now follows directly from \cref{Thm::GenSubR::MainThm}.
\end{proof}

\begin{rmk}
The most important part of \cref{Thm::NSW::RA} is \cref{NSW::6}; which allows us to think of $\Phi_{x,\delta}$ as a scaling map.  Indeed, one thinks of the vector fields $\delta^{d_1}X_1,\ldots, \delta^{d_q} X_q$ as being
``small'' (for $\delta$ small).  However, $\Phi_{x,\delta}$ gives a coordinate system in which these vector fields are unit size.  Indeed, the vector fields $Y_1^{x,\delta},\ldots, Y_q^{x,\delta}$ are the
vector fields $\delta^{d_1} X_1,\ldots, \delta^{d_q} X_q$ written in the coordinate system given by $\Phi_{x,\delta}$.  These vector fields are real analytic uniformly in $x$ and $\delta$ (i.e., \cref{Eqn::NSW::EstYs})
and span the tangent space uniformly in $x$ and $\delta$ (i.e., \cref{Eqn::NSW::Span}).  Thus, we have ``rescaled'' the vector fields to be unit size.
\end{rmk}

		\subsubsection{H\"ormander's condition}\label{Section::Scaling::Hormander}
The main way that \cref{Thm::NSW::RA} arises is via vector fields which satisfy H\"ormander's condition.  Suppose $V_1,\ldots, V_r$ are real analytic vector fields on an open set $\Omega\subseteq \R^n$.
We assume $V_1,\ldots, V_r$ satisfy H\"ormander's condition of order $m$ on $\Omega$.  I.e., we assume that the finite list of vector fields
\begin{equation*}
V_1,\ldots, V_r, \ldots, [V_i,V_j],\ldots, [V_i,[V_j,V_k]],\ldots, \ldots, \text{commutators of order }m,
\end{equation*}
span the tangent space at every point of $\Omega$.

To each $V_1,\ldots, V_r$, we assign the formal degree $1$.  If $Z$ has formal degree $e$, we assign to $[V_j,Z]$ the formal degree $e+1$.
Let $(X_1,d_1),\ldots, (X_q,d_q)$ denote the finite list of vector fields with formal degree $d_j\leq m$.  H\"ormander's condition implies
$X_1,\ldots, X_q$ span the tangent space at every point of $\Omega$.

We claim \cref{Eqn::NSW::MainAssump} holds, and therefore \cref{Thm::NSW::RA} applies to $(X_1,d_1),\ldots, (X_q,d_q)$.  Indeed, if $d_j+d_k\leq m$ we have
\begin{equation*}
[X_j,X_k]=\sum_{d_l=d_j+d_k} c_{j,k}^l X_l,
\end{equation*}
where $c_{j,k}^l$ are constants by the Jacobi identity.  If $d_j+d_k>m$, then since $X_1,\ldots, X_q$ are real analytic and span the tangent space at every point, we have
$\forall x\in \Omega$ there exists a neighborhood $U_x\subseteq \Omega$ of $x$ such that
\begin{equation*}
[X_j,X_k]=\sum_{l=1}^q c_{j,k}^{l,x} X_l=\sum_{d_l\leq d_j+d_k} c_{j,k}^{l,x} X_l, \quad c_{j,k}^{l,x}\in \CjSpace{\omega}[U_x].
\end{equation*}
Thus, \cref{Eqn::NSW::MainAssump} holds and \cref{Thm::NSW::RA} applies.

Let $\Compact \Subset \Omega$ be a compact set.  Applying \cref{Thm::NSW::RA} for $\delta\in (0,1]$, $x\in \Compact$, we obtain $\eta_1>0$ and $\Phi_{x,\delta}:B^n(\eta_1)\rightarrow B_{(X,d)}(x,\delta)$ as in that theorem.
Set $V_j^{x,\delta}:=\Phi_{x,\delta}^{*} \delta V_j$.
If $d_k=l$, then
\begin{equation*}
X_k = [V_{j_1}, [V_{j_2},\cdots,[V_{j_{l-1}}, V_{j_l}]]],
\end{equation*}
and so
\begin{equation*}
\Phi_{x,\delta}^{*} \delta^{d_k} X_k = \Phi_{x,\delta}^{*} [\delta V_{j_1}, [\delta V_{j_2}, \cdots [\delta V_{j_{l-1}}, \delta V_{j_l}]]] = [V_{j_1}^{x,\delta}, [V_{j_2}^{x,\delta},\cdots ,[V_{j_{l-1}}^{x,\delta}, V_{j_l}^{x,\delta}]]].
\end{equation*}
\Cref{Thm::NSW::RA} implies the vector fields $\Phi_{x,\delta}^{*} \delta^{d_j} X_j$ are real analytic and span the tangent space, uniformly for $x\in \Compact$, $\delta\in (0,1]$.
We conclude that the vector fields $V_1^{x,\delta},\ldots, V_q^{x,\delta}$ are real analytic and satisfy H\"ormander's condition, uniformly for $x\in \Compact$ and $\delta\in (0,1]$.
In short, the map $\Phi_{x,\delta}^{*}$ takes $\delta V_1,\ldots, \delta V_r$ to $V_1^{x,\delta},\ldots, V_r^{x,\delta}$ which are real analytic and satisfy H\"ormander's condition ``uniformly'';
i.e., it takes the case of $\delta$ small and rescales it to the case $\delta=1$, while preserving real analyticity in a quantitative way.
		
		\subsubsection{Beyond H\"ormander's condition}\label{Section::Scaling::BeyondHormander}
Let $V_1,\ldots, V_r$ be real analytic vector fields defined on an open set $\Omega\subseteq \R^n$.  It turns out that the main conclusions of \cref{Section::Scaling::Hormander}
hold without assuming $V_1,\ldots, V_r$ satisfy H\"ormander's condition, so long as one is willing to work on an injectively immersed submanifold.  We describe this here--many of these methods
appeared in \cite{SteinStreetIII} and are based on an idea of Lobry \cite{LobryControlabiliteDesSystems}.

Fix a large integer $m$ to be chosen later and a compact set $\Compact\Subset \Omega$.  Assign to each $V_1,\ldots, V_r$ the formal degree $1$.  If $Z$ has formal degree $e$, we assign to $[V_j, Z]$ the formal degree $e+1$.
Let $(X_1,d_1),\ldots, (X_q,d_q)$ denote the finite list of vector fields with formal degree $d_j\leq m$.
The results that follow are essentially independent of $m$, provided $m$ is chosen sufficiently large; how large $m$ needs to be depends on $V_1,\ldots, V_r$ and $\Compact$.
As above, for $\delta\in (0,1]$, we let $\delta^d X$ denote the list of vector fields $\delta^{d_1} X_1,\ldots, \delta^{d_q} X_q$.  We sometimes identify $\delta^d X$ with the $n\times q$ matrix
$( \delta^{d_1} X_1 | \cdots | \delta^{d_q} X_q)$.
Set $B_{(X,d)}(x,\delta):=B_{\delta^{d} X}(x,1)$, where the later ball is defined in \cref{Eqn::FuncManifold::DefnBX}.

The classical Frobenius theorem applies to the involutive distribution generated by $V_1,\ldots, V_r$ (see \cite{HermannOnTheAccessibilityProblemInControlTheory,NaganoLienarDifferentialSystemsWithSingularities,LobryControlabiliteDesSystems,SussmanOrbitsOfFamiliesOfVectorFieldsAndIntegrabilityOfDistributions})
to see that the ambient space is foliated into real analytic leaves\footnote{The various leaves may have different dimensions; i.e., this may be a singular foliation.  See \cref{Rmk::BeyondHor::UsefulForSingular} for further comments.}.  Let $L_x$ denote the leaf passing through $x$.  $V_1,\ldots, V_r$ satisfy H\"ormander's condition on each leaf.  If $m$ is sufficiently large and $\Omega'$ is an open set with $\Compact\Subset\Omega'\Subset \Omega$, then $V_1,\ldots, V_r$ satisfy H\"ormander's condition of order at most $m$ on
$L_x\cap \Omega'$, $\forall x\in \Compact$.
Therefore $X_1,\ldots, X_q$ span the tangent space at every point of $L_x\cap \Omega'$, $\forall x\in \Compact$, and $B_{(X,d)}(x,\delta)$ is an open subset of $L_x$.  Let $\nu_x$ denote the induced Lebesgue density on
the leaf passing through $x$.
For an $n\times q$ matrix $A$, and for $n_0\leq \min\{n,q\}$, let $\det_{n_0\times n_0} A$ denote the vector consisting of the determinants of the $n_0\times n_0$ submatricies of $A$--the order of the components does not matter.

For each $x\in \Omega$ set $n_0(x):=\dim \Span\{X_1(x),\ldots, X_q(x)\}=\dim L_x$.  For each $x\in \Omega$, $\delta\in (0,1]$, pick $j_1=j_1(x,\delta),\ldots, j_{n_0(x)}=j_{n_0(x)}(x,\delta)$ so that
\begin{equation*}
\mleft| \det_{n_0(x)\times n_0(x)} \mleft( \delta^{d_{j_1}} X_{j_1}(x) | \cdots | \delta^{d_{j_{n_0(x)}}} X_{j_{n_0(x)}}(x)\mright) \mright|_{\infty} = \mleft| \det_{n_0(x)\times n_0(x)} \delta^{d} X\mright|_{\infty}.
\end{equation*}
For this choice of $j_1=j_1(x,\delta),\ldots, j_{n_0(x)}=j_{n_0(x)}(x,\delta)$ set (writing $n_0$ for $n_0(x)$):
\begin{equation}\label{Eqn::Scaling::BeyondHor::DefinePhi}
\Phi_{x,\delta}(t_1,\ldots,t_{n_0}) := \exp\mleft( t_1 \delta^{d_{j_1}} X_{j_1}+\cdots+ t_{n_0} \delta^{d_{j_{n_0}}} X_{j_{n_0}} \mright)x.
\end{equation}

\begin{thm}\label{Thm::Scaling::BeyondHor}
Fix a compact set $\Compact\Subset \Omega$ and $x\in \Compact$, take $m$ sufficiently large (depending on $\Compact$ and $V_1,\ldots, V_r$), and define $(X_1,d_1),\ldots, (X_q,d_q)$ as above.
Define $n_0(x)$, $\nu_x$, and $\Phi_{x,\delta}(t_1,\ldots, t_{n_0(x)})$ as above.
We write
$A\lesssim B$ for $A\leq C B$ where $C$ is a positive constant which may depend on $\Compact$, but does not depend on the particular points $x\in \Compact$ and $u\in \R^{n_0(x)}$ under consideration,
or on the scale $\delta\in (0,1]$; we write $A\approx B$ for $A\lesssim B$ and $B\lesssim A$.  There exists $\eta_1,\xi_0\approx 1$ such that $\forall x\in \Compact$,
\begin{enumerate}[(i)]
\item\label{Item::BeyondHor::Estimatenux} $\nu_x(B_{(X,d)}(x,\delta))\approx \mleft| \det_{n_0(x)\times n_0(x)} \delta^{d} X(x) \mright|_{\infty}$, $\forall \delta\in (0,\xi_0]$.
\item\label{Item::BeyondHor::Doubling} $\nu_x(B_{(X,d)}(x,2\delta))\lesssim \nu_x(B_{(X,d)}(x,\delta))$, $\forall \delta\in (0,\xi_0/2]$.
\item\label{Item::BeyondHor::RADiffeo} $\forall \delta\in (0,1]$, $\Phi_{x,\delta}(B^{n_0(x)}(\eta_1))\subseteq L_x$ is open and $\Phi_{x,\delta}:B^n(\eta_1)\rightarrow \Phi_{x,\delta}(B^n(\eta_1))$ is a real analytic diffeomorphism.
\item\label{Item::BeyondHor::Estimateh} For $\delta\in (0,1]$, define $h_{x,\delta}(t)$ on $B^{n_0(x)}(\eta_1)$ by $h_{x,\delta}\LebDensity = \Phi_{x,\delta}^{*} \nu_x$.  Then, $h_{x,\delta}(t)\approx \mleft| \det_{n_0(x)\times n_0(x)} \delta^{d} X(x) \mright|_{\infty}$,
$\forall t\in B^{n_0(x)}(\eta_1)$, and there exists $s\approx 1$ with $\ANorm{h_{x,\delta}}{n_0(x)}{s}\lesssim  \mleft| \det_{n_0(x)\times n_0(x)} \delta^{d} X(x) \mright|_{\infty}$.
\item\label{Item::BeyondHor::Containments} $B_{(X,d)}(x,\xi_0\delta)\subseteq \Phi_{x,\delta}(B^{n_0(x)}(\eta_1))\subseteq B_{(X,d)}(x,\delta)$, $\forall \delta\in (0,1]$.
\item\label{Item::BeyongHor::Pullbacks} For $\delta\in (0,1]$, $x\in \Compact$, let $Y_j^{x,\delta}:=\Phi_{x,\delta}^{*} \delta^{d_j} X_j$, so that $Y_j^{x,\delta}$ is a real analytic vector field on $B^{n_0(x)}(\eta_1)$.  We have
\begin{equation*}
\ANorm{Y_j^{x,\delta}}{n_0(x)}{\eta_1}[\R^n]\lesssim 1, \quad 1\leq j\leq q.
\end{equation*}
Finally, $Y_1^{x,\delta}(u),\ldots, Y_q^{x,\delta}(u)$ span $T_u B^{n_0(x)}(\eta_1)$, uniformly in $x$, $\delta$, and $u$, in the sense that
\begin{equation*}
\max_{k_1,\ldots, k_{n_0(x)}\in \{1,\ldots, q\}} \inf_{u\in B^{n_0(x)}(\eta_1)} \mleft| \det\mleft( Y_{k_1}^{x,\delta}(u)| \cdots | Y_{k_{n_0(x)}}^{x,\delta}(u)  \mright) \mright|\approx 1.
\end{equation*}
\end{enumerate}
\end{thm}

\begin{rmk}
See \cref{Section::MultiParam} for a generalization of \cref{Thm::Scaling::BeyondHor} to the ``multi-parameter'' setting.
\end{rmk}

For the proof of \cref{Thm::Scaling::BeyondHor}, see \cref{Section::ScalingRevis}.

\begin{rmk}\label{Rmk::BeyondHor::UsefulForSingular}
\Cref{Thm::Scaling::BeyondHor} is useful even when restricting attention to $\delta=1$.  Indeed, the leaves $L_x$ are real analytic manifolds, and account for the foliation of $\Omega$ associated to the real analytic
vector fields $V_1,\ldots, V_r$.  This may be a singular foliation:  $\dim L_x$ may not be constant in $x$.  Suppose $x_0\in \Compact$ is a singular point; i.e., $\dim L_x$ is not constant on any neighborhood of $x_0$.
Then, the usual constructions of the real analytic coordinate systems on $L_x$ ``blow up'' as $x$ approaches $x_0$; one does not obtain a useful quantitative control of these charts.
One can think of the map $\Phi_{x,1}$ from \cref{Thm::Scaling::BeyondHor} as a real analytic coordinate system near the point $x$.  The conclusions of \cref{Thm::Scaling::BeyondHor} amount to a quantitative control
of this chart, which is uniform in $x$.  A similar sort of uniform control was an important ingredient in \cite{SteinStreetIII}, and will likely be useful in other questions from analysis regarding real analytic vector fields.
\end{rmk}

	\subsection{Generalized Sub-Riemannian Geometries}
The results described in \cref{Section::NSW} concern the classical setting of sub-Riemannian geometry; which arises in many questions, including in the study of ``maximally hypoelliptic'' differential equations (see \cite[Chapter 2]{StreetMultiParamSingInt} for details).  In \cite[\SSGenSubR]{StovallStreet}, this setting was generalized to account for certain situations which arise
for some partial differential equations which are not maximally hypoelliptic.  With the results of this paper in hand, the results from \cite[\SSGenSubR]{StovallStreet} transfer to the real analytic setting.  We
present these results here (with a few slight modifications from the setting in \cite[\SSGenSubR]{StovallStreet}).  One important thing to note is that, in this section (and unlike the settings described above) we do not
require that the vector fields be a priori real analytic.  We only require them to be $C^1$, along with certain estimates which allow us to construct a coordinate system in which they are real analytic.

Let $M$ be a connected $n$ dimensional $C^2$ manifold and for each $\delta\in (0,1]$, let $X^{\delta}=X_1^{\delta},\ldots, X_q^{\delta}$ be a list of $C^1$ vector fields on $\Omega$
which span the tangent space at every point.  For $x\in \Omega$, $\delta\in (0,1]$ set $B(x,\delta):=B_{X^{\delta}}(x,1)$,
where $B_{X^{\delta}}(x,1)$ is defined by \cref{Eqn::FuncManifold::DefnBX}.  Let $\nu$ be a $C^1$ density on $M$.  Our goal is to give conditions under
which the balls $B(x,\delta)$ when paired with the density $\nu$ give a real analytic space of homogeneous type.
Some of the conditions we give can be thought of as analogs of the axioms of a space of homogeneous type, while others can be thought of as 
endowing this space of homogeneous type with an adapted real analytic structure.
In what follows, we write $X^{\delta}$ as the column vector of vector fields $[X_1^{\delta},\ldots, X_q^{\delta}]^{\transpose}$.  Because of this, if we are given a matrix
$A:M\rightarrow \M^{q\times q}$, it makes sense to consider $A(x)X^{\delta}$ which again gives a column vector of vector fields.
It also makes sense to consider the space $L^\infty(M)$, which can be defined with any strictly positive density on $M$--all such densities give the same space and norm.

We assume the following, which are a real analytic version of the assumptions in \cite[\SSGenSubR]{StovallStreet}:
\begin{enumerate}[(I)]
\item $\forall \delta\in(0,1]$, $x\in M$, we have $\Span\{X_1^{\delta}(x),\ldots, X_q^{\delta}(x)\}=T_xM$.
\item\label{Item::GenSub::Assume::C1Norms} $X_1^{\delta},\ldots, X_q^{\delta}$ are uniformly $C^1$ in the following sense.  For every $x\in M$, there exists an open set $U\subseteq \R^n$ and a diffeomorphism $\Psi:U\rightarrow V$, where $V$ is a neighborhood of $x$ in $M$,
such that
\begin{equation*}
	\sup_{\delta\in (0,1]} \CjNorm{\Psi^{*} X_j^{\delta}}{1}[U][\R^n]<\infty.
\end{equation*}
\item $X_j^{\delta}\rightarrow 0$ and $\delta\rightarrow 0$, uniformly on compact subsets of $M$.  More precisely, for every $x\in M$, there exists an open set $U\subseteq \R^n$ and a diffeomorphism $\Psi:U\rightarrow V$, where $V$ is a neighborhood of $x$ in $M$,
such that
\begin{equation*}
\lim_{\delta\rightarrow 0} \CjNorm{\Psi^{*} X_j^{\delta}}{0}[U][\R^n]=0.
\end{equation*}

\item $\forall 0<\delta_1\leq \delta_2\leq 1$, $X^{\delta_1}=T_{\delta_1,\delta_2} X^{\delta_2}$, where $T_{\delta_1,\delta_2}\in L^\infty(M;\M^{q\times q})$ with $\Norm{T_{\delta_1,\delta_2}}[L^\infty(M;\M^{q\times q})]\leq 1$.

\item $\exists B_1,B_2\in (1,\infty)$, $b_1,b_2\in (0,1)$ such that $\forall \delta\in (0,1/B_1]$, $\exists S_\delta\in L^\infty(M;\M^{q\times q})$ and $\forall \delta\in (0,1/B_2]$, $\exists R_\delta\in L^\infty(M;\M^{q\times q})$
with $S_\delta X^{B_1 \delta} = X^{\delta}$, $R_\delta X^{\delta} = X^{B_2\delta}$, and
\begin{equation*}
	\sup_{0<\delta\leq 1/B_1} \Norm{S_\delta}[L^\infty(M;\M^{q\times q})]\leq b_1, \quad \sup_{0<\delta\leq 1/B_2} \Norm{R_\delta}[L^\infty(M;\M^{q\times q})]\leq b_2^{-1}.
\end{equation*}

\item\label{Item::GenSub::Assump::Boudnscs} For every compact set $\Compact\Subset M$, there exists $\xi\in (0,1]$ and $s>0$ such that $\forall \delta\in (0,1]$ and $x\in \Compact$, $[X_j^{\delta}, X_k^{\delta}]=\sum_{l=1}^q c_{j,k}^{l,x,\delta} X_l^{\delta}$ on $B_{X^{\delta}}(x,\xi)$, where
\begin{equation*}
	\sup_{x\in \Compact,\delta\in (0,1]} \BCXomegaNorm{c_{j,k}^{l,x,\delta}}{X^{\delta}}{s}[B_{X^{\delta}}(x,\xi)]<\infty.
\end{equation*}

\item We assume $\Lie{X_j^{\delta}} \nu = f_j^{\delta} \nu$, for $1\leq j\leq q$, $\delta\in (0,1]$, where for every relatively compact open set $\Omega'\Subset M$, there exists $s>0$ with
\begin{equation*}
	\sup_{\delta\in (0,1]} \CXomegaNorm{f_j^{\delta}}{X^{\delta}}{s}[\Omega']<\infty.
\end{equation*}
\end{enumerate}

Fix $\zeta\in (0,1]$ (we  take $\zeta=1$ for many applications).
For each $x\in M$, $\delta\in (0,1]$, pick $j_1=j_1(x,\delta),\ldots, j_n=j_n(x,\delta)\in \{1,\ldots, q\}$ so that
\begin{equation*}
	\max_{k_1,\ldots, k_n\in \{1,\ldots, q\}} \frac{ X_{k_1}^{\delta}(x)\wedge \cdots\wedge X_{k_n}^\delta(x)}{ X_{j_1}^{\delta}(x)\wedge \cdots \wedge X_{j_n}^{\delta}(x)}\leq \zeta^{-1}.
\end{equation*}
For this choice of $j_1=j_1(x,\delta),\ldots, j_n=j_n(x,\delta)$, define
\begin{equation*}
	\Phi_{x,\delta}(t_1,\ldots, t_n) := \exp(t_1 X_{j_1}^{\delta}+\cdots+ t_nX_{j_n}^{\delta})x.
\end{equation*}
Define, for $x\in M$, $\delta\in (0,1]$
\begin{equation*}
\Lambda(x,\delta):=\max_{k_1,\ldots, k_n\in \{1,\ldots, q\}} |\nu(X_{k_1},\ldots, X_{k_n})(x)|.
\end{equation*}

\begin{thm}\label{Thm::GenSubR::MainThm}
\begin{enumerate}[label=(\roman*),series=gensubriemannianenumeration]
\item\label{Item::GenSub::1} $B(x,\delta_1)\subseteq B(x,\delta_2)$, $\forall x\in M$, $0<\delta_1\leq \delta_2\leq 1$.
\item\label{Item::GenSub::2} $\bigcap_{\delta\in (0,1]} \overline{B(x,\delta)} = \{x\}$, $\forall x\in M$.
\item\label{Item::GenSub::3} $B(x,\delta)\cap B(y,\delta)\ne \emptyset \Rightarrow B(y,\delta)\subseteq B(x,C\delta)$, $\forall \delta\in (0,1/C]$, where $C=B_1^k$ and $k$ is chosen so that $b_1^k\leq \frac{1}{3}$.
\end{enumerate}
Fix a compact set $\Compact\Subset M$.  In what follows, we write $A\lesssim B$ for $A\leq C B$ where $C$ is a positive constant which may depend on $\Compact$, but does not depend on the particular point $x\in \Compact$
or on the scale $\delta\in (0,1]$.  We write $A\approx B$ for $A\lesssim B$ and $B\lesssim A$.  There exist $\eta_1, \xi_0\approx 1$ such that $\forall x\in \Compact$:
\begin{enumerate}[resume*=gensubriemannianenumeration]
\item\label{Item::GenSubR::nuSigns} $\nu$ is either everywhere strictly positive, everywhere strictly negative, or everywhere $0$.
\item\label{Item::GenSubR::nuValues} $\nu(B(x,\delta))\approx \nu\mleft(X_{j_1(x,\delta)}^{\delta},\ldots, X_{j_n(x,\delta)}^{\delta}\mright)(x)$ and $|\nu(B(x,\delta))|\approx \Lambda(x,\delta)$, $\forall \delta\in (0,\xi_0]$.
\item\label{Item::GenSubR::Doubling} $|\nu(B(x,2\delta))|\lesssim |\nu(B(x,\delta))|$, $\forall \delta\in (0,\xi_0/2]$.
\item\label{Item::GenSubR::Diffeo} $\forall \delta\in (0,1]$, $\Phi_{x,\delta}(B^n(\eta_1))\subseteq M$ is open and $\Phi_{x,\delta}:B^n(\eta_1)\rightarrow \Phi_{x,\delta}(B^n(\eta_1))$ is a $C^2$ diffeomorphism.
\item\label{Item::GenSubR::Esth} Define $h_{x,\delta}\in \CSpace{B^n(\eta_1)}$ by $h_{x,\delta}\LebDensity=\Phi_{x,\delta}^{*} \nu$.  
Then, 
\begin{equation*}
h_{x,\delta}(t)\approx \nu(X_{j_1(x,\delta)}^{\delta},\ldots, X_{j_n(x,\delta)}^{\delta})(x)\text{ and } |h_{x,\delta}(t)|\approx \Lambda(x,\delta),\quad  \forall t\in B^n(\eta_1).
\end{equation*}
Also, $h_{x,\delta}\in \ASpace{n}{\eta_1}$ and $\ANorm{h_{x,\delta}}{n}{\eta_1}\lesssim \Lambda(x,\delta)$.
\item\label{Item::GenSubR::Contaiment} $B(x,\xi_0\delta)\subseteq \Phi_{x,\delta}(B^n(\eta_1))\subseteq B(x,\delta)$, $\forall \delta\in (0,1]$.
\item\label{Item::GenSubR::Ysmooth} Let $Y_j^{x,\delta}:=\Phi_{x,\delta}^{*} X_j^{\delta}$, $1\leq j\leq q$, so that $Y_j^{x,\delta}$ is a vector field on $B^{n}(\eta_1)$.  Then, $Y_j^{x,\delta}\in \ASpace{n}{\eta_1}$ and
\begin{equation*}
	\ANorm{Y_j^{x,\delta}}{n}{\eta_1}[\R^n]\lesssim 1, \quad 1\leq j\leq q.
\end{equation*}
Furthermore, $Y_1^{x,\delta}(u),\ldots, Y_q^{x,\delta}(u)$ span $T_u B^n(\eta_1)$, uniformly in $x$, $\delta$, and $u$ in the sense that
\begin{equation*}
	\max_{k_1,\ldots, k_n \in \{1,\ldots, q\}} \inf_{u\in B^n(\eta_1)} \mleft| \det\mleft( Y_{k_1}^{x,\delta}(u) | \cdots| Y_{k_n}^{x,\delta}(u)  \mright) \mright|\approx 1.
\end{equation*}
\end{enumerate}
\end{thm}
\begin{proof}
\Cref{Item::GenSub::1}, \cref{Item::GenSub::2}, and \cref{Item::GenSub::3} all follow just as in the corresponding results in \cite[\SSGenSubResult]{StovallStreet}.
For the remaining parts, the goal is to apply \cref{Thm::QuantRes::MainThm}, \cref{Thm::Density::MainThm}, and \cref{Cor::Density::MainCor} to the vector fields
$X_1^{\delta},\ldots, X_q^{\delta}$, for $\delta\in (0,1]$ and for each base point $x_0\in \Compact$ (uniformly for $\delta\in (0,1]$, $x_0\in \Compact$)--we use the choice
of $\xi$ and $\nu$ from above, and take $\zeta=1$.

By the Picard--Lindel\"of theorem, we make take $\eta\in (0,1]$ depending on $\Compact$ and the bounds from \cref{Item::GenSub::Assume::C1Norms}, so that
$\forall x\in \Compact$, $\delta\in (0,1]$, $X_1^{\delta},\ldots, X_q^{\delta}$ satisfy $\sC(x,\eta,M)$.
Take $\delta_0>0$ as in \cref{Lemma::PfQual::Existsetadelta}, when applied to $X_1^{\delta},\ldots, X_q^{\delta}$.  It can be seen from the proof of \cref{Lemma::PfQual::Existsetadelta} (which can be found
in \cite[\SSLemmaMoreOnAssump]{StovallStreet}) that $\delta_0$ can be chosen independent of $\delta\in (0,1]$ (this uses \cref{Item::GenSub::Assume::C1Norms}).
In light of \cref{Item::GenSub::Assump::Boudnscs} (see, also, \cref{Rmk::QuantRest::CanUseComegaNormsInstead}), \cref{Thm::QuantRes::MainThm}, \cref{Thm::Density::MainThm}, and \cref{Cor::Density::MainCor}
apply to the vector fields $X_1^{\delta},\ldots, X_q^{\delta}$ for each $\delta\in (0,1]$ and each base point $x\in \Compact$ (with $\eta$ replaced by $\min\{\eta,s\}$).  Each constant which is
$0$-admissible, admissible, $\nu$-admissible, or $0;\nu$-admissible in these results can be chosen independent of $x\in \Compact$ and $\delta\in (0,1]$.  Let $\xi_1,\xi_2,\eta_1>0$ be as in \cref{Thm::QuantRes::MainThm},
so that $\xi_1$, $\xi_2$, and $\eta_1$ can be chosen independent of $x\in \Compact$ and $\delta\in (0,1]$.  The map $\Phi_{x,\delta}$ is precisely the map $\Phi$ from \cref{Thm::QuantRes::MainThm}
when using the base point $x$ and the vector fields $X_1^{\delta},\ldots, X_q^{\delta}$.

\Cref{Item::GenSubR::Diffeo} follows from \cref{Thm::QuantRes::MainThm} \cref{Item::QuantRes::PhiOpen} and \cref{Item::QuantRes::PhiDiffeo}.
\Cref{Item::GenSubR::Esth} follows from \cref{Thm::Density::MainThm} and \cref{Cor::Density::MainCor}.
\Cref{Item::GenSubR::Esth} implies that on a neighborhood of each point, $\nu$ is either strictly positive, strictly negative, or identically $0$.  Since $M$ is connected, it follows that
$\nu$ is either everywhere strictly positive, everywhere strictly negative, or everywhere $0$; i.e., \cref{Item::GenSubR::nuSigns} holds.
By multiplying $\nu$
by $\pm 1$, we may henceforth assume (without loss of generality) that $\nu$ is everywhere non-negative--and is either identically $0$ or everywhere strictly positive.

\Cref{Item::GenSubR::Contaiment}: \cref{Thm::QuantRes::MainThm} gives $\xi_2\approx 1$ ($\xi_2<1$) such that
\begin{equation*}
	B_{X^{\delta}}(x,\xi_2)\subseteq \Phi_{x,\delta}(B^n(\eta_1))\subseteq B_{X^{\delta}}(x,\xi)\subseteq B_{X^{\delta}}(x,1)=B(x,\delta).
\end{equation*}
Thus, to prove \cref{Item::GenSubR::Contaiment}, we wish to show $\exists \xi_0\approx 1$ with 
\begin{equation}\label{Eqn::PfQual::Containmentxi0}
B(x,\xi_0\delta)\subseteq B_{X^{\delta}}(x,\xi_2).
\end{equation}
This follows just as in \cite[\SSGenSubResult]{StovallStreet}, where it is shown that we may take $\xi_0=B_1^{-k}$, where $k$ is chosen so that $b_1^k\leq \xi_2$.

We claim, for $\delta_1\leq \delta_2\leq 1$,
\begin{equation}\label{Eqn::PfQual::LambdaIncrease}
	\Lambda(x,\delta_1)\lesssim \Lambda(x,\delta_2),
\end{equation}
where the implicit constant can be chosen to depend only on $q$.  Indeed,
\begin{equation*}
\begin{split}
\Lambda(x,\delta_1) &= \max_{k_1,\ldots, k_n\in \{1,\ldots, q\}} \mleft| \nu(X_{k_1}^{\delta_1},\ldots, X_{k_{n}}^{\delta_1})(x)  \mright|
\\&= \max_{k_1,\ldots, k_n\in \{1,\ldots, q\}} \mleft| \nu\mleft( (T_{\delta_1,\delta_2} X^{\delta_2})_{k_1}, \ldots, (T_{\delta_1,\delta_2} X^{\delta_2})_{k_n}  \mright)(x) \mright|.
\end{split}
\end{equation*}
Since $\Norm{T_{\delta_1,\delta_2}}[L^\infty(M;\M^{q\times q})]\leq 1$, the right hand side is $\nu$ evaluated at a linear combination, with (variable) coefficients bounded by $1$, 
of the vector fields $X_1^{\delta_2},\ldots, X_q^{\delta_2}$.  Using the properties of densities, it follows that
\begin{equation*}
\mleft| \nu\mleft( (T_{\delta_1,\delta_2} X^{\delta_2})_{k_1}, \ldots, (T_{\delta_1,\delta_2} X^{\delta_2})_{k_n}  \mright)(x) \mright|\lesssim \Lambda(x,\delta_2),
\end{equation*}
\cref{Eqn::PfQual::LambdaIncrease} follows.

Next we claim, for $c>0$ fixed,
\begin{equation}\label{Eqn::PfQual::LambdaConst}
\Lambda(x,c\delta)\approx \Lambda(x,\delta),\quad \delta, c\delta\in (0,1],
\end{equation}
where the implicit constant depends on $c$, but not on $x$ or  $\delta$.  It suffices to prove \cref{Eqn::PfQual::LambdaConst} for $c<1$.  By \cref{Eqn::PfQual::LambdaIncrease}, it suffices to prove 
\cref{Eqn::PfQual::LambdaConst} for $c=B_2^{-k}$ for some $k$.  We have
\begin{equation}\label{Eqn::PfQual::DoThecs}
\begin{split}
	\Lambda(x,\delta) &= \max_{k_1,\ldots, k_n\in \{1,\ldots, q\}} \mleft| \nu\mleft( X_{k_1}^{\delta},\ldots, X_{k_n}^{\delta} \mright)(x) \mright|
	\\&=\max_{k_1,\ldots, k_n\in \{1,\ldots, q\}}  \mleft| \nu\mleft( (AX^{c\delta})_{k_1},\ldots, (AX^{c\delta})_{k_n} \mright)(x) \mright|,
\end{split}
\end{equation}
where $A(x) = R_{B_2^{-1}\delta}(x) R_{B_2^{-2}\delta}(x)\cdots R_{B_2^{-k}\delta}(x)$.  Since $\sup_{x\in M} \Norm{A(x)}[\M^{q\times q}]\leq b_2^{-k}\lesssim 1$ (where the implicit constant depends on $k$), it follows that the
right hand side of \cref{Eqn::PfQual::DoThecs} is $\nu$ evaluated at linear combinations, with (variable) coefficients which have absolute value $\lesssim 1$, of the vectors
$X_1^{c\delta},\ldots, X_q^{c\delta}$.  It follows from the properties of densities that
\begin{equation*}
	\mleft| \nu\mleft( (AX^{c\delta})_{k_1},\ldots, (AX^{c\delta})_{k_n} \mright)(x) \mright|\lesssim \Lambda(x,c\delta).
\end{equation*}
We conclude $\Lambda(x,\delta)\lesssim \Lambda(x,c\delta)$.  Combining this with \cref{Eqn::PfQual::LambdaIncrease} proves \cref{Eqn::PfQual::LambdaConst}.

\Cref{Cor::Density::MainCor} and using that we have (without loss of generality) assumed $\nu$ is non-negative, shows
\begin{equation}\label{Eqn::PfQual::nuapproxLambda}
\nu(B_{X^{\delta}}(x,\xi_2))\approx \Lambda(x,\delta).
\end{equation}
Combining \cref{Eqn::PfQual::nuapproxLambda} with \cref{Eqn::PfQual::LambdaConst} and \cref{Eqn::PfQual::Containmentxi0} shows
\begin{equation}\label{Eqn::PfQual::nuLessLambda}
\nu(B(x,\xi_0\delta))\leq \nu(B_{X^{\delta}}(x,\xi_2)) \approx \Lambda(x,\delta)\approx \Lambda(x,\xi_0\delta).
\end{equation}
Conversely, using \cref{Eqn::PfQual::nuapproxLambda} again, we have
\begin{equation}\label{Eqn::PfQual::LambdaLessnu}
\Lambda(x,\delta)\approx \nu(B_{X^{\delta}}(x,\xi_2))\leq \nu(B_{X^{\delta}}(x,1))=\nu(B(x,\delta)).
\end{equation}
Combining \cref{Eqn::PfQual::nuLessLambda} and \cref{Eqn::PfQual::LambdaLessnu} proves $|\nu(B(x,\delta))|\approx \Lambda(x,\delta)$, $\forall \delta\in (0,\xi_0]$.
Since we have assumed $\nu$ is everywhere non-negative, \cref{Item::GenSubR::nuValues} follows from this and \cref{Cor::Density::MainCor}.
\Cref{Item::GenSubR::Doubling} follows from \cref{Item::GenSubR::nuValues} and \cref{Eqn::PfQual::LambdaConst}.
\Cref{Item::GenSubR::Ysmooth} follows from \cref{Thm::QuantRes::MainThm}.
\end{proof}

\section{Results from the first paper}\label{Section::Part1}
In this section, we describe the main results needed from \cite{StovallStreet}.
We do not require as detailed information as is discussed in that paper, and so we instead
state an immediate consequence of the results in that paper.

We take the same setup and hypotheses as in \cref{Thm::QuantRes::MainThm}, and define $0$-admissible constants
and admissible constants as in that theorem.  We take $\Phi(t)$ as in \cref{Eqn::QuantRes::DefnPhi}.
We separate the results we need into two parts.

\begin{prop}\label{Prop::Part1::chi}
There exists a $0$-admissible constant $\chi\in (0,\xi]$ such that:
\begin{enumerate}[(a)]
\item $\forall y\in B_{X_{J_0}}(x_0,\chi)$, $\bigwedge X_{J_0}(y)\ne 0$.
\item $\forall y\in B_{X_{J_0}}(x_0,\chi)$,
\begin{equation*}
\sup_{J\in \sI(n,q)} \left|\frac{\bigwedge X_J(y)}{\bigwedge X_{J_0}(y)}\right|\approx_0 1.
\end{equation*}
\item $\forall \chi'\in (0,\chi]$, $B_{X_{J_0}}(x_0,\chi')$ is an open subset of $B_X(x_0,\xi)$ and is therefore a submanifold.
\end{enumerate}
\end{prop}
\begin{proof}
This is contained in \cite[\SSMainResult]{StovallStreet}.
\end{proof}

\begin{prop}\label{Prop::PartI::MainProp}
In the special case $n=q$ (so that $X_{J_0}=X$), there exists an admissible constant $\etai\in (0,\eta_0]$ such that:
\begin{enumerate}[label=(\alph*),series=partitheoremenumeration]
\item\label{Item::Part1::PhiOpen} $\Phi(B^n(\etai))$ is an open subset of $B_{X}(x_0,\xi)$, and is therefore a submanifold.
\item\label{Item::Part1::PhiDiffeo} $\Phi:B^n(\etai)\rightarrow \Phi(B^n(\etai))$ is a $C^2$ diffeomorphism.
\end{enumerate}
Let $Y_j=\Phi^{*}X_j$ and write $Y=(I+A)\grad$, where $Y$ denotes the column vector of vector
fields $Y=[Y_1,\ldots, Y_n]^{\transpose}$, $\grad$ denotes the gradient in $\R^n$ thought of as a column vector,
and $A\in \CSpace{B^n(\etai)}[\M^{n\times n}]$.
\begin{enumerate}[resume*=partitheoremenumeration]
\item\label{Item::Part1::ASize} $A(0)=0$. 
\end{enumerate}
For $t\in B^n(\etai)$, let $C(t)$ denote the $n\times n$ matrix with $j,k$ component given by
$\sum_{l=1}^n t_l c_{j,l}^k(\Phi(t))$.  For $t\in B^n(\etai)$ write $t$ in polar coordinates $t=r\theta$.
\begin{enumerate}[resume*=partitheoremenumeration]
\item\label{Item::Part1::ADiffEq} $A$ satisfies the differential equation
\begin{equation}\label{Eqn::Part1::ADiffEq}
\diff{r} r A(r\theta) = -A(r\theta)^2 - C(r\theta)A(r\theta)-C(r\theta).
\end{equation}
\end{enumerate}
\end{prop}
\begin{proof}
All of the results except \cref{Item::Part1::ADiffEq} are contained in \cite[\SSDifferentOneAdmis]{StovallStreet}.
\cite[\SSDifferentOneAdmis]{StovallStreet} uses the notion of a $1'$-admissible constant, which involves bounds on
$\CjNorm{c_{j,k}^l\circ \Phi}{1}[B^n(\eta_0)]$.
Our assumptions imply $\ANorm{c_{j,k}^l\circ \Phi}{n}{\eta_0}\lesssim 1$; this does not quite imply $c_{j,k}^l\circ \Phi\in \CjSpace{1}[B^n(\eta_0)]$ (the problem
is that while $c_{j,k}^l\circ \Phi$ is $C^1$, its first derivatives might not be bounded on $B^n(\eta_0)$).
Instead, we proceed as follows.  By defining $\etat:=\eta_0/2$, we do have
$\CjNorm{c_{j,k}^l\circ \Phi}{1}[B^n(\etat)]\lesssim \ANorm{c_{j,k}^l\circ \Phi}{n}{\eta_0}\lesssim 1$.  Applying \cite[\SSDifferentOneAdmis]{StovallStreet}
with $\eta$ replaced by $\etat$ yields all of the above except \cref{Item::Part1::ADiffEq}.

\Cref{Item::Part1::ADiffEq} is an immediate consequence of \cite[\SSDeriveODE]{StovallStreet}.
\end{proof}

\section{Proofs}\label{Section::Proofs}
In this section, we prove the main results of this paper; namely, \cref{Thm::QualRes::LocalThm,Thm::QualRes::GlobalThm,Thm::QuantRes::MainThm,Thm::Density::MainThm}.
In \cref{Section::Proofs::RAAndODE} we  describe the main way we show functions are real analytic.  Namely, we prove the function in question satisfies an appropriate real analytic ODE, which
forces it to be real analytic; \cref{Section::Proofs::RAAndODE} contains several quantitative  instances of this.
In \cref{Section::Proofs::QuantiativeThm} we prove the main quantitative theorem:  \cref{Thm::QuantRes::MainThm}.  In \cref{Section::Proofs::Densities} we study densities
and prove \cref{Thm::Density::MainThm}.  Finally, in \cref{Section::Proofs::Qual}, we prove the qualitative results (\cref{Thm::QualRes::LocalThm,Thm::QualRes::GlobalThm}),
which are simple consequences of \cref{Thm::QuantRes::MainThm}.

	\subsection{Real Analytic Functions and ODEs}\label{Section::Proofs::RAAndODE}
At various points, we will need to prove functions are real analytic.  The way we will do this is by showing the functions
satisfy a ODE which depends real analytically on the appropriate parameters.
We begin with a simple and classical version of this where we deduce the solutions to a certain real analytic PDE are real analytic, by reducing it to a real analytic ODE.

\begin{prop}\label{Prop::PfRealAl::IdentifyRA::Basic}
Let $\Omega\subseteq \R^n$ be an open set and $f\in \CjSpace{1}[\Omega][\R^m]$ satisfy
\begin{equation*}
\diff{t_j} f(t) = F_j(t,f(t)), \quad \forall t\in \Omega, 1\leq j\leq n,
\end{equation*}
where $F_j$ is real analytic in both variables.  Then, $f$ is real analytic.
\end{prop}
\begin{proof}[Proof Sketch]
Fix $s\in \Omega$.  We will show $f$ is real analytic near $s$.
Set $g(\epsilon,t):=f(\epsilon t+s)$.  Then, we have
\begin{equation*}
\diff{\epsilon} g(\epsilon,t) = \sum_{j=1}^n t_j \frac{\partial f}{\partial t_j}(\epsilon t+s)
=\sum_{j=1}^n t_j F_j(\epsilon t+s,f(\epsilon t+s))=:G_s(\epsilon,t,g(\epsilon,t)),
\end{equation*}
where $G_s$ is analytic in all its variables.  Also, $g(0,t)=f(s)$ (which is constant in $t$).

Hence, $g(\epsilon,t)$ satisfies a real analytic ODE, and classical results show $g$ is real analytic for $\epsilon$ and $t$ small.
Since $g(\epsilon,t) = f(\epsilon t+s)$, this shows that $f$ is real analytic near $s$, completing the proof.
\end{proof}
	
		\subsubsection{A Particular ODE}
Fix $D,\etai>0$ and for $1\leq j\leq n$, let $C_j\in \ASpace{n}{\etai}[\M^{n\times n}]$
with $\sum_{j=1}^n \ANorm{C_j}{n}{\etai}[\M^{n\times n}]\leq D$.
Set $C(t):=\sum_{j=1}^n t_j C_j(t)$.  
For $t\in B^n(\etai)$, we write $t$ in polar coordinates $t=r\theta$.
We consider the differential equation, defined for functions $A(t)$ taking values in $\M^{n\times n}$, given by
\begin{equation}\label{Eqn::PartODE::MainEqn}
\diff{r} rA(r\theta)=-A(r\theta)^2-C(r\theta)A(r\theta)-C(r\theta),\quad A(0)=0.
\end{equation}

\begin{prop}\label{Prop::PartODE::MainProp}
Let $\eta_1\in (0, \min\{\etai,5/8D\}]$.  There exists a solution $A\in \ASpace{n}{\eta_1}[\M^{n\times n}]$
to \cref{Eqn::PartODE::MainEqn}.  Moreover, this solution satisfies $\ANorm{A}{n}{\eta_1}[\M^{n\times n}]\leq \frac{1}{2}$.
Finally, this solution is unique in the sense that if $B(t)\in \CSpace{B^n(\delta)}[\M^{n\times n}]$ 
is another solution to \cref{Eqn::PartODE::MainEqn}, then $A(t)=B(t)$ for $|t|<\min\{\delta,\eta_1\}$.
\end{prop}

The rest of this section is devoted to the proof of \cref{Prop::PartODE::MainProp}.
Following the proof in \cite[\SSExistODE]{StovallStreet}, we introduce the
map $\sT:\ASpace{n}{\eta_1}[\M^{n\times n}]\rightarrow \ASpace{n}{\eta_1}[\M^{n\times n}]$ given by
\begin{equation*}
\sT(A)(x):=-\int_0^1 A(sx)^2 +C(sx)A(sx)+C(sx)\: ds.
\end{equation*}
Using that $\ASpace{n}{\eta_1}[\M^{n\times n}]$ is an algebra (\cref{Lemma::FuncSpaceRev::Algebra}) it is immediate to verify
$\sT:\ASpace{n}{\eta_1}[\M^{n\times n}]\rightarrow \ASpace{n}{\eta_1}[\M^{n\times n}]$.
A simple change of variables shows, for $r>0$,
\begin{equation*}
\sT(A)(r\theta)=\frac{1}{r}\int_0^r -A(s\theta)^2-C(s\theta)A(s\theta)-C(s\theta)\: ds.
\end{equation*}
Thus, $A$ is a solution to \cref{Eqn::PartODE::MainEqn} if and only if $\sT(A)=A$ and $A(0)=0$.
We will prove the existence of such a fixed point by using the contraction mapping principle.

\begin{lemma}\label{Lemma::PartODE::NormA}
Let $A\in \ASpace{n}{\eta_1}[\M^{n\times n}]$ satisfy $A(0)=0$.
For $s\in [0,1]$ set $A_s(x)=A(sx)$.
Then, $\ANorm{A_s}{n}{\eta_1}[\M^{n\times n}]\leq s \ANorm{A}{n}{\eta_1}[\M^{n\times n}]$.
\end{lemma}
\begin{proof}This is immediate from the definitions.\end{proof}

\begin{lemma}\label{Lemma::PartODE::NormC}
For $s\in [0,1]$, set $C_s(t)=C(st)$.  Then,
$\ANorm{C_s}{n}{\eta_1}[\M^{n\times n}]\leq \eta_1 s D$.
\end{lemma}
\begin{proof}
Let $C_j(t) =\sum_{\alpha\in \N^n} \frac{c_{\alpha,j}}{\alpha!} t^{\alpha}$.
Then, $C_s(t) = \sum_{\alpha\in \N^n} \sum_{j=1}^n \frac{c_{\alpha,j}}{\alpha!} s^{|\alpha|+1} t_j t^{\alpha}$.  Thus,
\begin{equation*}
\begin{split}
\ANorm{C_s}{n}{\eta_1}[\M^{n\times n}] &\leq \sum_{j=1}^n \sum_{\alpha\in \N^n} \frac{\Norm{c_{\alpha,j}}[\M^{n\times n}]}{\alpha!} s^{|\alpha|+1} \eta_1^{|\alpha|+1}
\leq s\eta_1 \sum_{j=1}^n \sum_{\alpha\in \N^n} \frac{\Norm{c_{\alpha,j}}[\M^{n\times n}]}{\alpha!}  \eta_1^{|\alpha|}
\\&=s\eta_1 \sum_{j=1}^n \ANorm{C_j}{n}{\eta_1}[\M^{n\times n}] \leq s\eta_1 D,
\end{split}
\end{equation*}
completing the proof.
\end{proof}

Define 
$$\sM:=\left\{A\in \ASpace{n}{\eta_1}[\M^{n\times n}]  :  A(0)=0\text{ and }\ANorm{A}{n}{\eta_1}[\M^{n\times n}]\leq \frac{1}{2}\right\}.$$
We give $\sM$ the induced metric as a subset of $\ASpace{n}{\eta_1}[\M^{n\times n}]$.  With this metric, $\sM$ is a complete metric space.

\begin{lemma}\label{Lemma::PartODE::Contraction}
For $\eta_1 \in (0, \min\{\etai,5/8D\}]$,
$\sT:\sM\rightarrow \sM$ and is a strict contraction.
\end{lemma}
\begin{proof}
Using that $\ANorm{B_1B_2}{n}{\eta_1}[\M^{n\times n}]\leq \ANorm{B_1}{n}{\eta_1}[\M^{n\times n}]\ANorm{B_2}{n}{\eta_1}[\M^{n\times n}]$ (\cref{Lemma::FuncSpaceRev::Algebra})
and \cref{Lemma::PartODE::NormA,Lemma::PartODE::NormC},
we have, for $A\in \sM$,
\begin{equation*}
\begin{split}
&\ANorm{\sT(A)}{n}{\eta_1}\leq \int_0^1 \ANorm{A(s\cdot)^2}{n}{\eta_1}+ \ANorm{C(s\cdot)A(s\cdot)}{n}{\eta_1}+ \ANorm{C(s\cdot)}{n}{\eta_1}\: ds
\\&\leq \int_0^1 s^2 \ANorm{A}{n}{\eta_1}^2 + (D\eta_1 s^2) \ANorm{A}{n}{\eta_1} + D\eta_1 s\: ds
\leq \frac{1}{3}\cdot \frac{1}{4} + \frac{D\eta_1}{3}\cdot \frac{1}{2} + \frac{D\eta_1}{2}
\\&=\frac{1}{12}+\frac{2}{3} D\eta_1
\leq \frac{1}{12}+\frac{2}{3} \cdot \frac{5}{8} = \frac{1}{2}.
\end{split}
\end{equation*}
Clearly, since $A(0)=0$ and $C(0)=0$, we have $\sT(A)(0)=0$.  We conclude $\sT:\sM\rightarrow \sM$.

For $A,B\in \sM$, we have using $A^2-B^2 = \frac{1}{2}(A+B)(A-B)+\frac{1}{2}(A-B)(A+B)$,
\begin{equation*}
\begin{split}
&\ANorm{\sT(A)-\sT(B)}{n}{\eta_1} \leq \int_0^1 \ANorm{A(s\cdot)+B(s\cdot)}{n}{\eta_1} \ANorm{A(s\cdot)-B(s\cdot)}{n}{\eta_1} + \ANorm{C(s\cdot)}{n}{\eta_1} \ANorm{A(s\cdot)-B(s\cdot)}{n}{\eta_1}\: ds
\\&\leq \int_0^1 s^2 \ANorm{A-B}{n}{\eta_1} + D\eta_1 s^2 \ANorm{A-B}{n}{\eta_1}\: ds\leq \frac{1+D\eta_1}{3} \ANorm{A-B}{n}{\eta_1}\leq \frac{13}{24} \ANorm{A-B}{n}{\eta_1},
\end{split}
\end{equation*}
completing the proof.
\end{proof}

\begin{proof}[Proof of \cref{Prop::PartODE::MainProp}]
Uniqueness for \cref{Eqn::PartODE::MainEqn} was established in \cite[\SSExistODE]{StovallStreet};
and the same proof yields the claimed uniqueness in \cref{Prop::PartODE::MainProp}.
For existence, \cref{Lemma::PartODE::Contraction} shows that the contraction mapping principle applies to $\sT:\sM\rightarrow \sM$ to show that there is a unique fixed point, $A\in \sM$, of $\sT$.  As described above, this $A$ is a solution to \cref{Eqn::PartODE::MainEqn}
and clearly satisfies $\ANorm{A}{n}{\eta_1}[\M^{n\times n}]\leq \frac{1}{2}$ (since $A\in \sM$).  This completes the proof.
\end{proof}

		
		\subsubsection{Identifying Real Analytic Functions I: Euclidean Space}
In \cref{Prop::PfRealAl::IdentifyRA::Basic}, we showed how to prove a function was real analytic by introducing
a new variable and proving the function satisfied an ODE in this new variable.  In this section, we present a quantitative version
of a similar argument.  We make no effort to state the result in the greatest generality, and instead focus on the setting needed for this paper.

Fix $n,N,L\in \N$ and $r>0$.  We consider functions $F(t)=(F_1(t),\ldots, F_N(t))\in \CSpace{B^n(r)}[\R^N]$ satisfying a certain ODE.
For each $1\leq l\leq N$, $1\leq j\leq n$, and $\alpha\in \N^N$ with $|\alpha|\leq L$ ($\alpha$ a multi-index), let
$a_{\alpha,j,l}\in \ASpace{n}{r}$. For $1\leq l\leq N$, fix $F_{l,0}\in \R$.  We consider the following system of differential equations for $1\leq l\leq N$:
\begin{equation}\label{Eqn::PfIdentRAEuclid::MainEquation}
\diff{\epsilon} F_l(\epsilon t) = \sum_{j=1}^n \sum_{\substack{\alpha\in \N^N \\ |\alpha|\leq L}} t_j a_{\alpha, j,l}(\epsilon t) F(\epsilon t)^{\alpha}, \quad F_l(0) = F_{l,0}.
\end{equation}
Fix $D$ such that
\begin{equation*}
	|F_{l,0}|, \ANorm{a_{\alpha,j,l}}{n}{r}\leq D, \quad 1\leq j\leq n, 1\leq l\leq N, |\alpha|\leq L.
\end{equation*}

\begin{prop}\label{Prop::PfIdentRAEuclid::MainProp}
Set $r':=\min\{ r, D(n 2^L (L+1)^N (\max\{1,D\})^{L+1})^{-1}, ( n(L+1)^{N+1} (\max\{1,D\})^{L} 2^L)^{-1} \}$.  Suppose $F\in \CSpace{B^n(r)}[\R^N]$ satisfies \cref{Eqn::PfIdentRAEuclid::MainEquation}.
Then, $F\big|_{B^n(r')}\in \ASpace{n}{r'}[\R^N]$.  Moreover, for each $1\leq l\leq N$, $\ANorm{F_l}{n}{r'}\leq 2D$.
\end{prop}

To prove \cref{Prop::PfIdentRAEuclid::MainProp} we prove a more general auxiliary result where we separate $\epsilon t$ into two variables $(\epsilon,t)$.  To this end,
set 
\begin{equation}\label{Eqn::PfIdentREEuclid::Definerp}
r':=\min\{ r, D(n 2^L (L+1)^N (\max\{1,D\})^{L+1})^{-1}, ( n(L+1)^{N+1} (\max\{1,D\})^{L} 2^L)^{-1} \}.
\end{equation}
We will consider functions $\Ft(\epsilon, t):[0,1]\times B^n(r')\rightarrow \R^N$,
and we will think of these as functions $\Ft(\epsilon, t)\in \ASpace{n}{r'}[\CSpace{[0,1]}[\R^N]]$.  I.e.,
\begin{equation*}
	\Ft(\epsilon,t) = \sum_{\beta\in \N^n} \frac{t^{\beta}}{\beta!} c_\beta(\epsilon),
\end{equation*}
where $c_\beta\in \CSpace{[0,1]}[\R^N]$ and
\begin{equation*}
	\ANorm{\Ft}{n}{r'}[\CSpace{[0,1]}[\R^N]] = \sum_{\beta\in \N^n} \frac{ (r')^{|\beta|} }{\beta!} \CNorm{c_\beta}{[0,1]}[\R^N].
\end{equation*}
For $1\leq j\leq n$, $1\leq l\leq N$, and $\alpha\in \N^N$ with $|\alpha|\leq L$, let $\at_{\alpha,j,l}(\epsilon,t)\in \ASpace{n}{r}[\CSpace{[0,1]}]$.
For $1\leq l\leq N$, fix a constant $\Ft_{l,0}\in \R$.  We consider the following system of differential equations for $1\leq l\leq N$:
\begin{equation}\label{Eqn::PfIdentRAEuclid::AuxEquation}
\diff{\epsilon} \Ft_l(\epsilon,t) = \sum_{j=1}^n \sum_{\substack{\alpha\in \N^N \\ |\alpha|\leq L}} t_j \at_{\alpha,j,l}(\epsilon, t) \Ft(\epsilon, t)^{\alpha}, \quad \Ft_l(0,t)\equiv \Ft_{l,0}.
\end{equation}
We suppose:
\begin{equation*}
	|\Ft_{l,0}|, \ANorm{\at_{\alpha,j,l}}{n}{r}[\CSpace{[0,1]}]\leq D, \quad 1\leq j\leq n, 1\leq l\leq N, |\alpha|\leq L.
\end{equation*}

\begin{prop}\label{Prop::PfIdentRAEuclid::AuxProp}
Let $r'$ be given by \cref{Eqn::PfIdentREEuclid::Definerp}.
There exists $\Ft(\epsilon, t)\in \ASpace{n}{r'}[\CSpace{[0,1]}[\R^N]]$ satisfying \cref{Eqn::PfIdentRAEuclid::AuxEquation}.  This solution satisfies
\begin{equation}\label{Eqn::PfIndentRAEuclid::AuxBound}
	\max_{1\leq l\leq N} \ANorm{\Ft_l}{n}{r'}[\CSpace{[0,1]}] \leq 2D.
\end{equation}
Finally, this solution is unique in the sense that if for some $\delta>0$,
$\Fh(\epsilon, t)\in \CSpace{[0,1]\times B^n(\delta)}[\R^N]$ is another solution to \cref{Eqn::PfIdentRAEuclid::AuxEquation}, then $\Ft(\epsilon, t)=\Fh(\epsilon, t)$
for $|t|<\min\{ \delta, r'\}$ and $\epsilon\in [0,1]$.
\end{prop}

The rest of this section is devoted to the proofs of \cref{Prop::PfIdentRAEuclid::MainProp,Prop::PfIdentRAEuclid::AuxProp}.  We begin with \cref{Prop::PfIdentRAEuclid::AuxProp}, the existence portion of which we will prove
using the contraction mapping principle.

\begin{lemma}\label{Lemma::PfIdentRAEuclid::Integragte}
For $f(\epsilon,t)\in \ASpace{n}{r'}[\CSpace{[0,1]}]$, set $g(\epsilon, t):=\int_0^\epsilon f(\epsilon',t)\: d\epsilon'$.  Then, $g\in \ASpace{n}{r'}[\CSpace{[0,1]}]$ and
\begin{equation*}
	\ANorm{g}{n}{r'}[\CSpace{[0,1]}]\leq \ANorm{f}{n}{r'}[\CSpace{[0,1]}].
\end{equation*}
\end{lemma}
\begin{proof}This is immediate from the definitions.\end{proof}

For $\Ft(\epsilon,t)=(\Ft_1(\epsilon,t),\ldots, \Ft_N(\epsilon,t))\in \ASpace{n}{r'}[\CSpace{[0,1]}[\R^N]]$ and $1\leq l\leq N$, set
\begin{equation*}
\sT_l(\Ft)(\epsilon,t) = \Ft_{l,0} + \int_0^{\epsilon} \mleft( \sum_{j=1}^n \sum_{\substack{\alpha\in \N^N \\ |\alpha|\leq L}} t_j \at_{\alpha,j,l}(\epsilon',t) \Ft(\epsilon',t)^{\alpha} \mright)\: d\epsilon',
\end{equation*}
and set $\sT(\Ft):= (\sT_1(\Ft),\ldots, \sT_N(\Ft))$.  Using that $\ASpace{n}{r'}[\CSpace{[0,1]}]$ is an algebra (\cref{Lemma::FuncSpaceRev::Algebra}), $r'\leq r$, and \cref{Lemma::PfIdentRAEuclid::Integragte},
we have $\sT:\ASpace{n}{r'}[\CSpace{[0,1]}[\R^N]]\rightarrow \ASpace{n}{r'}[\CSpace{[0,1]}[\R^N]]$.  Furthermore, $\Ft$ solves \cref{Eqn::PfIdentRAEuclid::AuxEquation} if and only if
$\sT(\Ft)=\Ft$.

Set
\begin{equation*}
	\sM := \mleft\{ \Ft=(\Ft_1,\ldots, \Ft_N)\in \ASpace{n}{r'}[\CSpace{[0,1]}[\R^N]] : \ANorm{\Ft_l}{n}{r'}[\CSpace{[0,1]}]\leq 2D, 1\leq l\leq N  \mright\}.
\end{equation*}
We give $\sM$ the metric $\rho(\Ft, \Fh):= \max_{1\leq l\leq N} \ANorm{\Ft_l-\Fh_l}{n}{r'}[\CSpace{[0,1]}]$.  With this metric, $\sM$ is a complete metric space.

\begin{lemma}\label{Lemma::PfIdentRAEuclid::Contraction}
$\sT:\sM\rightarrow \sM$ and is a strict contraction.
\end{lemma}
\begin{proof}
For $\Ft\in \sM$, $|\alpha|\leq L$, $1\leq j\leq n$, and $1\leq l\leq N$, we have by \cref{Lemma::FuncSpaceRev::Algebra},
\begin{equation*}
	\ANorm{t_j \at_{\alpha,j,l} \Ft^{\alpha}}{n}{r'}[\CSpace{[0,1]}] \leq \ANorm{t_j}{n}{r'} \ANorm{\at_{\alpha,j,k}}{n}{r'}[\CSpace{[0,1]}] (2D)^{|\alpha|}
	\leq r' 2^{|\alpha|} D^{|\alpha|+1} \leq r' 2^{L} (\max\{1,D\})^{L+1}.
\end{equation*}
Thus, by \cref{Lemma::PfIdentRAEuclid::Integragte},
\begin{equation*}
	\ANorm{\sT_l(\Ft)}{n}{r'}[\CSpace{[0,1]}] \leq |\Ft_{l,0}| + \sum_{j=1}^n \sum_{|\alpha|\leq L} r' 2^L (\max\{1,D\})^{L+1}
	\leq D + r' n 2^L (L+1)^N (\max\{1,D\})^{L+1}\leq 2D,
\end{equation*}
where the last inequality follows from the choice of $r'$.  It follows that $\sT:\sM\rightarrow \sM$.

Again using \cref{Lemma::FuncSpaceRev::Algebra}, we have for $\Ft,\Fh\in \sM$, $|\alpha|\leq L$, $1\leq j\leq n$, and $1\leq l\leq N$,
\begin{equation*}
\begin{split}
	&\ANorm{t_j \at_{\alpha,j,l} (\Ft^{\alpha}-\Fh^{\alpha})}{n}{r'}[\CSpace{[0,1]}]
	\leq \ANorm{t_j}{n}{r'} \ANorm{\at_{\alpha,j,l}}{n}{r'}[\CSpace{[0,1]}] |\alpha| (2D)^{|\alpha|-1} \max_{1\leq k\leq N} \ANorm{\Ft_k-\Fh_k}{n}{r'}[\CSpace{[0,1]}]
	\\& \leq r' 2^{|\alpha|-1} D^{|\alpha|} |\alpha| \rho(\Ft, \Fh)
	\leq r' 2^{L-1} (\max\{1,D\})^{L} L \rho(\Ft,\Fh).
\end{split}
\end{equation*}
Thus, by \cref{Lemma::PfIdentRAEuclid::Integragte},
\begin{equation*}
\begin{split}
&\ANorm{\sT_l(\Ft)-\sT_l(\Fh)}{n}{r'}[\CSpace{[0,1]}]
\leq \sum_{j=1}^n \sum_{|\alpha|\leq L}  r' 2^{L-1} (\max\{1,D\})^{L} L \rho(\Ft,\Fh)
\\&\leq n(L+1)^{N+1} r' 2^{L-1} (\max\{1,D\})^{L} \rho(\Ft,\Fh)
\leq \frac{1}{2} \rho(\Ft,\Fh),
\end{split}
\end{equation*}
where the last inequality follows from the choice of $r'$.  It follows that $\rho(\sT(\Ft), \sT(\Fh))\leq \frac{1}{2} \rho(\Ft,\Fh)$, and therefore $\sT:\sM\rightarrow \sM$ is a strict contraction.
\end{proof}

\begin{proof}[Proof of \cref{Prop::PfIdentRAEuclid::AuxProp}]
\Cref{Lemma::PfIdentRAEuclid::Contraction} shows that the contraction mapping principle applies to $\sT:\sM\rightarrow \sM$ to yield a unique fixed point $\Ft\in \sM$ of $\sT$.
This fixed point is the desired solution to \cref{Eqn::PfIdentRAEuclid::AuxEquation}.  Since $\Ft\in \sM$, \cref{Eqn::PfIndentRAEuclid::AuxBound} follows.
Finally, since \cref{Eqn::PfIdentRAEuclid::AuxEquation} is a standard ODE, standard uniqueness theorems (using, for example, Gr\"onwall's inequality) give the claimed uniqueness.
\end{proof}

\begin{proof}[Proof of \cref{Prop::PfIdentRAEuclid::MainProp}]
Suppose $F\in \CSpace{B^n(r)}[\R^N]$ satisfies \cref{Eqn::PfIdentRAEuclid::MainEquation}.  Set $\Ft(\epsilon,t)=F(\epsilon t)$, $\at_{\alpha,j,l}(\epsilon, t) = a_{\alpha,j,l}(\epsilon t)$,
and $\Ft_{l,0}:= F_{l,0}$.  Then, $\Ft$ satisfies \cref{Eqn::PfIdentRAEuclid::AuxEquation}.  The uniqueness from \cref{Prop::PfIdentRAEuclid::AuxProp} shows that $\Ft$ is the solution
described in that result, and therefore $\Ft(\epsilon, t)\in \ASpace{n}{r'}[\CSpace{[0,1]}[\R^N]]$ and \cref{Eqn::PfIndentRAEuclid::AuxBound} holds.  Since $F(t) = \Ft(1,t)$, the result follows.
\end{proof}

		\subsubsection{Identifying Real Analytic Functions II: Manifolds}
Let $X_1,\ldots, X_q$ be $C^1$ vector fields on a $C^2$ manifold $\fM$.  Fix $x_0\in \fM$ and suppose $X_1,\ldots, X_q$ satisfy $\sC(x_0,\eta,\fM)$ for some $\eta>0$.
Fix $N,L\in \N$, $\xi>0$.  For $1\leq l\leq N$, $1\leq j\leq q$, let $P_{l,j}$ be a polynomial of degree $L$, in $N$ indeterminates, with coefficients in $\AXSpace{X}{x_0}{\eta}\cap \CSpace{B_X(x_0,\xi)}$:
\begin{equation*}
P_{l,j}(x,y) = \sum_{\substack{\alpha\in \N^N \\ |\alpha|\leq L}} b_{\alpha,l,j}(x) y^{\alpha},
\end{equation*}
where $b_{\alpha,l,j}\in \AXSpace{X}{x_0}{\eta}\cap \CSpace{B_X(x_0,\xi)}$.  Fix $D>0$ with $\AXNorm{b_{\alpha,l,j}}{X}{x_0}{\eta}\leq D$, $\forall \alpha,j,l$.

\begin{prop}\label{Prop::IdentRAManifold::MainProp}
Suppose $G=(G_1,\ldots, G_N)\in \CSpace{B_X(x_0,\xi)}[\R^N]$ satisfies, for $1\leq j\leq q$, $1\leq l\leq N$,
\begin{equation*}
	X_j G_l(x) = P_{l,j}(x,G(x)).
\end{equation*}
Then, $\exists r'\in (0,\eta]$ with $G\in \AXSpace{X}{x_0}{r'}[\R^N]$.  Furthermore, $\AXNorm{G}{X}{x_0}{r'}\leq C$ where $C$ and $r'$
can be chosen to depend only on upper bounds for $q$, $\eta^{-1}$, $\xi^{-1}$, $D$, $L$, $N$, and $|G(x_0)|$.
\end{prop}
\begin{proof}
Let $\Psi(t_1,\ldots, t_q):= \exp(t_1 X_1+\cdots + t_q X_q)x_0$, and set $F(t):=G(\Psi(t))$.
The goal is to show $F\in \ASpace{q}{r'}[\R^N]$ with $\ANorm{F}{q}{r'}[\R^N]\leq C$, where $r'$ and $C$
are as in the statement of the result.
Note that $F$ satisfies
\begin{equation*}
	\diff{\epsilon} F_l(\epsilon t) = \sum_{j=1}^q t_j (X_j F_l)(\Psi(\epsilon t))
	=\sum_{j=1}^q \sum_{\substack{\alpha\in \N^N \\ |\alpha|\leq L}} t_j b_{\alpha,l,j}\circ \Phi(\epsilon t) F(\epsilon t)^{\alpha}.
\end{equation*}
By hypothesis, $b_{\alpha,l,j}\circ \Psi\in \ASpace{q}{\eta}$ with $\ANorm{b_{\alpha,l,j}\circ \Psi}{q}{\eta}\leq D$.
From here, the result follows immediately from \cref{Prop::PfIdentRAEuclid::MainProp} (with $n=q$ and $r=\min\{\eta,\xi\}$).
\end{proof}
	
	\subsection{The Quantitative Theorem}\label{Section::Proofs::QuantiativeThm}
In this section, we prove \cref{Thm::QuantRes::MainThm} which is the main theorem of this paper.
We separate the proof into two parts:  when the vector fields are linearly independent at $x_0$ (i.e., when $n=q$)
and more generally when the vector fields may be linearly dependent at $x_0$ (i.e., when $q\geq n$).
	
		\subsubsection{Linearly Independent}\label{Section::PfQaunt::LI}
In this section we prove \cref{Thm::QuantRes::MainThm} in the special case when $n=q$; so that we have $X=X_{J_0}$.
By \cref{Prop::QualRes::InjectiveImmresion}, $X_1,\ldots, X_n$ span the tangent space at every point of $B_X(x_0,\xi)$.
Since we know $B_X(x_0,\xi)$ is an $n$-dimensional manifold, we have that $X_1,\ldots, X_n$ form a basis for the tangent
space at every point of $B_X(x_0,\xi)$.
Taking $\chi=\xi$, \cref{Thm::QuantRes::MainThm} \cref{Item::QuantRes::LI}, \cref{Item::QuantRes::BestBasis}, and \cref{Item::QuantRes::chiSubMfld} follow immediately.

We apply \cref{Prop::PartI::MainProp} to obtain $\etai\in (0,\eta_0]$ as in that proposition; so that $\Phi:B^n(\etai)\rightarrow B_X(x_0,\xi)$ and is a $C^2$ diffeomorphism
onto its image.  We let $Y_j=\Phi^{*} X_j$, and define $A$ as in \cref{Thm::QuantRes::MainThm} by $Y=(I+A) \grad$.
Our main goal is to show that $A$ is real analytic; this will imply that $Y_1,\ldots, Y_n$ are real analytic as well.

We have assumed $c_{j,k}^l\in \AXSpace{X}{x_0}{\eta}$ with $\AXNorm{c_{j,l}^l}{X}{x_0}{\eta}\lesssim 1$.  Thus, by the definition of
$\AXSpace{X}{x_0}{\eta}$, we have $c_{j,k}^l\circ \Phi\in \ASpace{n}{\eta}$ with $\ANorm{c_{j,k}^l\circ \Phi}{n}{\eta}\lesssim 1$.
We conclude $c_{j,k}^l\circ \Phi \in  \ASpace{n}{\etai}$ with $\ANorm{c_{j,k}^l\circ \Phi}{n}{\etai}\lesssim 1$.

For $1\leq l\leq n$, define an $n\times n$ matrix $C_l(t)$ by letting the $j,k$ component of $C_l(t)$ equal $c_{j,l}^k\circ \Phi$.
Thus, $C_l\in \ASpace{n}{\etai}[\M^{n\times n}]$ with $\ANorm{C_l}{n}{\etai}[\M^{n\times n}]\lesssim 1$.
Set $C(t):=\sum_{l=1}^n t_l C_l(t)$.
By \cref{Prop::PartI::MainProp} \cref{Item::Part1::ASize}, $A(0)=0$, and by \cref{Prop::PartI::MainProp} \cref{Item::Part1::ADiffEq},
$A$ satisfies the differential equation \cref{Eqn::Part1::ADiffEq}.

\Cref{Prop::PartODE::MainProp} shows that there is an admissible constant $\eta_1\in (0,\etai]$ such that \cref{Eqn::Part1::ADiffEq} has a solution
$\Ah\in \ASpace{n}{\eta_1}[\M^{n\times n}]$ with $\ANorm{\Ah}{n}{\eta_1}[\M^{n\times n}]\leq \frac{1}{2}$.
By the uniqueness of this solution described in \cref{Prop::PartODE::MainProp}, $A\big|_{B^n(\eta_1)} = \Ah$.
This establishes \cref{Thm::QuantRes::MainThm} \cref{Item::QuantRes::BoundA}.   \Cref{Thm::QuantRes::MainThm} \cref{Item::QuantRes::BoundY}
is an immediate consequence of \cref{Item::QuantRes::BoundA} (since $Y=Y_{J_0}$ when $n=q$).

\Cref{Prop::PartI::MainProp} \cref{Item::Part1::PhiOpen} shows $\Phi(B^n(\etai))$ is an open subset of $B_{X}(x_0,\xi)$
and \cref{Prop::PartI::MainProp} \cref{Item::Part1::PhiDiffeo} shows $\Phi:B^n(\etai)\rightarrow \Phi(B^n(\etai))$ is a $C^2$ diffeomorphism.
Since $\eta_1\leq \etai$, we see $\Phi(B^n(\eta_1))$ is an open subset of $B_{X}(x_0,\xi)$
and $\Phi:B^n(\eta_1)\rightarrow \Phi(B^n(\eta_1))$ is a $C^2$ diffeomorphism.
This establishes \cref{Thm::QuantRes::MainThm} \cref{Item::QuantRes::PhiOpen,Item::QuantRes::PhiDiffeo}.

Finally, we prove \cref{Thm::QuantRes::MainThm} \cref{Item::QuantRes::xi1xi2}.  We have already taken $\chi=\xi$, and we have
$\Phi(B^n(\eta_1))\subseteq B_X(x_0,\xi)$.  Thus it suffices to prove the existence of $\xi_1$ and $\xi_2$.  Since $X_{J_0}=X$,
we may take $\xi_1=\xi_2$, and therefore we only need to prove the existence of $\xi_1$.
This follows just as in \cite[\SSExistXiOne]{StovallStreet}.
		
		\subsubsection{Linearly Dependent}
In this section we prove \cref{Thm::QuantRes::MainThm} in the general case, $q\geq n$.  As in \cite{StovallStreet} the goal is to reduce the problem to the case $q=n$.
Set
\begin{equation*}
	\sI_0(n,q):=\{(i_1,\ldots, i_n)\in \sI(n,q) : 1\leq i_1<i_2<\cdots<i_n\leq q\}.
\end{equation*}

\begin{lemma}\label{Lemma::PfQuant::Liewedge}
For $J\in \sI(n,q)$, $1\leq j\leq n$,
\begin{equation*}
	\Lie{X_j} \bigwedge X_J = \sum_{K\in\sI_0(n,q)} g_{j,J}^K \bigwedge X_K, \text{ on }B_{X_{J_0}}(x_0,\xi),
\end{equation*}
where $\CNorm{g_{j,J}^K}{B_{X_{J_0}}(x_0,\xi)}\lesssim_0 1$ 
and $\AXNorm{g_{j,J}^K}{X_{J_0}}{x_0}{\eta}\lesssim 1$.
Here, $\Lie{X_j}$ denotes the Lie derivative with respect to $X_j$.
\end{lemma}
\begin{proof}
Let $J=(j_1,\ldots, j_n)$.  By the definition of $\Lie{X_j}$ (see \cite[\SSDivideWedge]{StovallStreet} for more details), we have
\begin{equation*}
\Lie{X_j} \bigwedge X_J = \Lie{X_j}\left(X_{j_1}\wedge X_{j_2}\wedge \cdots \wedge X_{j_n}\right) = \sum_{l=1}^n X_{j_1}\wedge X_{j_2}\wedge \cdots \wedge X_{j_{l-1}}\wedge [X_j, X_{j_l}]\wedge X_{j_{l+1}}\wedge \cdots \wedge X_{j_n}.
\end{equation*}
But, $[X_j,X_{j_l}]=\sum_{k=1}^q c_{j,j_l}^k X_k$, by assumption.  Thus,
\begin{equation*}
\Lie{X_j} \bigwedge X_J = \sum_{l=1}^n \sum_{k=1}^q c_{j,j_l}^k X_{j_1}\wedge X_{j_2}\wedge \cdots \wedge X_{j_{l-1}}\wedge X_k \wedge X_{j_{l+1}}\wedge \cdots \wedge X_{j_n}.
\end{equation*}
The result now follows from the anti-commutativity of $\wedge$ and the assumptions on $c_{i,j}^k$.
\end{proof}

Take $\chi\in (0,\xi]$ to be the $0$-admissible constant given by \cref{Prop::Part1::chi}.  With this choice of $\chi$,
\cref{Thm::QuantRes::MainThm} \cref{Item::QuantRes::LI}, \cref{Item::QuantRes::BestBasis}, and \cref{Item::QuantRes::chiSubMfld} follow immediately.
In particular, $\bigwedge X_{J_0}(y)\ne 0$ for $y\in B_X(x_0,\xi)$.  It therefore makes sense to consider
$\frac{\bigwedge X_J(y)}{\bigwedge X_{J_0}(y)}$ for any $J\in \sI(n,q)$ and $y\in B_{X_{J_0}}(x_0,\chi)$.

\begin{lemma}\label{Lemma::PfQuant::DerivOfQuotient}
For $J\in \sI(n,q)$, $1\leq j\leq n$,
\begin{equation*}
	X_j \frac{\bigwedge X_J}{\bigwedge X_{J_0}} = \sum_{K\in \sI_0(n,q)} g_{j,J}^K \frac{\bigwedge X_K}{\bigwedge X_{J_0}} - \sum_{K\in \sI_0(n,q)} g_{j,J_0}^K \frac{\bigwedge X_J}{\bigwedge X_{J_0}}\frac{\bigwedge X_K}{\bigwedge X_{J_0}},
\end{equation*}
where $g_{j,J}^K$ are the functions from \cref{Lemma::PfQuant::Liewedge}.
\end{lemma}
\begin{proof}
We use the identity
\begin{equation*}
	X_j  \frac{\bigwedge X_J}{\bigwedge X_{J_0}} = \frac{\Lie{X_j} \bigwedge X_J}{\bigwedge X_{J_0}} - \frac{\bigwedge X_J}{\bigwedge X_{J_0}} \frac{\Lie{X_j} \bigwedge X_{J_0}}{\bigwedge X_{J_0}},
\end{equation*}
which is proved in \cite[\SSDerivWedge]{StovallStreet}.
From here, the result follows immediately from \cref{Lemma::PfQuant::Liewedge}.
\end{proof}


\begin{lemma}\label{Lemma::PfQuantLd::C1Estimates}
For $J\in \sI(n,q)$ we have $\frac{\bigwedge X_J}{\bigwedge X_{J_0}}\in \CSpace{B_{X_{J_0}}(x_0,\chi)}$ and
\begin{equation*}
	\BCNorm{\frac{\bigwedge X_J}{\bigwedge X_{J_0}}}{B_{X_{J_0}}(x_0,\chi)}\lesssim_0 1.
\end{equation*}
\end{lemma}
\begin{proof}
This is just a restatement of \cref{Thm::QuantRes::MainThm} \cref{Item::QuantRes::BestBasis}, which we have already shown.
\end{proof}

\begin{lemma}\label{Lemma::PfQuantLD::RAEstimates}
There exists an admissible constant $\eta'\in (0,\eta]$ such that $\forall J\in \sI(n,q)$, $\frac{\bigwedge X_J}{\bigwedge X_{J_0}}\in \AXSpace{X_{J_0}}{x_0}{\eta'}$
and
\begin{equation*}
	\BAXNorm{\frac{\bigwedge X_J}{\bigwedge X_{J_0}}}{X_{J_0}}{x_0}{\eta'}\lesssim 1.
\end{equation*}
\end{lemma}
\begin{proof}
Set $N:=|\sI_0(n,q)|$ and
let $G(x) \in \CSpace{B_{X_{J_0}}(x_0,\chi)}[\R^N]$ be given by $G_J(x) = \frac{\bigwedge X_J(x)}{\bigwedge X_{J_0}(x)}$, for $J\in \sI_0(n,q)$.
\Cref{Lemma::PfQuant::DerivOfQuotient} shows that \cref{Prop::IdentRAManifold::MainProp} applies to $G$ (with $\xi=\chi$, $X=X_{J_0}$, $q=n$, $L=2$, $D$ an admissible constant, and $|G(x_0)|\leq \zeta^{-1}$).
Here we taken $\eta'$ to be $r'$ from that result.  The lemma follows.
\end{proof}

\begin{lemma}\label{Lemma::PfQuantLD::DefnBts}
For $1\leq k\leq q$, $1\leq l\leq n$, there exist $\bt_{k}^l \in \CSpace{B_{X_{J_0}}(x_0,\chi)}\cap \AXSpace{X_{J_0}}{x_0}{\eta'}$ such that
\begin{equation}\label{Eqn::PfQuantLD::GiveBts}
	X_k=\sum_{l=1}^n \bt_k^l X_l.
\end{equation}
Furthermore, $\CNorm{\bt_k^l}{B_{X_{J_0}}(x_0,\chi)} \lesssim_0 1$ and $\AXNorm{\bt_k^l}{X_{J_0}}{x_0}{\eta'}\lesssim 1$.
\end{lemma}
\begin{proof}
For $1\leq k\leq q$, $1\leq l\leq n$, set $J(l,k)=(1,2,\ldots, l-1,k,l+1,\ldots, n)\in \sI(n,q)$ and define 
$$\bt_k^l:=\frac{\bigwedge X_{J(l,k)}}{\bigwedge X_{J_0}}.$$
Cramer's rule shows \cref{Eqn::PfQuantLD::GiveBts} (see \cite[\SSDivideWedge]{StovallStreet} for more comments on this application of Cramer's rule).
The desired estimates follow from \cref{Lemma::PfQuantLd::C1Estimates,Lemma::PfQuantLD::RAEstimates}, completing the proof.
\end{proof}

\begin{lemma}\label{Lemma::PfQuantLD::HyposHold}
For $1\leq i,j,k\leq n$, there exist $\ch_{i,j}^k \in \CSpace{B_{X_{J_0}}(x_0,\chi)}\cap \AXSpace{X_{J_0}}{x_0}{\eta'}$ such that
\begin{equation}\label{Eqn::PfQaunt::Defnch}
[X_i,X_j]=\sum_{k=1}^n \ch_{i,j}^k X_k.
\end{equation}
Furthermore, $\CNorm{\ch_{i,j}^k}{B_{X_{J_0}}(x_0,\chi)}\lesssim_0 1$ and $\AXNorm{\ch_{i,j}^k}{X_{J_0}}{x_0}{\eta'}\lesssim 1$.
\end{lemma}
\begin{proof}
For $1\leq i,j\leq n$ we have using \cref{Lemma::PfQuantLD::DefnBts}
\begin{equation*}
[X_i,X_j]=\sum_{k=1}^q c_{i,j}^k X_k = \sum_{l=1}^n \left(\sum_{k=1}^q c_{i,j}^k \bt_k^l\right) X_l.
\end{equation*}
Setting $\ch_{i,j}^l:= \sum_{k=1}^q c_{i,j}^k \bt_k^l$, \cref{Eqn::PfQaunt::Defnch} follows.  
We have
$\CNorm{c_{i,j}^k}{B_{X_{J_0}}(x_0,\xi)}\lesssim_0 1$ and $\AXNorm{c_{i,k}^k}{X_{J_0}}{x_0}{\eta}\lesssim 1$, by the definition of admissible constants.
Combining this with \cref{Lemma::PfQuantLD::DefnBts}, the fact that $\CNorm{fg}{B_{X_{J_0}}(x_0,\xi)}\leq \CNorm{f}{B_{X_{J_0}}(x_0,\xi)}\CNorm{g}{B_{X_{J_0}}(x_0,\xi)} $ (which is immediate from the definition), and $\AXNorm{fg}{X_{J_0}}{x_0}{\eta'}\leq \AXNorm{f}{X_{J_0}}{x_0}{\eta'}\AXNorm{g}{X_{J_0}}{x_0}{\eta'}$ (see \cref{Lemma::FuncSpaceRev::Algebra}), the desired estimates follow.
\end{proof}

In light of \cref{Lemma::PfQuantLD::HyposHold}, the case $n=q$ of \cref{Thm::QuantRes::MainThm} (which we proved in \cref{Section::PfQaunt::LI}) applies to $X_1,\ldots, X_n$ with $\eta$ replaced by $\eta'$
and $\xi$ replaced by $\chi$,\footnote{When $n=q$, we proved \cref{Thm::QuantRes::MainThm} with $\chi=\xi$.} yielding admissible constants $\eta_1\in (0,\eta']$ and $\xi_1>0$ as in that theorem.
This establishes \cref{Item::QuantRes::PhiOpen}, \cref{Item::QuantRes::PhiDiffeo}, \cref{Item::QuantRes::BoundA}, \cref{Item::QuantRes::xi1xi2} (except the existence of $\xi_2$),
and \cref{Item::QuantRes::BoundY} for $1\leq j\leq n$.  Since we have already shown \cref{Item::QuantRes::LI}, \cref{Item::QuantRes::BestBasis}, and \cref{Item::QuantRes::chiSubMfld},
all that remains to show is the existence of $\xi_2$ as in \cref{Item::QuantRes::xi1xi2} and \cref{Item::QuantRes::BoundY}  for $n+1\leq j\leq q$.
The existence of $\xi_2$ follows directly from \cite[\SSExistXiTwo]{StovallStreet}.

All that remains is \cref{Item::QuantRes::BoundY}  for $n+1\leq j\leq q$.  \Cref{Lemma::PfQuantLD::DefnBts} shows $\AXNorm{\bt_k^l}{X_{J_0}}{x_0}{\eta'}\lesssim 1$.
Thus, by the definition of $\AXSpace{X_{J_0}}{x_0}{\eta'}$, we have $\ANorm{\bt_{k}^l\circ \Phi}{n}{\eta_1}\lesssim 1$.  Pulling back \cref{Eqn::PfQuantLD::GiveBts} via $\Phi$, we have
for $n+1\leq k\leq q$,
\begin{equation*}
	Y_k = \sum_{l=1}^n \left(\bt_{k}^l\circ \Phi\right) Y_l.
\end{equation*}
Since we already know $\ANorm{Y_l}{n}{\eta_1}[\R^n]\lesssim 1$ ($1\leq l\leq n$) it follows that $\ANorm{Y_k}{n}{\eta_1}[\R^n]\lesssim 1$.  Here, we have used
that $\ASpace{n}{\eta_1}$ is a Banach algebra (\cref{Lemma::FuncSpaceRev::Algebra}).  This completes the proof.

	\subsection{Densities}\label{Section::Proofs::Densities}
In this section, we prove \cref{Thm::Density::MainThm} and \cref{Cor::Density::MainCor}.  We take the same setting and notation as in that \cref{Section::Results::Desnities}.
In particular, we have all of the conclusions of \cref{Thm::QuantRes::MainThm}, and take $\eta_1$, $\chi$, and $\Phi$ as in that theorem.
In addition, we have a density $\nu$ on $B_{X_{J_0}}(x_0,\chi)$ (where $\chi$ is as in \cref{Thm::QuantRes::MainThm}).
Let $f_j$ be as in \cref{Section::Results::Desnities}.  As in \cref{Thm::QuantRes::MainThm}, we (without loss of generality) take $J_0=(1,\ldots, n)$.

Following \cite{StovallStreet}, we introduce an auxiliary density on $B_{X_{J_0}}(x_0,\chi)$ given by
\begin{equation}\label{Eqn::PfDensities::Defnnu0}
\nu_0(Z_1,\ldots, Z_n):=\left|\frac{Z_1\wedge Z_2\wedge \cdots \wedge Z_n}{X_1\wedge X_2\wedge \cdots \wedge X_n}\right|.
\end{equation}
Since $X_1\wedge X_2\wedge \cdots \wedge X_n$ is never zero on $B_{X_{J_0}}(x_0,\chi)$ (\cref{Thm::QuantRes::MainThm} \cref{Item::QuantRes::LI}),
$\nu_0$ is defined on $B_{X_{J_0}}(x_0,\chi)$ and is clearly a density.
Note that $\nu_0(X_1,\ldots, X_n)\equiv 1$, so that $\nu_0$ is nonzero everywhere on $B_{X_{J_0}}(x_0,\chi)$.


\begin{lemma}\label{Lemma::PfDensities::fj0RA}
For $1\leq j\leq n$, $\Lie{X_j} \nu_0 =f_j^{0} \nu_0$, where $f_j^0\in \AXSpace{X_{J_0}}{x_0}{\eta_1}\cap \CSpace{B_{X_{J_0}}(x_0,\chi)}$ and $\AXNorm{f_j^0}{X_{J_0}}{x_0}{\eta_1}\lesssim 1$,
$\CNorm{f_j^0}{B_{X_{J_0}}(x_0,\chi)}\lesssim_0 1$.
\end{lemma}
\begin{proof}
The function $f_j^0$ was computed in the proof of \cite[\SSComputefjzero]{StovallStreet}.  There it is shown
$$f_j^0=-\sum_{K\in \sI_0(n,q)} g_{j,J_0}^K \frac{\bigwedge X_K}{\bigwedge X_{J_0}},$$
where $g_{j,J_0}^K$ is the function from \cref{Lemma::PfQuant::Liewedge}.
We have by \cref{Lemma::PfQuant::Liewedge},
\begin{equation*}
\AXNorm{g_{j,J_0}^K}{X_{J_0}}{x_0}{\eta_1}\leq \AXNorm{g_{j,J_0}^K}{X_{J_0}}{x_0}{\eta}\lesssim 1, \quad \CNorm{g_{j,J_0}^K}{B_{X_{J_0}}(x_0,\chi)}\leq \CNorm{g_{j,J_0}^K}{B_{X_{J_0}}(x_0,\xi)}\lesssim_0 1.
\end{equation*}
Also, by \cref{Lemma::PfQuantLd::C1Estimates} and \cref{Lemma::PfQuantLD::RAEstimates},
\begin{equation*}
\BAXNorm{\frac{\bigwedge X_K}{\bigwedge X_{J_0}}}{X_{J_0}}{x_0}{\eta_1} \leq \BAXNorm{\frac{\bigwedge X_K}{\bigwedge X_{J_0}}}{X_{J_0}}{x_0}{\eta'} \lesssim 1, \quad \BCNorm{\frac{\bigwedge X_J}{\bigwedge X_{J_0}}}{B_{X_{J_0}}(x_0,\chi)}\lesssim 1.
\end{equation*}
where we have used that $\eta'$ from the proof of \cref{Thm::Density::MainThm} satisfies $\eta'\geq \eta_1$.
Using the above,
the result follows from \cref{Lemma::FuncSpaceRev::Algebra}.
\end{proof}

Define $h_0\in \CSpace{B^n(\eta_1)}$ by $\Phi^{*}\nu_0 = h_0\LebDensity$, where $\LebDensity$ denotes the Lebesgue density on $\R^n$.

\begin{lemma}\label{Lemma::PfDesnities::hRA}
$h_0(t)=\det(I+A(t))^{-1}$ where $A$ is the matrix from \cref{Thm::Density::MainThm}.  Furthermore, 
$h_0(t)\approx_0 1$, $\forall t\in B^n(\eta_1)$, and $h_0\in \ASpace{n}{\eta_1}$ with $\ANorm{h_0}{n}{\eta_1}\lesssim_0 1$.
\end{lemma}
\begin{proof}
Because $\sup_{t\in B^n(\eta_1)} \Norm{A(t)}[\M^{n\times n}]\leq \ANorm{A}{n}{\eta_1}[\M^{n\times n}]\leq \frac{1}{2}$ 
(by \cref{Thm::QuantRes::MainThm} \cref{Item::QuantRes::BoundA}),
we have $|\det(I+A(t))^{-1}|=\det(I+A(t))^{-1}$, for all $t\in B^n(\eta_1)$.  Thus,
\begin{equation*}
\begin{split}
	&h_0(t) = (\Phi^{*}\nu_0)(t)\left(\diff{t_1},\diff{t_2},\ldots, \diff{t_n}\right) = (\Phi^{*}\nu_0)(t) \left((I+A(t))^{-1} Y_1(t),\ldots, (I+A(t))^{-1} Y_n(t)\right)
	\\&=|\det (I+A(t))^{-1}| (\Phi^{*}\nu_0)(t)(Y_1(t),\ldots, Y_n(t)) = \det(I+A(t))^{-1} \nu_0(\Phi(t))(X_1(\Phi(t)),\ldots, X_n(\Phi(t)))
	\\&=\det(I+A(t))^{-1}.
\end{split}
\end{equation*}
That $h_0(t)\approx_0 1$, $\forall t\in B^n(\eta_1)$, now follows from the above mentioned fact that $\sup_{t\in B^n(\eta_1)}\Norm{A(t)}[\M^{n\times n}]\leq \frac{1}{2}$.

Since $\ASpace{n}{\eta_1}[\M^{n\times n}]$ is a Banach algebra (\cref{Lemma::FuncSpaceRev::Algebra}) and since
$\ANorm{A}{n}{\eta_1}[\M^{n\times n}]\leq \frac{1}{2}$, it follows that $I+A$ is invertible in $\ASpace{n}{\eta_1}[\M^{n\times n}]$
with $\ANorm{(I+A)^{-1}}{n}{\eta_1}[\M^{n\times n}]\leq 2$.  We conclude $\ANorm{h_0}{n}{\eta_1}=\ANorm{\det (I+A)^{-1}}{n}{\eta_1}\lesssim_0 1$.
\end{proof}

Since $\nu_0$ is an everywhere nonzero density, there is a unique $g\in \CSpace{B_{X_{J_0}}(x_0,\chi)}$ such that
$\nu_0=g\nu$.  

\begin{lemma}\label{Lemma::PfDensity::Derivg}
For $1\leq j\leq n$, $X_j g= (f_j- f_j^{0})g$, and $g(x)\approx_{0;\nu}g(x_0)=\nu(x_0)(X_1(x_0),\ldots, X_n(x_0))$ for all $x\in B_{X_{J_0}}(x_0,\chi)$.
\end{lemma}
\begin{proof} See \cite[\SSSectionDensities]{StovallStreet}.\end{proof}

\begin{lemma}\label{Lemma::PfDensities::gRA}
Set $s:=\min \{\eta_1,r\}$.
Then, 
$g\in \AXSpace{X_{J_0}}{x_0}{s}$ and
$\AXNorm{g}{X_{J_0}}{x_0}{s}\lesssim_{\nu} |\nu(X_1,\ldots, X_n)(x_0)|$.
\end{lemma}
\begin{proof}
Set 
$B(t):=g\circ\Phi(t)$.
The result can be rephrased as saying $B\in \ASpace{n}{s}$ with $\ANorm{B}{n}{s}\lesssim_{\nu} |B(0)|=|g(x_0)|=|\nu(X_1,\ldots, X_n)(x_0)|$.
We have, using \cref{Lemma::PfDensity::Derivg},
\begin{equation}\label{Eqn::PfDesnities::LinearODE}
\begin{split}
	&\diff{\epsilon} B(\epsilon t) = ((t_1X_1+\cdots+t_nX_n)g)(\Phi(\epsilon t)) = \sum_{j=1}^n t_j (f_j\circ \Phi(\epsilon t)-f_j^0\circ\Phi(\epsilon t)) g\circ\Phi(\epsilon t) 
	\\&= 
	\sum_{j=1}^n t_j (f_j\circ \Phi(\epsilon t)-f_j^0\circ\Phi(\epsilon t)) B(\epsilon t).
\end{split}
\end{equation}
Solving the linear ODE \cref{Eqn::PfDesnities::LinearODE} we have
\begin{equation*}
B(t) = e^{\int_0^1 \sum_{j=1}^n t_j(f_j\circ\Phi(\epsilon t)-f_j^0\circ\Phi(\epsilon t))\: d\epsilon} B(0).
\end{equation*}
Since $\AXNorm{f_j}{X_{J_0}}{x_0}{r}\lesssim_{\nu} 1$, by assumption, and $\AXNorm{f_j^0}{X_{J_0}}{x_0}{\eta_1}\lesssim_\nu 1$ by \cref{Lemma::PfDensities::fj0RA},
we have $\ANorm{f_j\circ \Phi - f_j^0\circ\Phi}{n}{s}\lesssim_{\nu} 1$. 
By \cref{Lemma::PfIdentRAEuclid::Integragte},
if $F(t) = \int_0^1 \sum_{j=1}^n t_j (f_j\circ\Phi(\epsilon t)-f_j^0\circ\Phi(\epsilon t))\: d\epsilon$, then $\ANorm{F}{n}{s}\lesssim_{\nu} 1$.
Finally, since $\ASpace{n}{s}$ is a Banach algebra (by \cref{Lemma::FuncSpaceRev::Algebra}), $\ANorm{e^{F}}{n}{s}\leq e^{\ANorm{F}{n}{s}}\lesssim_{\nu} 1$.
We conclude $\ANorm{B}{n}{s}\lesssim_{\nu} |B(0)|$, completing the proof.
\end{proof}

\begin{proof}[Proof of \cref{Thm::Density::MainThm}]
We have
\begin{equation*}
h\LebDensity = \Phi^{*} \nu = \Phi^{*} g\nu_0 = (g\circ \Phi) \Phi^{*} \nu_0= (g\circ\Phi) h_0 \LebDensity,
\end{equation*}
and therefore $h=(g\circ \Phi) h_0$.  \Cref{Item::Density::HConst} follows by combining the fact that $g\approx_{0;\nu} \nu(X_1,\ldots, X_n)(x_0)$ (\cref{Lemma::PfDensity::Derivg}) and
$h_0\approx 1$ (\cref{Lemma::PfDesnities::hRA}).

Since $\ANorm{g\circ\Phi}{n}{s}=\AXNorm{g}{X_{J_0}}{x_0}{s}\lesssim_{\nu} |\nu(X_1,\ldots, X_n)(x_0)|$ (\cref{Lemma::PfDensities::gRA}) and
$\ANorm{h_0}{n}{s}\leq \ANorm{h_0}{n}{\eta_1}\lesssim_0 1$ (\cref{Lemma::PfDesnities::hRA}), \cref{Item::Density::HRA} follows
by the formula $h=(g\circ\Phi) h_0$ and \cref{Lemma::FuncSpaceRev::Algebra}.
\end{proof}

Having proved \cref{Thm::Density::MainThm} we turn to \cref{Cor::Density::MainCor}.  To facilitate this, we introduce a corollary of \cref{Thm::QuantRes::MainThm}.
\begin{cor}\label{Cor::PfDesnity::MakeTheBalls}
Let $\eta_1, \xi_1, \xi_2>0$ be as in \cref{Thm::QuantRes::MainThm}.  Then, there exist admissible constants $\eta_2\in (0,\eta_1]$, $0<\xi_4\leq \xi_3\leq \xi_2$ such that
\begin{equation*}
\begin{split}
&B(x_0,\xi_4)\subseteq B_{X_{J_0}}(x_0,\xi_3)\subseteq \Phi(B^n(\eta_2))\subseteq B_{X_{J_0}}(x_0,\xi_2)\subseteq B_X(x_0,\xi_2)\\
&\subseteq B_{X_{J_0}}(x_0,\xi_1)\subseteq \Phi(B^n(\eta_2))\subseteq B_{X_{J_0}}(x_0,\chi)\subseteq B_X(x_0,\xi).
\end{split}
\end{equation*}
\end{cor}
\begin{proof}
After obtaining $\eta_1,\xi_1,\xi_2$ from \cref{Thm::QuantRes::MainThm}, apply \cref{Thm::QuantRes::MainThm} again with $\xi$ replaced by $\xi_2$ to obtain $\eta_2$, $\xi_3$,
and $\xi_4$ as in the statement of the corollary.
\end{proof}

\begin{proof}[Proof of \cref{Cor::Density::MainCor}]
We have
\begin{equation}\label{Eqn::PfDensities::Integral}
\begin{split}
&\nu(B_{X_{J_0}}(x_0,\xi_2)) = \int_{B_{X_{J_0}}(x_0,\xi_2)} \nu = \int_{\Phi^{-1}(B_{X_{J_0}}(x_0,\xi_2))} \Phi^{*} \nu 
\\&=\int_{\Phi^{-1}(B_{X_{J_0}}(x_0,\xi_2))} h(t)\: dt \approx_{0;\nu} \LebDensity(\Phi^{-1}(B_{X_{J_0}}(x_0,\xi_2))) \nu(X_1,\ldots,X_n)(x_0),
\end{split}
\end{equation}
where $\LebDensity$ denotes Lebesgue measure, and we have used \cref{Thm::Density::MainThm} \cref{Item::Density::HConst}.
By \cref{Cor::PfDesnity::MakeTheBalls} and the fact that $\eta_1,\eta_2>1$ are admissible constants, we have
\begin{equation}\label{Eqn::PfDesnity::EstimateLebBalls}
1\approx \LebDensity(B^n(\eta_2)) \leq \LebDensity(\Phi^{-1}(B_{X_{J_0}}(x_0,\xi_2))\leq \LebDensity(B^n(\eta_1))\approx 1.
\end{equation}
Combining \cref{Eqn::PfDensities::Integral} and \cref{Eqn::PfDesnity::EstimateLebBalls} proves $\nu(B_{X_{J_0}}(x_0,\xi_2))\approx_{\nu} \nu(X_1,\ldots, X_n)(x_0)$.
The same proof works with $B_{X_{J_0}}(x_0,\xi_2)$ replaced by $B_X(x_0,\xi_2)$, which completes the proof of \cref{Eqn::Density::MainCor::NoAbs}.

To prove \cref{Eqn::Density::MainCor::Abs}, all that remains is to show
\begin{equation}\label{Eqn::PfDensity::FinalToShow}
 |\nu(X_1,\ldots, X_n)(x_0)|\approx_0 \max_{j_1,\ldots, j_n\in \{1,\ldots, q\}} |\nu(X_{j_1},\ldots, X_{j_n})(x_0)|
\end{equation}
Note that
\begin{equation*}
1=|\nu_0(X_1,\ldots, X_n)(x_0)|\leq \max_{k_1,\ldots, k_n\in \{1,\ldots, q\}} |\nu_0(X_{k_1},\ldots, X_{k_n})(x_0)|\leq \zeta^{-1}\lesssim _01,
\end{equation*}
by the definition of $\zeta$.  
We conclude
\begin{equation*}
\max_{k_1,\ldots, k_n\in \{1,\ldots, q\}} |\nu_0(X_{k_1},\ldots, X_{k_n})(x_0)|\approx_0 1.
\end{equation*}
Thus,
\begin{equation*}
\begin{split}
&|\nu(X_1,\ldots, X_n)(x_0)| = |g(x_0) \nu_0(X_1,\ldots, X_n)(x_0)| = |g(x_0)| 
\\&\approx_0 |g(x_0)| \max_{k_1,\ldots, k_n\in \{1,\ldots, q\}} |\nu_0(X_{k_1},\ldots, X_{k_n})(x_0)|
= \max_{k_1,\ldots, k_n\in \{1,\ldots, q\}} |\nu(X_{k_1},\ldots, X_{k_n})(x_0)|
\end{split}
\end{equation*}
This establishes \cref{Eqn::PfDensity::FinalToShow} and completes the proof.
\end{proof}

	\subsection{Qualitative Results}\label{Section::Proofs::Qual}
In this section, we prove the qualitative results; i.e, \cref{Thm::QualRes::LocalThm,Thm::QualRes::GlobalThm}.
These are simple consequences of \cref{Thm::QuantRes::MainThm}.
We begin with \cref{Thm::QualRes::LocalThm}.  For this we recall \cite[\SSLemmaMoreOnAssump]{StovallStreet}.

\begin{lemma}[\SSLemmaMoreOnAssump{} of \cite{StovallStreet}]\label{Lemma::PfQual::Existsetadelta}
Let $X_1,\ldots, X_q$ be $C^1$ vector fields on a $C^2$ manifold $\fM$.
\begin{enumerate}[(i)]
\item\label{Item::PfQual::existeta} Let $\Compact\Subset \fM$ be a compact set.  Then $\exists \eta>0$ such that $\forall x_0\in \Compact$, $X_1,\ldots, X_n$ satisfy $\sC(x_0,\eta,\fM)$.
%
\item\label{Item::PfQual::existdelta0} Let $\Compact\Subset \fM$ be a compact set.  Then, there exists $\delta_0>0$ such that $\forall \theta\in S^{q-1}$ if $x\in \Compact$
is such that $\theta_1 X_1(x)+\cdots+\theta_qX_q(x)\ne 0$, then $\forall r\in (0,\delta_0]$,
\begin{equation*}
	e^{r\theta_1 X_1+\cdots+ r\theta_qX_q}x\ne x.
\end{equation*}
\end{enumerate}
\end{lemma}
\begin{proof}[Comments on the proof]
In  \cite[\SSLemmaMoreOnAssump]{StovallStreet}, \cref{Item::PfQual::existeta} was only stated for a fixed $x_0\in \fM$ and not ``uniformly on compact sets.''  However, the same proof yields \cref{Item::PfQual::existeta}; indeed,
it is an immediate consequence of the Picard-Lindel\"of theorem.  \Cref{Item::PfQual::existdelta0} is stated directly in \cite[\SSLemmaMoreOnAssump]{StovallStreet}.
\end{proof}

\begin{rmk}\label{Rmk::PfQual::Existsetadelta}
\Cref{Lemma::PfQual::Existsetadelta} shows that we always have $\eta$ and $\delta_0$ as in the assumptions of \cref{Thm::QuantRes::MainThm}.
Thus, if we wish to apply \cref{Thm::QuantRes::MainThm} to obtain a qualitative result, we do not need to verify the existence of $\eta$ and $\delta_0$.
\end{rmk}

\begin{lemma}\label{Lemma::PfQual::ShowCommRa}
Let $X_1,\ldots, X_q$ be $C^1$ vector fields on an $n$-dimensional $C^2$ manifold $M$. 
Let $V\subseteq M$ be open, let $U\subseteq \R^n$ be an open neighborhood of $0\in \R^n$, and let $\Phi:U\rightarrow V$ be a $C^2$-diffeomorphism.
Fix $r>0$ so that $B^n(r)\subseteq U$ and set $x_0=\Phi(0)$.  Set $Y_j=\Phi^{*}X_j$ and suppose $Y_j\in \ASpace{n}{r}[\R^n]$, $1\leq j\leq q$,
and for some $j_0,k_0$ $[Y_{j_0},Y_{k_0}]=\sum_{l=1}^q \ct_{j_0,k_0}^l Y_l$, where $\ct_{j_0,k_0}^l\in \ASpace{n}{r}$.
Then, there exists a neighborhood $V'$ of $x_0$ and $s>0$ such that $[X_{j_0},X_{k_0}]=\sum_{l=1}^q c_{j_0,k_0}^l X_l$ where
$c_{j_0,k_0}^l\in \CXomegaSpace{X}{s}[V']$ and $c_{j_0,k_0}^l=\ct_{j_0,k_0}^l\circ \Phi^{-1}$.
\end{lemma}
\begin{proof}
Since $\ct_{j_0,k_0}^l\in \ASpace{n}{r}$ and $Y_j\in \ASpace{n}{r}[\R^n]$ ($1\leq j\leq q$), \cref{Lemma::FiuncSpaceRev::Euclid::Compare} \cref{Item::FuncSpaceRev::Euclid::CBigger} shows
$\ct_{j_0,k_0}^l\in \ComegaSpace{r/2}[B^n(r/2)]$ and $Y_j\in \ComegaSpace{r/2}[B^n(r/2)][\R^n]$.  
\Cref{Prop::NelsonTheorem2}
shows that there exists $s>0$
with $\ct_{j_0,k_0}^l\in \CXomegaSpace{Y}{s}[B^n(r/2)]$, where $Y$ denotes the list of vector fields $Y_1,\ldots, Y_q$.
Set $c_{j_0,k_0}^l:=\ct_{j_0,k_0}^l\circ \Phi^{-1}$ and $V':=\Phi(B^n(r/2))$.  \Cref{Prop::FuncMfld::DiffeoInv} shows
$c_{j_0,k_0}^l\in \CXomegaSpace{X}{s}[V']$ and pushing the formula $[Y_{j_0},Y_{k_0}]=\sum_{l=1}^q \ct_{j_0,k_0}^l Y_l$ forward via $\Phi$
shows $[X_{j_0},X_{k_0}]=\sum_{l=1}^q c_{j_0,k_0}^l X_l$.  This completes the proof.
\end{proof}

\begin{proof}[Proof of \cref{Thm::QualRes::LocalThm}]
\Cref{Item::QualRes::Local::Coord}$\Rightarrow$\cref{Item::QualRes::Local::Basis}:  Let $U$, $V$, $\Phi$, and $x_0$ be as in \cref{Item::QualRes::Local::Coord}.  
By the definition of $\CjSpace{\omega}[U][\R^n]$, there exists an $r>0$ such that $\Phi^{*}X_1,\ldots, \Phi^{*}X_q\in \ComegaSpace{r}[U][\R^n]$.
Without loss of
generality, assume $0\in U$ and $\Phi(0)=x_0$.  Reorder $X_1,\ldots, X_q$ so that $X_1(x_0),\ldots, X_n(x_0)$ are linearly independent and let $Y_j:=\Phi^{*} X_j$,
so that $Y_j\in \ComegaSpace{r}[U][\R^n]$. 
  Note that $Y_1(0),\ldots, Y_n(0)$ span the tangent space $T_0U$.
   Take $r_1\in (0, r]$ so small $B^n(r_1)\subseteq U$. 
   By \cref{Lemma::FiuncSpaceRev::Euclid::Compare} \cref{Item::FuncSpaceRev::Euclid::ABigger}, $Y_j\in \ASpace{n}{r_1}(\R^n)$, $1\leq j\leq q$.
   Since $Y_1,\ldots, Y_n$ are real analytic and form a basis for the tangent space 
   near $0$,
   there exists $r_2>0$ such that
   \begin{equation*}
   	[Y_i,Y_j]=\sum_{k=1}^n \ct_{i,j}^k Y_k,\quad \ct_{i,j}^k\in \ASpace{n}{r_2}.
   \end{equation*}
   Pushing this statement forward via $\Phi$ shows, for $1\leq i,j\leq n$, $[X_i,X_j]=\sum_{k=1}^n \ch_{i,j}^k X_k$, where $\ch_{i,j}^k:=\ct_{i,j}^k\circ \Phi^{-1}$.
   \Cref{Lemma::PfQual::ShowCommRa} shows there exists $s_1>0$ with $\ch_{i,j}^k\in \CXomegaSpace{X}{s_1}[V']\subseteq \CXjSpace{X}{\omega}[V']$ for some neighborhood $V'$ of $x_0$.
   
   Furthermore, since each $Y_j$ ($1\leq j\leq q$) is real analytic, and $Y_1,\ldots, Y_n$ form a basis of the tangent space near $0$, there is $s_2>0$
   such that, for $n+1\leq j\leq q$,
   \begin{equation}\label{Eqn::QualRes::1to2::Commutator}
   	Y_j=\sum_{k=1}^n \bt_j^k Y_k, \quad \bt_{j}^k\in \ASpace{n}{s_2}.
   \end{equation}
   By \cref{Lemma::FiuncSpaceRev::Euclid::Compare} \cref{Item::FuncSpaceRev::Euclid::CBigger},
   $\bt_{j}^k\in \ComegaSpace{s_2/2}[B^n(s_2/2)]$.
   \Cref{Prop::NelsonTheorem2}
   shows that there exists $s_3>0$ such that $\bt_{j}^k\in \CXomegaSpace{Y}{s_3}[B^n(s_2/2)]$.
   \Cref{Prop::FuncMfld::DiffeoInv} shows $b_{j}^k:=\bt_j^k\circ \Phi^{-1}\in \CXomegaSpace{X}{s_3}[\Phi(B^n(s_2/2))]\subseteq \CXjSpace{X}{\omega}[\Phi(B^n(s_2/2)]$.
   Pushing \cref{Eqn::QualRes::1to2::Commutator} forward via $\Phi$, we have $X_j=\sum_{k=1}^n b_j^k X_k$.
   Combining the above proves \cref{Item::QualRes::Local::Basis}.
   
   \Cref{Item::QualRes::Local::Basis}$\Rightarrow$\cref{Item::QualRes::Local::Commute}:  Suppose \cref{Item::QualRes::Local::Basis} holds.
   Let $V$ be as in \cref{Item::QualRes::Local::Basis} and take $r>0$ so that $\ch_{i,j}^k, b_j^k\in \CXomegaSpace{X}{r}[V]$, $\forall i,j,k$.
     We wish to show for $1\leq i,j\leq q$,
   \begin{equation}\label{Eqn::QualRes::ToShow::Commute}
   	[X_i,X_j]=\sum_{k=1}^q c_{i,j}^k X_k, \quad c_{i,j}^k \in \CXomegaSpace{X}{r/2}[V].
   \end{equation}
   For $1\leq i,j\leq n$, \cref{Eqn::QualRes::ToShow::Commute} follows directly from the hypothesis \cref{Item::QualRes::Local::Basis}.  We prove \cref{Eqn::QualRes::ToShow::Commute} for $n+1\leq i,j\leq q$.
   The remaining cases ($1\leq i\leq n$ and $n+1\leq j\leq q$, or $n+1\leq i\leq q$ and $1\leq j\leq n$) are similar and easier.
   We have
   \begin{equation*}
   	[X_i,X_j]=\left[\sum_{k_1=1}^n b_i^{k_1} X_{k_1}, \sum_{k_2=1}^n b_j^{k_2} X_{k_2}\right] =\sum_{k_1,k_2=1}^n \left(b_i^{k_1}(X_{k_1} b_j^{k_2}) X_{k_2} - b_{j}^{k_2} (X_{k_2} b_{i}^{k_1}) X_{k_1} + \sum_{l=1}^n b_{i}^{k_1}b_j^{k_2} \ch_{k_1,k_2}^l X_l \right).
   \end{equation*}
   We are given $b_j^k,\ch_{k_1,j_2}\in \CXomegaSpace{X}{r}[V]\subseteq \CXomegaSpace{X}{r/2}[V]$ and \cref{Lemma::FuncSpaceRev::DerivComega} shows $X_l b_j^k \in \CXomegaSpace{X}{r/2}[V]$.
   Now the result follows from the fact that $\CXomegaSpace{X}{r/2}[V]$ is a Banach algebra (\cref{Lemma::FuncSpaceRev::Algebra}).
   
   \Cref{Item::QualRes::Local::Commute}$\Rightarrow$\cref{Item::QualRes::Local::Coord}:  This is a consequence of \cref{Thm::QuantRes::MainThm}.  We make a few comments to this end.
   First of all, as discussed in \cref{Lemma::PfQual::Existsetadelta,Rmk::PfQual::Existsetadelta}, there exist $\eta$ and $\delta_0$ as in the hypotheses of  \cref{Thm::QuantRes::MainThm}.
   In \cref{Item::QualRes::Local::Commute} we are given $c_{i,j}^k\in \CXomegaSpace{X}{r}$ near $x_0$ for some $r>0$, while in \cref{Thm::QuantRes::MainThm} it is assumed
   $c_{j,k}^l$ is continuous near $x_0$ and $c_{j,k}^k \in \AXSpace{X_{J_0}}{x_0}{\eta}$ (where $J_0$ is chosen as in that theorem with $\zeta=1$).  However,  
   $c_{i,j}^k\in \CXomegaSpace{X}{r}$ near $x_0$ clearly implies $c_{j,k}^l$ is continuous near $x_0$,
   and \cref{Lemma::FuncSpaceRev::Mfld::ABigger} shows
   $\CXomegaSpace{X}{\eta}\subseteq \AXSpace{X}{x_0}{\eta}\subseteq  \AXSpace{X_{J_0}}{x_0}{\eta}$, so by shrinking $\eta$ to be $\leq r$, those hypotheses follow.
   With these remarks, \cref{Thm::QuantRes::MainThm} applies to yield the coordinate chart $\Phi$ as in that theorem.  This coordinate chart
   has the properties given in \cref{Item::QualRes::Local::Coord}, completing the proof.
\end{proof}

We now turn to \cref{Thm::QualRes::GlobalThm}.  The uniqueness of the real analytic structure described in that theorem follows from the next lemma.
\begin{lemma}\label{Lemma::PfQual::ShowDiffeoRA}
Let $M,N$ be two $n$-dimensional real analytic manifolds and suppose $X_1,\ldots, X_q$ are real analytic vector fields on $M$ which span the tangent space
at every point, and $Z_1,\ldots, Z_q$ are real analytic vector fields on $N$.  Let $\Psi:M\rightarrow N$ be a $C^2$ diffeomorphism such that
$\Psi_{*} X_j=Z_j$.  Then $\Psi$ is real analytic.
\end{lemma}
\begin{proof}
Fix a point $x_0\in M$.  We will show $\Psi$ is real analytic near $x_0$.  Reorder the vector fields $X_1,\ldots, X_q$ so that
$X_1(x_0),\ldots, X_n(x_0)$ are linearly independent; and therefore form a basis for the tangent space near $x_0$.
Reorder $Z_1,\ldots, Z_q$ in the corresponding way, so that we have $\Psi_{*} X_j=Z_j$.  I.e., we have
\begin{equation*}
	d\Psi(x)X_{j}(x)=Z_j(\Psi(x)).
\end{equation*}
Pick a real analytic coordinate system, $x_1,\ldots, x_n$, on $M$ near $x_0$.  Since $X_1,\ldots, X_n$ span the tangent space near $x_0$ and are real analytic, and $Z_1,\ldots, Z_n$ are real analytic, we have, for $x$ near $x_0$,
\begin{equation*}
	\frac{\partial \Psi}{\partial x_j} (x) = \sum_{l=1}^n a_{j,l}(x) F_{j,l}(\Psi(x)),\quad 1\leq j\leq n,
\end{equation*}
where $a_{j,l}$ and $F_{j,l}$ are real analytic near $x_0$ and $\Psi(x_0$), respectively.
\Cref{Prop::PfRealAl::IdentifyRA::Basic} applies to show $\Psi$ is real analytic near $x_0$, completing the proof.
\end{proof}

\begin{proof}[Proof of \cref{Thm::QualRes::GlobalThm}]
\Cref{Item::QualRes::Global::Atlas}$\Rightarrow$\cref{Item::QualRes::Global::Conds} is obvious.  For the converse, suppose \cref{Item::QualRes::Global::Conds} holds.
Thus, for each $x\in M$, there exist open sets $U_x\subseteq \R^n$, $V_x\subseteq M$, and a $C^2$ diffeomorphism $\Phi_x:U_x\rightarrow V_x$
such that if $Y_j^x=\Phi_x^{*} X_j$, then $Y_j^x$ is real analytic on $U_x$.  We wish to show that the colletion
$\{(\Phi_x^{-1}, V_x): x\in M\}$ forms a real analytic atlas on $M$; once this is shown, \cref{Item::QualRes::Global::Atlas} will follow
since then $X_j$ will be real analytic with respect to this atlas by definition, and this atlas is clearly 
compatible with the
$C^2$ structure on $M$.  Hence, we need only verify that the transition functions are real analytic.
Take $x_1,x_2\in M$ such that $V_{x_1}\cap V_{x_2}\ne \emptyset$.  Set $\Psi=\Phi_{x_2}^{-1}\circ \Phi_{x_1}:U_{x_1}\cap \Phi_{x_1}^{-1}(V_{x_2})\rightarrow U_{x_2}\cap \Phi_{x_2}^{-1}(V_{x_1})$.
We wish to show $\Psi$ is a real analytic diffeomorphism.  We already know $\Psi$ is a $C^2$ diffeomorphism and $\Psi_{*}Y_j^{x_1}=Y_j^{x_2}$.
That $\Psi$ is real analytic now follows from \cref{Lemma::PfQual::ShowDiffeoRA}, completing the proof of \cref{Item::QualRes::Global::Atlas}.
As mentioned before, the uniqueness of the real analytic structure, as described in the theorem, follows from \cref{Lemma::PfQual::ShowDiffeoRA}, completing the proof.
\end{proof}

\section{Densities in Euclidean Space}\label{Section::DensitiesInEuclidean}
While the hypotheses in \cref{Section::Results::Desnities} concern densities on abstract manifolds, the most important special case which arises in applications is that of the induced Lebesgue density on real analytic
submanifolds of Euclidean space.  In this section, we describe how to apply \cref{Thm::Density::MainThm,Cor::Density::MainCor} in such a setting.

Let $\Omega\subseteq \R^N$ be open and fix $x_0\in \Omega$.  Let $X_1,\ldots, X_q\in \CjSpaceloc{\omega}[\Omega][\R^N]$ be real analytic vector fields on $\Omega$.
We suppose $X_1,\ldots, X_q$ satisfy the conditions of \cref{Thm::QuantRes::MainThm} with $\fM=\Omega$, and so we have admissible constants as in that theorem,
and $\xi$, $\delta_0$, $\eta$, $J_0$, and $\zeta$ as in the hypotheses of  \cref{Thm::QuantRes::MainThm}, and we take $\chi$ as in the conclusion of  \cref{Thm::QuantRes::MainThm}.\footnote{There is a sense in which \cref{Thm::QuantRes::MainThm} can always be applied to real analytic vector fields.
This is the subject of \cref{Section::Scaling::BeyondHormander} and \cref{Section::ScalingRevis}.  However, this section has a different thrust and so we assume the hypotheses of  \cref{Thm::QuantRes::MainThm}.}
As in \cref{Thm::QuantRes::MainThm}, we take $n:=\dim \Span\{X_1(x_0),\ldots, X_q(x_0)\}$.  We also assume that $\xi$ is chosen so that $B_{X}(x_0,\xi)\Subset \Omega$.
Fix an open set $\Omega'$ with $B_X(x_0,\xi)\Subset \Omega'\Subset \Omega$.

Under these hypotheses, \cref{Prop::QualRes::InjectiveImmresion} applies to show that $B_X(x_0,\xi)$ is an $n$-dimensional, injectively immersed
submanifold of $\Omega$.  Classical theorems show that $B_X(x_0,\xi)$  can be given the structure of a real analytic,\footnote{It is not important for the results in this section that $B_X(x_0,\xi)$ can be given a real analytic structure.} injectively immersed submanifold of $\Omega$ and $X_1,\ldots, X_q$ are real analytic vector fields on $B_X(x_0,\xi)$.
Let $\nu$ denote the induced Lebesgue density on $B_X(x_0,\xi)$.  The goal of this section is to describe how the hypotheses of \cref{Thm::Density::MainThm,Cor::Density::MainCor} hold, for this choice of $\nu$, in a quantitative way.\footnote{The quantitative estimates in this section do not follow from classical proofs.  The main difficulty is that classical proofs break down
near a singular point of the associated foliation. See \cref{Rmk::Intro::BadAtSingular,Rmk::BeyondHor::UsefulForSingular}.}

As in \cref{Thm::QuantRes::MainThm} we, without loss of generality, reorder the vector fields so that $J_0=(1,\ldots, n)$.   

Fix $\delta_1>0$ and $s_1>0$ so that
\begin{equation*}
	X_1,\ldots, X_n\in \ComegaSpace{s_1}[B^N(x_0,\delta_1)][\R^N],
\end{equation*}
where $B^N(x_0,\delta_1)=\{ y\in \R^N : |x_0-y|<\delta_1\}$.

\begin{defn}\label{Defn::EuclidDense::Eadmis}
We say $C$ is an $\Eadmis$-admissible constant\footnote{Here, $\Eadmis$ stands for ``Euclidean''.}
 if $C$ can be chosen to depend only on
anything an admissible constant may depend on (see \cref{Defn::QuantRes::AdmissibleConst}),
and upper bounds for $\delta_1^{-1}$, $s_1^{-1}$, $N$, $\CjNorm{X_j}{1}[\Omega']$, $1\leq j\leq n$,
and $\ComegaNorm{X_j}{s_1}[B^N(x_0,\delta_1)][\R^N]$, $1\leq j\leq n$.
We write $A\lesssim_{\Eadmis} B$ if $A\leq C B$, where $C$ is a positive $\Eadmis$-admissible constant.
We write $A\approx_{\Eadmis} B$ for $A\lesssim_{\Eadmis} B$ and $B\lesssim_{\Eadmis} A$.
\end{defn}

 The main result of this section is the following:

\begin{prop}\label{Prop::EuclidDense::MainProp}
There exists an $\Eadmis$-admissible constant $r>0$ and functions $f_1,\ldots, f_n\in \AXSpace{X_{J_0}}{x_0}{r}\cap \CSpace{B_{X_{J_0}}(x_0,\chi)}$ such that for $1\leq j\leq n$, $\Lie{X_j} \nu = f_j \nu$
and $\AXNorm{f_j}{X_{J_0}}{x_0}{r}\lesssim_{\Eadmis} 1$, $\CNorm{f_j}{B_{X_{J_0}}(x_0,\chi)}\lesssim_{\Eadmis} 1$.
\end{prop}

\begin{rmk}
It is an immediate consequence of \cref{Prop::EuclidDense::MainProp} that  \cref{Thm::Density::MainThm,Cor::Density::MainCor} hold, for this choice of $\nu$, where any constant which is
$\nu$-admissible in the sense of those results, is $\Eadmis$-admissible in the sense of \cref{Defn::EuclidDense::Eadmis}.
\end{rmk}

The rest of this section is devoted to the proof of \cref{Prop::EuclidDense::MainProp}.  By \cref{Lemma::PfQuantLD::HyposHold}, for $1\leq i,j,k\leq n$, there exists an admissible constant $\eta'>0$ and 
$\ch_{i,j}^k\in \CSpace{B_{X_{J_0}}(x_0,\chi)}\cap \AXSpace{X_{J_0}}{x_0}{\eta'}$ such that 
\begin{equation}\label{Eqn::EuclideDense::Commutator}
[X_i,X_j]=\sum_{k=1}^n \ch_{i,j}^k X_k,\quad 1\leq i,j\leq n,
\end{equation}
where $\CNorm{\ch_{i,j}^k}{B_{X_{J_0}}(x_0,\chi)}\lesssim_0 1$ and $\AXNorm{\ch_{i,j}^k}{X_{J_0}}{x_0}{\eta'}\lesssim 1$.

We abuse notation and write $X_{J_0}$ to both denote the list of vector fields $X_1,\ldots, X_n$ and the $N\times n$ matrix whose columns are given by $X_1,\ldots, X_n$.
For $K\in \sI_0(n,N)$ we write $X_{K,J_0}$ to denote the $n\times n$ submatirx of $X_{J_0}$ given by taking the rows listed in $K$.
We set $\det_{n\times n} X_{J_0} = (\det X_{K,J_0})_{K\in\sI_0(n,N)}$, so that $\det_{n\times n} X_{J_0} $ is a vector (it is not important in which order the coordinates are arranged).
Since $X_1\wedge X_2\wedge \cdots \wedge X_n$ is never zero on $B_{X_{J_0}}(x_0,\chi)$ (by \cref{Thm::QuantRes::MainThm} \cref{Item::QuantRes::LI}),
we have $|\det_{n\times n} X_{J_0}|>0$ on $B_{X_{J_0}}(x_0,\chi)$.

\begin{lemma}\label{Lemma::EuclidDense::DerivDet}
There exists an $\Eadmis$-admissible constant $\eta''>0$ such that for $1\leq j\leq n$ and $K\in \sI_0(n,N)$,
\begin{equation*}
	X_j \det X_{K,J_0} = \sum_{L\in \sI_0(n,N)} \gt_{j,K}^L \det X_{L,J_0},
\end{equation*}
where $\gt_{j,K}^L \in \AXSpace{X_{J_0}}{x_0}{\eta''} \cap \CSpace{B_{X_{J_0}}(x_0,\chi)}$ and $\AXNorm{\gt_{j,K}^L}{X_{J_0}}{x_0}{\eta''}\lesssim_{\Eadmis} 1$, $\CNorm{\gt_{j,K}^L}{B_{X_{J_0}}(x_0,\chi)}\lesssim_{\Eadmis} 1$.
\end{lemma}
\begin{proof}
For $K=(k_1,\ldots, k_n)\in \sI_0(n,N)$, set $\nu_K = dx_{k_1}\wedge dx_{k_2}\wedge \cdots \wedge dx_{k_n}$, so that $\nu_K$ is an $n$-form on $\R^N$.
Note that $\det X_{K,J_0}= \nu_K (\bigwedge X_{J_0})$.  Using \cite[Proposition 18.9]{LeeIntroToSmoothManifolds} we have for $1\leq j\leq n$,
\begin{equation}\label{Eqn::EuclideDense::LieDerivOfForm}
	X_j \det X_{K,J_0} = \Lie{X_j} \nu_K \mleft(\bigwedge X_{J_0}\mright) = (\Lie{X_j} \nu_K) \mleft(\bigwedge X_{J_0}\mright) + \nu_K\mleft( \Lie{X_j} \bigwedge X_{J_0}\mright).
\end{equation}
We address the two terms on the right hand side of \cref{Eqn::EuclideDense::LieDerivOfForm} separately.

For the second term on the right hand side of \cref{Eqn::EuclideDense::LieDerivOfForm}, we have
\begin{equation}\label{Eqn::EuclideDense::LieDerivOfForm::1}
\begin{split}
 &\nu_K\mleft( \Lie{X_j} \bigwedge X_{J_0}\mright) = \nu_K(\Lie{X_j} (X_1\wedge X_2\wedge \cdots \wedge X_n))
 \\&= \nu_K( [X_j, X_1]\wedge X_2\wedge \cdots \wedge X_n) + \nu_K ( X_1 \wedge [X_j, X_2]\wedge X_3\wedge \cdots \wedge X_n)
 \\&\quad\quad +\cdots + \nu_K (X_1\wedge X_2\wedge \cdots\wedge X_{n-1} \wedge [X_j, X_n])
\end{split}
\end{equation}
The terms on the right hand side of \cref{Eqn::EuclideDense::LieDerivOfForm::1} are all similar, so we address only the first.  We have, using \cref{Eqn::EuclideDense::Commutator},
\begin{equation*}
\nu_K([X_j, X_1]\wedge X_2\wedge \cdots \wedge X_n) = \sum_{l=1}^n \ch_{j,1}^l \nu_K(X_l\wedge X_2\wedge \cdots \wedge X_n) = \ch_{j,1}^1 \nu_K(X_1\wedge X_2\wedge \cdots \wedge X_n)
=\ch_{j,1}^1 \det X_{K,J_0}.
\end{equation*}
Since $\ch_{j,1}^1\in \AXSpace{X_{J_0}}{x_0}{\eta'}\cap \CSpace{B_{X_{J_0}}(x_0,\chi)}$ with $\AXNorm{\ch_{j,1}^1}{X_{J_0}}{x_0}{\eta'}\lesssim 1$, $\CNorm{\ch_{j,1}^1}{B_{X_{J_0}}(x_0,\chi)}\lesssim_0 1$, this is of the desired form for any $\eta''\leq \eta'$.

We now turn to the first term on the right hand side of \cref{Eqn::EuclideDense::LieDerivOfForm}.  
We have, for $K=(k_1,\ldots, k_n)$,
\begin{equation}\label{Eqn::EudlicdDense::LieOfForm::2}
\begin{split}
	&\Lie{X_j} \nu_K = \Lie{X_j} (dx_{k_1}\wedge dx_{k_2}\wedge \cdots \wedge dx_{k_n})
	\\&=(\Lie{X_j} d x_{k_1})\wedge d x_{k_2}\wedge \cdots d x_{k_n} + dx_{k_1} \wedge (\Lie{X_j} d x_{k_2}) \wedge \cdots \wedge dx_{k_n}
	+ \cdots + d x_{k_1}\wedge dx_{k_2}\wedge \cdots \wedge (\Lie{X_j}dx_{k_n}).
\end{split}
\end{equation}
Each term on the right hand side of \cref{Eqn::EudlicdDense::LieOfForm::2} is similar, so we describe the first.
For this, we write $X_j=\sum_{k=1}^N a_j^k \diff{x_k}$ ($1\leq j\leq n$), where
$a_j^k\in \ComegaSpace{s_1}[B^N(x_0,\delta_1)]\cap \CjSpace{1}[\Omega']$ with $\ComegaNorm{a_j^k}{s_1}[B^N(x_0,\delta_1)]\lesssim_{\Eadmis} 1$ and $\CjNorm{a_j^k}{1}[\Omega']\lesssim_{\Eadmis} 1$.
Then, $\Lie{X_j} d x_{k_1} = da_j^{k_1} = \sum_{l=1}^N \frac{\partial a_j^{k_1}}{\partial x_l} dx_l$.
Thus,
\begin{equation*}
	(\Lie{X_j} d x_{k_1})\wedge d x_{k_2} \wedge \cdots \wedge dx_{k_n}
	= \sum_{l=1}^N \frac{\partial a_j^{k_1}}{\partial x_l} d x_l \wedge dx_{k_2}\wedge \cdots \wedge dx_{k_n}
	= \sum_{l=1}^N  \frac{\partial a_j^{k_1}}{\partial x_l} \epsilon_{K,l} \nu_{K_l},
\end{equation*}
where $\epsilon_{K,l}\in \{-1,0,1\}$ and $K_l\in \sI_0(n,N)$ is obtained by reordering $(l,k_2,k_3,\ldots, k_n)$ to be non-decreasing.
Applying the same ideas to the other terms on the right hand side of \cref{Eqn::EudlicdDense::LieOfForm::2}, we see
\begin{equation}\label{Eqn::EudlicdDense::LieOfForm::3}
	(\Lie{X_j} \nu_K) \mleft(\bigwedge X_{J_0}\mright) = \sum_{L\in \sI_0(n,N)} b_{j,K}^L \nu_L\mleft(\bigwedge X_{J_0}\mright) = \sum_{L\in \sI_0(n,N)} b_{j,K}^L \det X_{L,J_0},
\end{equation}
where each $b_{j,K}^L$ is a sum of terms of the form $\pm \frac{\partial a_j^k}{\partial x_l}$ where $1\leq j\leq n$, $1\leq k,l\leq N$ (and the number of terms in this sum is $\lesssim_{\Eadmis} 1$).

Since $a_j^k\in \ComegaSpace{s_1}[B^N(x_0,\delta_1)]$ with $\ComegaNorm{a_j^k}{s_1}[B^N(x_0,\delta_1)]\lesssim_{\Eadmis} 1$, applying \cref{Lemma::FuncSpaceRev::DerivComega} (with $X=\grad$) shows
\begin{equation*}
	\frac{\partial a_j^k}{\partial x_l}\in \ComegaSpace{s_1/2}[B^N(x_0,\delta_1)], \quad \BComegaNorm{\frac{\partial a_j^k}{\partial x_l}}{s_1/2}[B^N(x_0,\delta_1)]\lesssim_{\Eadmis} 1, \quad 1\leq j\leq n, 1\leq k,l\leq N.
\end{equation*}
We conclude $b_{j,K}^L\in \ComegaSpace{s_1/2}[B^N(x_0,\delta_1)]$ with $\ComegaNorm{b_{j,K}^L}{s_1/2}[B^N(x_0,\delta_1)]\lesssim_{\Eadmis} 1$.

By \cref{Prop::NelsonTheorem2} \cref{Item::NelsonThm::Used}, there exists an $\Eadmis$-admissible constant $s_2>0$ such that
$b_{j,K}^L\in \CXomegaSpace{X_{J_0}}{s_2}[B^N(x_0,\delta_1)]$ with $\CXomegaNorm{b_{j,K}^L}{X_{J_0}}{s_2}[B^N(x_0,\delta_1)]\lesssim_{\Eadmis} 1$.
Also, it immediately follows from the properties of the $a_j^k$ that $b_{j,K}^L\in \CSpace{B_{X_{J_0}}(x_0,\chi)}$ with $\CNorm{b_{j,K}^L}{B_{X_{J_0}}(x_0,\chi)}\lesssim_{\Eadmis} 1$ (here we have used $\chi\leq \xi$
and therefore $B_{X_{J_0}}(x_0,\chi)\subseteq \Omega'$).

Because, for $1\leq j\leq n$, $\CjNorm{X_j}{1}[B^N(x_0,\delta_1)][\R^N] \lesssim_{\Eadmis} \ComegaNorm{X_j}{s_1}[B^N(x_0,\delta_1)][\R^N]\lesssim_{\Eadmis} 1$,
the Picard-Lindel\"of theorem shows that we may take an $\Eadmis$-admissible constant $\eta''\in (0,\min\{s_2,\eta'\}]$ so small that
$X_{J_0}$ satisfies $\sC(x_0,\eta'', B^N(x_0,\delta_1))$.  Then, by \cref{Lemma::FuncSpaceRev::Mfld::ABigger} we have, for $1\leq j\leq n$, $K,L\in \sI_0(n,N)$,
\begin{equation*}
	\AXNorm{b_{j,K}^L}{X_{J_0}}{x_0}{\eta''}\leq \CXomegaNorm{b_{j,K}^L}{X_{J_0}}{s_2}[B^N(x_0,\delta_1)]\lesssim_{\Eadmis} 1.
\end{equation*}
Combining this with \cref{Eqn::EudlicdDense::LieOfForm::3} and the above mentioned fact that $\CNorm{b_{j,K}^L}{B_{X_{J_0}}(x_0,\chi)}\lesssim_{\Eadmis} 1$, shows the first term on the right hand side of \cref{Eqn::EuclideDense::LieDerivOfForm} is of the desired form, completing the proof.
\end{proof}

\begin{lemma}\label{Lemma::EudlidDense::DerivOfFraction}
Let $K\in \sI_0(n,N)$, $1\leq j\leq n$.  Then,
\begin{equation*}
	X_j \frac{\det X_{K,J_0}}{ |\det_{n\times n} X_{J_0}|} = \sum_{L\in \sI_0(n,N)} \gt_{j,K}^L \frac{\det X_{L,J_0}}{|\det_{n\times n} X_{J_0}|} - \sum_{L_1,L_2\in \sI_0(n,N)} \gt_{j,L_1}^{L_2} \frac{ (\det X_{K,J_0})(\det X_{{L_1},J_0}) (\det  X_{L_2, J_0}) }{ |\det_{n\times n} X_{J_0}|^3},
\end{equation*}
where $\gt_{j,L_1}^{L_2}$ are the functions from \cref{Lemma::EuclidDense::DerivDet}.
\end{lemma}
\begin{proof}
We have, using \cref{Lemma::EuclidDense::DerivDet}, for $1\leq j\leq n$, $K\in \sI_0(n,N)$,
\begin{equation*}
\begin{split}
&X_j \frac{ \det X_{K, J_0}}{|\det_{n\times n} X_{J_0}|} = \frac{ X_j \det X_{K,J_0}}{|\det_{n\times n} X_{J_0}|} -  \frac{1}{2} \frac{\det X_{K,J_0}}{|\det_{n\times n} X_{J_0}|^3} X_j |\det_{n\times n} X_{J_0}|^2
\\&=\sum_{L\in \sI_0(n,N)} \gt_{j,K}^L \frac{ \det X_{L,J_0}}{|\det_{n\times n} X_{J_0}|} - \sum_{L_1\in \sI_0(n,N)} \frac{(\det X_{K, J_0}) (\det X_{L_1,J_0}) (X_j \det X_{L_1,J_0})}{ |\det_{n\times n} X_{J_0} |^3 }
\\&= \sum_{L\in \sI_0(n,N)} \gt_{j,K}^L \frac{\det X_{L,J_0}}{|\det_{n\times n} X_{J_0}|} - \sum_{L_1,L_2\in \sI_0(n,N)} \gt_{j,L_1}^{L_2} \frac{ (\det X_{K,J_0})(\det X_{{L_1},J_0}) (\det  X_{L_2, J_0}) }{ |\det_{n\times n} X_{J_0}|^3},
\end{split}
\end{equation*}
completing the proof.
\end{proof}

\begin{lemma}\label{Lemma::EuclidDense::QuotientRA}
There exists an $\Eadmis$-admissible constant $\eta'''>0$ such that $\forall K\in \sI_0(n,N)$,
\begin{equation*}
\frac{\det X_{K, J_0}}{|\det_{n\times n} X_{J_0}|} \in \AXSpace{X_{J_0}}{x_0}{\eta'''}, \quad \BAXNorm{ \frac{\det X_{K, J_0}}{|\det_{n\times n} X_{J_0}|}}{X_{J_0}}{x_0}{\eta'''}\lesssim_{\Eadmis} 1.
\end{equation*}
\end{lemma}
\begin{proof}
For $K\in \sI_0(n,N)$, set $G_K:= \frac{\det X_{K, J_0}}{|\det_{n\times n} X_{J_0}|}$.  \Cref{Lemma::EudlidDense::DerivOfFraction} shows, for $1\leq j\leq n$,
\begin{equation*}
	X_j G_K= 
	 \sum_{L\in \sI_0(n,N)} \gt_{j,K}^L G_L - \sum_{L_1,L_2\in \sI_0(n,N)} \gt_{j,L_1}^{L_2}  G_K G_{L_1}G_{L_2}.
\end{equation*}
Using the estimates on the functions $\gt_{j,L_1}^{L_2}$ described in \cref{Lemma::EuclidDense::DerivDet}, the result follows from \cref{Prop::IdentRAManifold::MainProp},
with $q=n$, $\eta=\eta''$, $\xi=\chi$, $D$ an $\Eadmis$-admissible constant, $L=3$, $N=|\sI_0(n,N)|$, and $|G(x_0)|= 1$.  Here, we take $\eta'''$ to be the $r'$ guaranteed by that result.
\end{proof}

\begin{lemma}\label{Lemma::EudlidDense::BoundDerivCont}
\begin{equation*}
	\frac{\det X_{K, J_0}}{|\det_{n\times n} X_{J_0}|} \in \CSpace{B_{X_{J_0}}(x_0,\chi)}, \quad \BCNorm{ \frac{\det X_{K, J_0}}{|\det_{n\times n} X_{J_0}|}}{B_{X_{J_0}}(x_0,\chi)}\leq 1.
\end{equation*}
\end{lemma}
\begin{proof}
	Since $\bigwedge X_{J_0}$ is never zero on $B_{X_{J_0}}(x_0,\chi)$ (by \cref{Thm::QuantRes::MainThm} \cref{Item::QuantRes::LI}), the continuity of
	 $\frac{\det X_{K, J_0}}{|\det_{n\times n} X_{J_0}|} $ follows immediately.  That it is bounded by $1$ follows directly from the definition.
\end{proof}

\begin{lemma}\label{Lemma::EuclidDense::Definehj}
There exists an $\Eadmis$-admissible constant $\eta_2>0$ and functions $h_j\in \AXSpace{X_{J_0}}{x_0}{\eta_2}\cap \CSpace{B_{X_{J_0}}(x_0,\chi)}$, $1\leq j\leq n$ so that
\begin{equation*}
	X_j \mleft|\det_{n\times n} X_{J_0}\mright| =h_j \mleft|\det_{n\times n} X_{J_0}\mright|,\quad 1\leq j\leq n,
\end{equation*}
and $\AXNorm{h_j}{X_{J_0}}{x_0}{\eta_2}\lesssim_{\Eadmis} 1$, $\CNorm{h_j}{B_{X_{J_0}}(x_0,\chi)}\lesssim_{\Eadmis} 1$.
\end{lemma}
\begin{proof}
Using \cref{Lemma::EuclidDense::DerivDet} we have
\begin{equation*}
\begin{split}
	&X_j \mleft|\det_{n\times n} X_{J_0}\mright| = \frac{1}{2} \mleft|\det_{n\times n} X_{J_0}\mright|^{-1} X_j \mleft|\det_{n\times n} X_{J_0}\mright|^2
	=\frac{1}{2} \mleft|\det_{n\times n} X_{J_0}\mright|^{-1} \sum_{K\in \sI_0(n,N)} X_j (\det X_{K,J_0})^2
	\\&= \mleft|\det_{n\times n} X_{J_0}\mright|^{-1} \sum_{K,L\in \sI_0(n,N)} \gt_{j,K}^L (\det X_{K,J_0}) (\det X_{L,J_0})
	\\&=\mleft|\det_{n\times n} X_{J_0}\mright| \sum_{K,L\in \sI_0(n,N)} \gt_{j,K}^L \frac{\det X_{K,J_0}}{|\det_{n\times n} X_{J_0}|} \frac{\det X_{L,J_0}}{|\det_{n\times n} X_{J_0}|}.
\end{split}
\end{equation*}
We set $h_j: = \sum_{K,L\in \sI_0(n,N)} \gt_{j,K}^L \frac{\det X_{K,J_0}}{|\det_{n\times n} X_{J_0}|} \frac{\det X_{L,J_0}}{|\det_{n\times n} X_{J_0}|}$, and let $\eta_2:=\eta''\wedge \eta'''$.
Then, using the bounds on $\gt_{j,K}^L$ from \cref{Lemma::EuclidDense::DerivDet}, the bounds on $\frac{\det X_{L,J_0}}{|\det_{n\times n} X_{J_0}|}$ from \cref{Lemma::EuclidDense::QuotientRA,Lemma::EudlidDense::BoundDerivCont},
and \cref{Lemma::FuncSpaceRev::Algebra}, we have $h_j\in \AXSpace{X_{J_0}}{x_0}{\eta_2}\cap \CSpace{B_{X_{J_0}}(x_0,\chi)}$ and $\AXNorm{h_j}{X_{J_0}}{x_0}{\eta_2}\lesssim_{\Eadmis} 1$, $\CNorm{h_j}{B_{X_{J_0}}(x_0,\chi)}\lesssim_{\Eadmis} 1$ completing the proof.
\end{proof}

\begin{lemma}\label{Lemma::EuclidDense::FormulaFornu}
$\nu=|\det_{n\times n} X_{J_0}| \nu_0$ on $B_{X_{J_0}}(x_0,\chi)$, where $\nu_0$ is given by \cref{Eqn::PfDensities::Defnnu0}.
\end{lemma}
\begin{proof}
Since $\nu_0(\bigwedge X_{J_0})\equiv 1$, $\nu_0$ is a strictly positive density on $B_{X_{J_0}}(x_0,\chi)$.  Thus, $\nu=f(x) \nu_0$ for some $f:B_{X_{J_0}}(x_0,\chi)\rightarrow \R$.
To solve for $f$ we evaluate this equation at $\bigwedge X_{J_0}$ and since $\nu_0(\bigwedge X_{J_0})\equiv 1$, we have
$f = \nu(\bigwedge X_{J_0})$.  Since $\nu$ is the induced Lebesgue density on an $n$-dimensional, injectively immersed, submanifold of $\R^N$ (to which $X_1,\ldots, X_q$ are tangent), we have
$\nu(\bigwedge X_{J_0})= |\det_{n\times n} X_{J_0}|$, completing the proof.
\end{proof}

\begin{proof}[Proof of \cref{Prop::EuclidDense::MainProp}]
For $1\leq j\leq n$, let $f_j:=f_j^0+h_j$, where $f_j^0$ is described in \cref{Lemma::PfDensities::fj0RA} and $h_j$ is described in \cref{Lemma::EuclidDense::Definehj}.  Then, if $r:=\min\{\eta_1,\eta_2\}>0$
we have that $r$ is an $\Eadmis$-admissible constant, $f_j\in \AXSpace{X_{J_0}}{x_0}{r}\cap \CSpace{B_{X_{J_0}}(x_0,\chi)}$, and $\AXNorm{f_j}{X_{J_0}}{x_0}{r}\lesssim_{\Eadmis} 1$,
$\CNorm{f_j}{B_{X_{J_0}}(x_0,\chi)}\lesssim_{\Eadmis} 1$.

Using \cref{Lemma::PfDensities::fj0RA,Lemma::EuclidDense::Definehj,Lemma::EuclidDense::FormulaFornu} we have, for $1\leq j\leq n$,
\begin{equation*}
	\Lie{X_j} \nu = \Lie{X_j} \mleft|\det_{n\times n} X_{J_0}\mright| \nu_0
	= \mleft(X_j \mleft|\det_{n\times n} X_{J_0}\mright|\mright) \nu_0 + \mleft|\det_{n\times n} X_{J_0}\mright| \Lie{X_j} \nu_0
	=h_j  \mleft|\det_{n\times n} X_{J_0}\mright| \nu_0 + f_j^0  \mleft|\det_{n\times n} X_{J_0}\mright|\nu_0
	= f_j \nu,
\end{equation*}
completing the proof.
\end{proof}

\section{\texorpdfstring{Scaling and real analyticity:  the proof of \cref{Thm::Scaling::BeyondHor}}{Scaling and real analyticity}}\label{Section::ScalingRevis}
In this section, we prove \cref{Thm::Scaling::BeyondHor}.
Fix a compact set $\Compact\Subset \Omega$.
The idea is that, if $m$ is chosen sufficiently large (depending on $\Compact$ and $V_1,\ldots, V_r$), then a proof similar to that of \cref{Thm::GenSubR::MainThm} will work
 uniformly for  $x\in \Compact$ with the manifold $M$ from \cref{Thm::GenSubR::MainThm} replaced by $L_x$.
As in \cref{Thm::Scaling::BeyondHor}, throughout this section we use $A\lesssim B$ to denote $A\leq C B$ where
$C$ can be chosen independent of $x\in \Compact$ and $\delta\in (0,1]$.

We first show how the appropriate conditions hold uniformly, which uses the Weierstrass preparation theorem.  This is based on an idea of Lobry \cite{LobryControlabiliteDesSystems}\footnote{\Cref{Lemma::ScalinvRevis::NoetherianA,Lemma::ScalingRevis::Noetherian} are classical, and it was Lobry \cite{LobryControlabiliteDesSystems} who first used them
to prove a result similar to \cref{Prop::ScalingRevis::NSW}.  Unfortunately, there is a slight error in \cite{LobryControlabiliteDesSystems}; see \cite{StefanIntegrabilityOfSystemsOfVectorFields}.
Below, we begin with a \textit{finite} collection of vector fields $V_1,\ldots, V_r$, which allows us to avoid this issue.}
 (see also
\cite[p. 188]{SussmanOrbitsOfFamiliesOfVectorFieldsAndIntegrabilityOfDistributions}), and was
used in a similar context in \cite{SteinStreetIII} and \cite[Section 2.15.5]{StreetMultiParamSingInt}.

We take the same setting and notation as in \cref{Thm::Scaling::BeyondHor}.  Thus, we have real analytic vector fields $V_1,\ldots, V_r$ on an open set $\Omega\subseteq \R^n$.
We assign to each $V_1,\ldots, V_r$ the formal degree $1$.  If $Z$ has formal degree $e$, we assign to $[V_j, Z]$ the formal degree $e+1$.
We let $\sS$ denote the (infinite) set of all such vector fields with formal degrees; thus each $(X,e)\in \sS$ is a pair of a real analytic vector field $X$ and $e\in \N$, where
$X$ is an iterated commutator of $V_1,\ldots, V_r$.

An important ingredient is the next proposition:
\begin{prop}\label{Prop::ScalingRevis::NSW}
Fix $x\in \Omega$.  Then, there exists an open neighborhood $U_x\subseteq \Omega$ of $x$ and $m_x\in \N$
such that the following holds.
Let $m\geq m_x$, and let
 $(X_1,d_1),\ldots, (X_{q},d_{q})$ be an enumeration of those elements $(Z,e)\in \sS$ with $e\leq m$.  Then there exist
real analytic functions $\ch_{j,k}^{l,x}\in \CjSpace{\omega}[U_{x}]$ with
\begin{equation*}
	[X_j,X_k]=\sum_{d_l\leq d_j+d_k} \ch_{j,k}^{l,x} X_l \text{ on } U_{x}.
\end{equation*}
\end{prop}

We turn to the proof of \cref{Prop::ScalingRevis::NSW}.  In the next few results, we (without loss of generality) relabel the fixed point $x\in \Omega$ from \cref{Prop::ScalingRevis::NSW} to be $0$.
Thus, we work near the point $0\in \R^n$.
We write $f:\R^n_0\rightarrow \R^n$ to denote that $f$ is a germ of a function defined near $0\in \R^n$.  Let
\begin{equation*}
\sA_n:=\mleft\{ f:\R^n_0\rightarrow \R\: \big|\: f\text{ is real analytic}\mright\},
\end{equation*}
\begin{equation*}
\sA_n^m:=\mleft\{ f:\R^n_0\rightarrow \R^m\: \big|\: f\text{ is real analytic}\mright\}.
\end{equation*}
Notice that $\sA_n^m$ can be identified with the $m$-fold Cartesian product of $\sA_n$; justifying our notation.  $\sA_n^n$ can also be identified with the space of germs of real analytic vector fields
near $0\in \R^n$.

\begin{lemma}\label{Lemma::ScalinvRevis::NoetherianA}
The ring $\sA_n$ is Noetherian.
\end{lemma}
\begin{proof}[Comments on the proof]
This is a simple consequence of the Weierstrass preparation theorem.  See page 148 of \cite{ZariskiSamuel}.  The proof in \cite{ZariskiSamuel} is for the formal power series ring, however,
as mentioned on page 130 of \cite{ZariskiSamuel}, the proof also works for the ring of convergent power series.  This is exactly the ring $\sA_n$.
\end{proof}

\begin{lemma}\label{Lemma::ScalingRevis::Noetherian}
The module $\sA_n^m$ is a Noetherian $\sA_n$ module.
\end{lemma}
\begin{proof}[Comments on the proof]
It is a standard fact that for any Noetherian ring $R$, the $R$-module $R^m$ is Noetherian.
\end{proof}

\begin{lemma}\label{Lemma::ScalingRevis::NoetherianWithFormal}
Let $\sS\subset \sA_n^n\times \N$.  Then, there exists a finite subset $\sF\subseteq \sS$ such that for every $(g,e)\in \sS$,
\begin{equation*}
	g(x) = \sum_{\substack{(f,d)\in \sF \\ d\leq e}} c_{(f,d)}(x) f(x),
\end{equation*}
with $c_{f,d}\in \sA_n$.
\end{lemma}
\begin{proof}
Define a map $\iota:\sA_n^n\times \N\rightarrow \sA_{n+1}^n$ by $\iota(f,d)=t^d f(x)$, where $t\in \R$, $x\in \R^n$.
Let $\sM$ be the submodule of $\sA_{n+1}^n$ generated by $\iota \sS$.  $\sM$ is finitely generated by \cref{Lemma::ScalingRevis::Noetherian}.
Let $\sF\subseteq \sS$ be a finite subset so that $\iota \sF$ generates $\sM$.  We will show that $\sF$ satisfies the conclusions of the lemma.

Let $(g,e)\in \sS$.  Since $t^e g\in \sM$, we have
\begin{equation*}
	t^e g(x) = \sum_{(f,d)\in \sF} \ch_{(f,d)}(t,x) t^d f(x),\quad \ch_{(f,d)}\in \sA_{n+1},
\end{equation*}
on a neighborhood of $(0,0)\in \R\times \R^n$.  Applying $\frac{1}{e!} \partial_t^e\big|_{t=0}$ to both sides and using that
$\frac{1}{e!} \partial_t^e\big|_{t=0} t^d c_{(f,d)}(t,x) =0$ if $d>e$, we have
\begin{equation*}
	g(x) = \sum_{\substack{(f,d)\in \sF \\ d\leq e}} \mleft[ \frac{1}{e!} \partial_t^e \big|_{t=0} t^d  \ch_{(f,d)}(t,x)\mright] f(x).
\end{equation*}
The result follows.
\end{proof}

We return to the setting at the start of this section.  We let $\sSh$ denote the smallest collection of vector fields paired with formal degrees such that:
\begin{itemize}
\item $(V_1,1),\ldots, (V_r,1)\in \sSh$.
\item If $(Y,d),(Z,e)\in \sSh$, then $([Y,Z],d+e)\in \sSh$.
\end{itemize}
Note that $\sSh$ and $\sS$ (where $\sS$ is defined at the start of this section) are closely related.  Indeed, $\sS\subseteq \sSh$, and if $(Z,e)\in \sSh$ then
$Z$ is a finite linear combination (with constant coefficients) of vector fields $Z'$ such that $(Z',e)\in \sS$ (this follows from the Jacobi identity).

\begin{lemma}\label{Lemma::ScalingRevis::MakeUAndF}
Fix $x\in \Omega$.
There exists a finite set $\sF_x\subseteq \sS$ and an open neighborhood $U_x\Subset \Omega$ of $x$ such that for every $(X,d)\in \sSh$,
\begin{equation*}
	X = \sum_{\substack{(Z,e)\in \sF_x\\ e\leq d}} c_{(X,d)}^{(Z,e)} Z, \quad c_{(X,d)}^{(Z,e)}\in \CjSpace{\omega}[U_x].
\end{equation*}
\end{lemma}
\begin{proof}
In this proof we relabel $x$ to be $0$.
We apply \cref{Lemma::ScalingRevis::NoetherianWithFormal} to $\sSh$ to obtain a finite set $\sFh\subseteq \sSh$ as in that lemma.  Thus, every element $(X,d)\in \sSh$
can be written in the form:
\begin{equation*}
X = \sum_{\substack{(Z,e)\in \sFh\\ e\leq d}} c_{(X,d)}^{(Z,e)} Z, \quad c_{(X,d)}^{(Z,e)}\in \sA_n.
\end{equation*}
Since every vector field appearing in $\sSh$ is a finite linear combination (with constant coefficients) of vector fields with the same degree in $\sS$, there is a finite set $\sF_1\subseteq \sS$
such that every element $(X,d)\in \sSh$
can be written in the form:
\begin{equation}\label{Eqn::ScalinvRevis::SumF0}
X = \sum_{\substack{(Z,e)\in \sF_1\\ e\leq d}} c_{(X,d)}^{(Z,e)} Z, \quad c_{(X,d)}^{(Z,e)}\in \sA_n.
\end{equation}
The problem is that, a priori, the neighborhood of $0$ on which $c_{(X,d)}^{(Z,e)}$ are defined might depend on $(X,d)$.  Our goal is to find a common neighborhood of $0$
which works for all $(X,d)\in \sSh$.

Let $m:=\max \{ d : \exists (X,d)\in \sF_1\}$ and set $\sF:=\{ (X,d)\in \sS : d\leq m\}$.  Note that $\sF\subset \sS$ is a finite set and $\sF_1\subseteq \sF$.
Furthermore, \cref{Eqn::ScalinvRevis::SumF0} holds with $\sF_1$ replaced by $\sF$ (since $\sF_1\subseteq \sF$).
For each $(X_1,d_1),(X_2,d_2)\in \sF$, $([X_1,X_2],d_1+d_2)\in \sSh$ and therefore \cref{Eqn::ScalinvRevis::SumF0} holds with $(X,d)$ replaced by $([X_1,X_2],d_1+d_2)$
and $\sF_1$ replaced by $\sF$.  Since there are only finitely many such vector fields, there is a common neighborhood $U_1$ of $0$ such that for each
$(X_1,d_1),(X_2,d_2)\in \sF$ we may write
\begin{equation}\label{Eqn::ScalingRevis::FirstCommutator}
	[X_1,X_2]=\sum_{\substack{(Z,e)\in \sF\\ e\leq d_1+d_2}}b_{(X_1,d_1), (X_2,d_2)}^{(Z,e)} Z, \quad b_{(X_1,d_1),(X_2,d_2)}^{(Z,e)}\in \CjSpaceloc{\omega}[U_1].
\end{equation}
We claim, $\forall (X,d)\in \sSh$,
\begin{equation}\label{Eqn::ScalingRevis::ToProveSum}
	X = \sum_{\substack{(Z,e)\in \sF\\ e\leq d}} c_{(X,d)}^{(Z,e)} Z, \quad c_{(X,d)}^{(Z,e)}\in \CjSpaceloc{\omega}[U_1].
\end{equation}
We prove \cref{Eqn::ScalingRevis::ToProveSum} by induction on $d$.  Since $(V_1,1),\ldots, (V_r,1)\in \sF$, the base case, $d=1$, is clear.  Let $d_0\geq 2$; we assume the result
for all $d<d_0$ and prove the result for $d_0$.  If $(X,d_0)\in \sSh$ with $d_0\geq 2$, then $X=[X_1,X_2]$ where $(X_1,d_1),(X_2,d_2)\in \sSh$ with $d_1+d_2=d_0$.
In particular, $d_1,d_2<d_0$ and so by the inductive hypothesis
we may write for $j=1,2$,
\begin{equation*}
	X_j = \sum_{\substack{(Z,e)\in \sF\\ e\leq d_j}} c_{(X_j,d_j)}^{(Z,e)} Z, \quad c_{(X_j,d_j)}^{(Z,e)}\in \CjSpaceloc{\omega}[U_1].
\end{equation*}
Thus, we have, using \cref{Eqn::ScalingRevis::FirstCommutator},
\begin{equation*}
\begin{split}
	&X = [X_1,X_2]
	=\mleft[\sum_{\substack{(Z_1,e_1)\in \sF\\ e_1\leq d_1}} c_{(X_1,d_1)}^{(Z_1,e_1)} Z_1, \sum_{\substack{(Z_2,e_2)\in \sF\\ e_2\leq d_2}} c_{(X_2,d_2)}^{(Z_2,e_2)} Z_2\mright]
	\\&=\sum_{\substack{(Z_1,e_1)\in \sF\\ e_1\leq d_1}} \sum_{\substack{(Z_2,e_2)\in \sF\\ e_2\leq d_2}} \mleft( \mleft( Z_1 c_{(X_2,d_2)}^{(Z_2,e_2)}\mright) Z_2 -  \mleft( Z_2 c_{(X_1,d_1)}^{(Z_1,e_1)}\mright) Z_1 + c_{(X_1,d_1)}^{(Z_1,e_1)}c_{(X_2,d_2)}^{(Z_2,e_2)} [Z_1,Z_2]\mright)
	\\&= \sum_{\substack{(Z_1,e_1)\in \sF\\ e_1\leq d_1}} \sum_{\substack{(Z_2,e_2)\in \sF\\ e_2\leq d_2}} \mleft( \mleft( Z_1 c_{(X_2,d_2)}^{(Z_2,e_2)}\mright) Z_2 -  \mleft( Z_2 c_{(X_1,d_1)}^{(Z_1,e_1)}\mright) Z_1 + c_{(X_1,d_1)}^{(Z_1,e_1)}c_{(X_2,d_2)}^{(Z_2,e_2)} \sum_{\substack{(Z_3,e_3)\in \sF \\ e_3\leq e_1+e_2 }}   b_{(Z_1,e_1),(Z_2,e_2)}^{(Z_3,e_3)} Z_3    \mright).
\end{split}
\end{equation*}
Since $c_{(X_1,d_1)}^{(Z_1,e_1)}, c_{(X_2,d_2)}^{(Z_2,e_2)}, b_{(Z_1,e_1),(Z_2,e_2)}^{(Z_3,e_3)} \in \CjSpaceloc{\omega}[U_1]$, $\forall (Z_1,e_1),(Z_2,e_2),(Z_3,e_3)\in \sF$,
\cref{Eqn::ScalingRevis::ToProveSum} follows for $(X,d)$, completing the proof of \cref{Eqn::ScalingRevis::ToProveSum}.
Taking $\sF_x:=\sF$ and $U\Subset U_1$ a relatively compact open set containing $0=x$, the result follows.
\end{proof}

\begin{proof}[Proof of \cref{Prop::ScalingRevis::NSW}]
Let $\sF_x\subset \sS$ and $U_x\subseteq \Omega$ be as in \cref{Lemma::ScalingRevis::MakeUAndF}.  Set $m_x:=\max\{ e : \exists (Z,e)\in \sF_x\}$.  For $m\geq m_x$
let $(X_1,d_1),\ldots, (X_q,d_q)$ be an enumeration of those elements $(Z,e)\in \sS$ with $e\leq m$.  Since $\sF_x\subseteq \{ (X_1,d_1),\ldots, (X_q,d_q)\}$ and
for each $i,j$, $([X_i,X_j],d_i+d_j)\in \sSh$, the result follows from \cref{Lemma::ScalingRevis::MakeUAndF}.
\end{proof}

\begin{lemma}\label{Lemma::ScalingRevis::Pickm}
There exists $m=m(\Compact)\in \N$ such that the following holds.  Let $(X_1,d_1),\ldots, (X_q,d_q)$ be the list of vector fields
with formal degrees $\leq m$ as defined at the start of this section.  Then, there exists $\xi\in (0,1]$, $s>0$ such that $\forall x\in \Compact$,
\begin{equation}\label{Eqn::ScalingRevis::CommuteAsInNSW}
	[X_j,X_k]=\sum_{d_l\leq d_j+d_k} c_{j,k}^{l,x} X_l, \quad c_{j,k}^{l,x}\in \CXomegaSpace{X}{s}[B_X(x,\xi)],
\end{equation}
where $X$ denotes the list of vector fields $X_1,\ldots, X_q$.  Furthermore,
\begin{equation*}
	\sup_{x\in \Compact} \CXomegaNorm{c_{j,k}^{l,x}}{X}{s}[B_X(x,\xi)]<\infty.
\end{equation*}
Finally, $\bigcup_{x\in \Compact} B_X(x,\xi)\Subset \Omega$.
\end{lemma}
\begin{proof}
For each $x\in \Omega$, let $m_x\in \N$ and $U_x\Subset \Omega$ be as in \cref{Prop::ScalingRevis::NSW}.  
Fix an open set $\Omega'$ with $\Compact\Subset \Omega'\Subset \Omega$, and set $\Compact_1:=\overline{\Omega'}\Subset \Omega$.
$\{ U_x: x\in \Compact_1\}$ is an open cover for $\Compact_1$ and we extract a finite sub-cover
$U_{x_1},\ldots, U_{x_R}$.  Set $m:=\max \{ m_{x_k} : 1\leq k\leq R\}$ and let $(X_1,d_1),\ldots, (X_q,d_q)\in \sS$ be an enumeration of all the vector fields in $\sS$ with degree $d_j\leq m$.

We claim that there exists $\xi\in (0,1]$ such that $\forall x\in \Compact$, $\exists k\in \{1,\ldots, R\}$ with $B_X(x,\xi)\subseteq U_{x_k}$.
Consider the list of vector fields $W=X_1,\ldots, X_q, \diff{x_1},\ldots, \diff{x_n}$.  
By the Phragmen-Lindel\"of principle, we may take $\xi'>0$ so small $B_W(x,\xi')\subseteq \Omega'\subset \Compact_1$, for all $x\in \Compact$.
The balls $B_W(x,\delta)$ are metric balls and the topology induced by these balls on $\Omega$ is the usual topology.
Let $\xi\in (0,\xi']$ be less than or equal to the Lebesgue number for the cover $U_{x_1},\ldots, U_{x_R}$ of $\Compact_1$, with respect to the metric associated to the balls $B_W(x,\delta)$.  Then, $\forall x\in \Compact$, $\exists k\in \{1,\ldots, R\}$ with
$B_X(x,\xi)\subseteq B_W(x,\xi) = B_W(x,\xi)\cap \Compact_1\subseteq U_{x_k}$, as desired.  Also,
$\bigcup_{x\in \Compact} B_X(x,\xi) \subseteq U_{x_1}\cup \cdots \cup U_{x_R}\Subset \Omega$.

Let $\ch_{j,k}^{l,x}$ be the functions from \cref{Prop::ScalingRevis::NSW} (with this choice of $m$).  Take $s_1>0$ so that $\forall r\in \{1,\ldots, R\}$,
$\ch_{j,k}^{l,x_r}\in \ComegaSpace{s_1}[U_{x_r}]$.  

Take $x\in \Compact$ and $1\leq r\leq R$ so that $B_X(x,\xi)\subseteq U_{x_r}$.  Set $c_{j,k}^{l,x}:=\ch_{j,k}^{l,x_r}\big|_{B_X(x,\xi)}$.  By \cref{Prop::NelsonTheorem2}, there exists $s\approx 1$ with
$c_{j,k}^{l,x}\in \CXomegaSpace{X}{s}[B_X(x,\xi)]\subseteq \CXomegaSpace{X}{s}[U_{x_r}]\subseteq \ComegaSpace{s_1}[U_{x_r}]$ and
\begin{equation*}
	\CXomegaNorm{c_{j,k}^{l,x}}{X}{s}[B_X(x,\xi)]\leq \CXomegaNorm{\ch_{j,k}^{l,x_r}}{X}{s}[U_{x_r}]\lesssim \ComegaNorm{\ch_{j,k}^{l,x_r}}{s_1}[U_{x_r}]\leq \max_{1\leq r\leq R} \ComegaNorm{\ch_{j,k}^{l,x_r}}{s_1}[U_{x_r}] \lesssim 1.
\end{equation*}
By the definition of $c_{j,k}^{l,x}$, we have $[X_j,X_k]=\sum_{l} c_{j,k}^{l,x} X_l$, for $x\in B_X(x,\xi)$, completing the proof.
\end{proof}

Let $m$, $c_{j,k}^{l,x}$, $\xi$, $s$, and $(X_1,d_1),\ldots, (X_q,d_q)$ be as in \cref{Lemma::ScalingRevis::Pickm}.  We will prove \cref{Thm::Scaling::BeyondHor} with this choice of $m$.
For $\delta\in (0,1]$, set $X_j^{\delta}:=\delta^{d_j} X_j$, let $X^{\delta}$ denote the list $X_1^{\delta},\ldots, X_q^{\delta}$, and set
\begin{equation*}
c_{j,k}^{l,x,\delta} :=
\begin{cases}
\delta^{d_j+d_k-d_l}c_{j,k}^{l,x},&\text{if }d_j+d_k\geq d_l,\\
0,&\text{otherwise}.
\end{cases}
\end{equation*}

\begin{lemma}\label{Lemma::ScalingRevis::Commutators}
$[X_j^\delta,X_k^{\delta}]=\sum_{l} c_{j,k}^{l,x,\delta} X_l^{\delta}$ on $B_{X^{\delta}}(x,\xi)$ and
\begin{equation*}
	\sup_{x\in \Compact,\delta\in (0,1]} \BCXomegaNorm{c_{j,k}^{l,x,\delta}}{X^{\delta}}{s}[B_{X^{\delta}}(x,\xi)]<\infty.
\end{equation*}
\end{lemma}
\begin{proof}This follows immediately from \cref{Lemma::ScalingRevis::Pickm}, by multiplying \cref{Eqn::ScalingRevis::CommuteAsInNSW} by $\delta^{d_j+d_k}$ and tracing through the definitions.
\end{proof}

We pick $j_1=j_1(x,\delta),\ldots, j_{n_0(x)}=j_{n_0(x)}(x,\delta)$ as in \cref{Section::Scaling::BeyondHormander} and we set
$$J_0=J_0(x,\delta):=(j_1(x,\delta),\ldots, j_{n_0(x)}(x,\delta))\in \sI(n_0(x),q) )\in \sI(n_0(x),q).$$
Let $X^{\delta}_{J_0}$ denote the list $X_{j_1(x,\delta)}^{\delta},\ldots, X_{j_{n_0(x)}(x,\delta)}^{\delta}$.

\begin{lemma}\label{Lemma::ScalingRevis::TheoremApplies}
The conditions \cref{Thm::QuantRes::MainThm} hold uniformly when applied to $X_1^{\delta},\ldots, X_q^{\delta}$, at the base point $x$, for $x\in \Compact$ and $\delta\in (0,1]$.  Here we take $\xi$ and
$J_0=J_0(x,\delta)$ as defined above.  The corresponding map $\Phi$ from \cref{Thm::QuantRes::MainThm} is the map $\Phi_{x,\delta}$ defined in \cref{Eqn::Scaling::BeyondHor::DefinePhi}.
In particular, any constant which is admissible in the sense of \cref{Thm::QuantRes::MainThm} can be taken independent of $x\in \Compact$ and $\delta\in (0,1]$, when applied to $X_1^{\delta},\ldots, X_q^{\delta}$ at the base point $x$.
\end{lemma}
\begin{proof}
We use $\xi>0$ and $c_{j,k}^{l,x,\delta}$ from \cref{Lemma::ScalingRevis::Commutators}.
The existence of $\delta_0$, independent of $x\in \Compact$ and $\delta\in (0,1]$, as in the hypotheses of \cref{Thm::QuantRes::MainThm} follows directly from \cref{Lemma::PfQual::Existsetadelta}.
The existence of $\eta>0$, independent of $x\in \Compact$ and $\delta\in (0,1]$, such that $X^{\delta}$ satisfies $\sC(x,\eta,\Omega)$, $\forall x\in \Compact$ also follows from \cref{Lemma::PfQual::Existsetadelta}.
With $J_0$ as above, we may take $\zeta\approx 1$ in \cref{Thm::QuantRes::MainThm}.  By \cref{Lemma::FuncSpaceRev::Mfld::ABigger} and \cref{Lemma::ScalingRevis::Commutators}, we have
\begin{equation*}
	\sup_{x\in \Compact, \delta\in (0,1]} \BAXNorm{c_{j,k}^{l,x,\delta}}{X_{J_0(x,\delta)}^{\delta}}{x}{\min\{ s,\eta,\xi\}} \leq \sup_{x\in \Compact,\delta\in (0,1]} \BCXomegaNorm{c_{j,k}^{l,x,\delta}}{X^{\delta}}{s}[B_{X^{\delta}}(x,\xi)]<\infty.
\end{equation*}
Thus, we use $\min\{s,\eta,\xi\}$ in place of $\eta$ in the hypotheses of \cref{Thm::QuantRes::MainThm}.  The result follows.
\end{proof}

\Cref{Lemma::ScalingRevis::TheoremApplies} shows that \cref{Thm::QuantRes::MainThm} applies to $X_1^{\delta},\ldots, X_q^{\delta}$ at the base point $x\in \Compact$, to yield constants
$\chi, \eta_1, \xi_1,\xi_2>0$ as in that theorem (with $0<\xi_2\leq \xi_1\leq \chi\leq \xi$), which are independent of $x\in \Compact$ and $\delta\in (0,1]$.

\begin{proof}[Proof of \cref{Thm::Scaling::BeyondHor} \cref{Item::BeyongHor::Pullbacks}]
This follows directly from \cref{Thm::QuantRes::MainThm} \cref{Item::QuantRes::BoundA} and \cref{Item::QuantRes::BoundY}.
\end{proof}

\begin{proof}[Proof of \cref{Thm::Scaling::BeyondHor} \cref{Item::BeyondHor::Containments}]
By \cref{Thm::QuantRes::MainThm} we have $B_{X^{\delta}}(x,\xi_2)\subseteq \Phi_{x,\delta}(B^n(\eta_1))\subseteq B_{X^{\delta}}(x,\xi)$.  Taking $\xi_0=\xi_2$, we have
$B_{(X,d)}(x,\xi_0 \delta) \subseteq B_{X^{\delta}}(x,\xi_2)$ and since $\xi\leq 1$, $B_{X^{\delta}}(x,\xi)\subseteq B_{X^\delta}(x,1)=B_{(X,d)}(x,\delta)$, completing the proof of \cref{Item::BeyondHor::Containments}.
\end{proof}

\begin{rmk}\label{Rmk::ScalingRevis::xi0xi2}
In the proof of \cref{Thm::Scaling::BeyondHor} \cref{Item::BeyondHor::Containments}, we chose $\xi_0$ so that $B_{(X,d)}(x,\xi_0 \delta) \subseteq B_{X^{\delta}}(x,\xi_2)$.
\end{rmk}

\begin{proof}[Proof of  \cref{Thm::Scaling::BeyondHor} \cref{Item::BeyondHor::RADiffeo}]
This is  contained in \cref{Thm::QuantRes::MainThm} \cref{Item::QuantRes::PhiOpen} and \cref{Item::QuantRes::PhiDiffeo}, except that \cref{Thm::QuantRes::MainThm}  \cref{Item::QuantRes::PhiDiffeo} only guarantees that $\Phi_{x,\delta}$ is a $C^2$ diffeomorphism, not a real analytic diffeomorphism.  However, using \cref{Item::BeyongHor::Pullbacks} (which we have already proved), $Y_1^{x,\delta},\ldots, Y_q^{x,\delta}$ form a real analytic
spanning set for $T_u B^{n_0(x)}(\eta_1)$ ($\forall u\in B^{n_0(x)}(\eta_1)$).  Since $(\Phi_{x,\delta})_{*}Y_j^{x,\delta}= \delta^{d_j} X_j$ where $X_1,\ldots, X_q$ are real analytic on $L_x$, \cref{Lemma::PfQual::ShowDiffeoRA}
shows $\Phi_{x,\delta}$ is real analytic.
\end{proof}

\begin{lemma}\label{Lemma::ScalingRevis::DensAssumpHold}
There exists $r>0$, independent of $x\in \Compact$, $\delta\in (0,1]$, and functions $f_l^{x,\delta}$ ($1\leq l\leq n_0(x)$) such that $\forall x\in \Compact$, $\delta\in(0,1]$,
\begin{equation*}
	\Lie{X^{\delta}_{j_l(x,\delta)}} \nu_x = f_l^{x,\delta} \nu_x, \quad 1\leq l\leq n_0(x),
\end{equation*}
where $f_l^{x,\delta}\in \AXSpace{X^{\delta}_{J_0(x,\delta)}}{x}{r}\cap \CSpace{B_{X^{\delta}_{J_0(x,\delta)}}(x,\chi)}$ with
\begin{equation*}
	\AXNorm{f_l^{x,\delta}}{X^\delta_{J_0(x,\delta)}}{x}{r}, \CNorm{f_l^{x,\delta}}{B_{X^{\delta}_{J_0(x,\delta)}}(x,\chi)} \lesssim 1, \quad 1\leq l\leq n_0(x).
\end{equation*}
\end{lemma}
\begin{proof}
It is immediate to verify that the hypotheses of \cref{Prop::EuclidDense::MainProp}, when applied to the vector fields $X^\delta$, hold uniformly for $x\in \Compact$, $\delta\in (0,1]$.
Thus, $\Eadmis$-admissible constants in that theorem can be chosen independent of $x\in \Compact$, $\delta\in (0,1]$.  Because of this, the lemma is an immediate consequence
of \cref{Prop::EuclidDense::MainProp}.
\end{proof}

In light of \cref{Lemma::ScalingRevis::DensAssumpHold}, the assumptions of \cref{Thm::Density::MainThm} and \cref{Cor::Density::MainCor} hold, when applied to the vector fields
$X^\delta$ and the density $\nu_x$, at the base point $x$, uniformly for $x\in \Compact$ and $\delta\in (0,1]$.  Thus, any constant which is $\nu$-admissible (or $0;\nu$-admissible) in the
sense of those results can be chosen independent of $x\in \Compact$, $\delta\in (0,1]$.

Since $\nu_x$ is the induced Lebesgue density on an $n_0(x)$-dimensional injectively immersed submanifold of $\R^N$, we have
\begin{equation}\label{Eqn::ScalingRevis::ComputeDensityDet}
	\nu_x(Z_1\wedge \cdots \wedge Z_{n_0(x)}) = \mleft|\det_{n_0(x)\times n_0(x)} \mleft( Z_1 | \cdots | Z_{n_0(x)}\mright)\mright|,
\end{equation}
where $Z_1,\ldots, Z_{n_0(x)}$ are vector fields tangent to $L_x$ and $(Z_1|\cdots | Z_{n_0(x)})$ denotes the $N\times n_0(x)$ matrix with columns given by $Z_1,\ldots, Z_{n_0(x)}$.

\begin{proof}[Proof of  \cref{Thm::Scaling::BeyondHor} \cref{Item::BeyondHor::Estimateh}]
This follows from \cref{Thm::Density::MainThm} and \cref{Cor::Density::MainCor}, using \cref{Eqn::ScalingRevis::ComputeDensityDet}.
\end{proof}

\begin{proof}[Proof of \cref{Thm::Scaling::BeyondHor} \cref{Item::BeyondHor::Estimatenux}]
By \cref{Rmk::ScalingRevis::xi0xi2}, \cref{Cor::Density::MainCor}, and \cref{Eqn::ScalingRevis::ComputeDensityDet}, we have for $\delta\in (0,1]$,
\begin{equation}\label{Eqn::ScalingRevis::Esimtatenux::1}
\begin{split}
	&\nu_x(B_{(X,d)}(x,\xi_0\delta)) \leq \nu_x(B_{X^\delta}(x,\xi_2)) \approx \max_{k_1,\ldots, k_{n_0(x)}\in \{1,\ldots, q\} } \nu_x (\delta^{d_{k_1}} X_{k_1},\ldots, \delta^{d_{k_{n_0(x)}}} X_{k_{n_0(x)}})
	\\&\approx \mleft| \det_{n_0(x)\times n_0(x)} \delta^{d} X(x) \mright|_{\infty}
	\approx \mleft| \det_{n_0(x)\times n_0(x)} (\xi_0 \delta)^{d} X(x) \mright|_{\infty},
\end{split}
\end{equation}
where the last $\approx$ follows immediately from the formula for $\mleft| \det_{n_0(x)\times n_0(x)} \delta^{d} X(x) \mright|_{\infty}$ and the fact that $\xi_0\approx 1$.

Similarly, since $\xi_2\leq 1$, we have by  \cref{Cor::Density::MainCor} and \cref{Eqn::ScalingRevis::ComputeDensityDet},
\begin{equation}\label{Eqn::ScalingRevis::Esimtatenux::2}
\begin{split}
	&\nu_x(B_{(X,d)}(x,\delta)) \geq \nu_x(B_{X^\delta}(x,\xi_2)) \approx \max_{k_1,\ldots, k_{n_0(x)}\in \{1,\ldots, q\} } \nu_x (\delta^{d_{k_1}} X_{k_1},\ldots, \delta^{d_{k_{n_0(x)}}} X_{k_{n_0(x)}})
	\\&\approx \mleft| \det_{n_0(x)\times n_0(x)} \delta^{d} X(x) \mright|_{\infty}
\end{split}
\end{equation}
Combining \cref{Eqn::ScalingRevis::Esimtatenux::1} and \cref{Eqn::ScalingRevis::Esimtatenux::2} completes the proof.
\end{proof}

\begin{proof}[Proof of \cref{Thm::Scaling::BeyondHor} \cref{Item::BeyondHor::Doubling}]
This follows immediately from the formula given in \cref{Thm::Scaling::BeyondHor} \cref{Item::BeyondHor::Estimatenux}.
\end{proof}

	\subsection{Multi-parameter geometries}\label{Section::MultiParam}
The results in \cref{Section::Scaling::BeyondHormander} concerned single-parameter sub-Riemannian geometries.  \Cref{Thm::Scaling::BeyondHor} can be generalized to the setting of multi-parameter
geometries with essentially the same proof.  We outline these ideas in this section.  Such multi-parameter geometries arise in applications:  see \cite{SteinStreetIII,StreetMultiParamSingInt}.

Let $V_1,\ldots, V_r$ be real analytic vector fields defined on an open set $\Omega\subseteq \R^n$.  Fix $\nu\in \N$ and to each $V_j$ assign a formal degree $0\ne e_j\in \N^\nu$.
If $Z$ has formal degree $e\in \N^\nu$, we assign to $[V_j,Z]$ the formal degree $e+e_j$.  Fix $m\in \N$ a large integer and let $(X_1,d_1),\ldots, (X_q,d_q)$ denote the finite list
of vector fields with $\nu$-parameter formal degree $d_j\in \N^\nu$ with $|d_j|_{\infty}\leq m$.  The results which follow are essentially independent of $m$, so long as $m$ is chosen sufficiently
large (depending on $(V_1,e_1),\ldots, (V_r,e_r)$ and $\Compact$).  For $\delta\in (0,1]^{\nu}$, we let $\delta^{d}X$ denote the list of vector fields $\delta^{d_1}X_1,\ldots, \delta^{d_q}X_q$.
We sometimes identify $\delta^d X$ with the $n\times q$ matrix $(\delta^{d_1} X_1|\cdots|\delta^{d_q} X_q)$.
As before, we set $B_{(X,d)}(x,\delta):=B_{\delta^{d} X}(x,1)$, though now $\delta\in (0,1]^{\nu}$.

As in \cref{Section::Scaling::BeyondHormander}, the involutive distribution generated by $V_1,\ldots, V_r$ foliates $\Omega$ into leaves, and we let $L_x$ denote the leaf passing through $x$,
and $\nu_x$ the induced Lebesgue density on $L_x$.  $B_{(X,d)}(x,\delta)$ is an open subset of $L_x$.

For each $x\in \Omega$ set $n_0(x):=\dim \Span\{X_1(x),\ldots, X_q(x)\}$.  For each $x\in \Omega$, $\delta\in (0,1]^\nu$, pick $j_1=j_1(x,\delta),\ldots, j_{n_0(x)}=j_{n_0(x)}(x,\delta)$ so that
\begin{equation*}
\mleft| \det_{n_0(x)\times n_0(x)} \mleft( \delta^{d_{j_1}} X_{j_1}(x) | \cdots | \delta^{d_{j_{n_0(x)}}} X_{j_{n_0(x)}}(x)\mright) \mright|_{\infty} = \mleft| \det_{n_0(x)\times n_0(x)} \delta^{d} X\mright|_{\infty}.
\end{equation*}
For this choice of $j_1=j_1(x,\delta),\ldots, j_{n_0(x)}=j_{n_0(x)}(x,\delta)$ set (writing $n_0$ for $n_0(x)$):
\begin{equation*}
\Phi_{x,\delta}(t_1,\ldots,t_{n_0}) := \exp\mleft( t_1 \delta^{d_{j_1}} X_{j_1}+\cdots+ t_{n_0} \delta^{d_{j_{n_0}}} X_{j_{n_0}} \mright)x.
\end{equation*}

\begin{thm}
Fix a compact set $\Compact\Subset \Omega$ and $x\in \Compact$, take $m$ sufficiently large (depending on $\Compact$ and $(V_1,e_1),\ldots, (V_r,e_r)$), and define $(X_1,d_1),\ldots, (X_q,d_q)$ as above.
Define $n_0(x)$, $\nu_x$, and $\Phi_{x,\delta}(t_1,\ldots, t_{n_0})$ as above.
We write
$A\lesssim B$ for $A\leq C B$ where $C$ is a positive constant which may depend on $\Compact$, but does not depend on the particular points $x\in \Compact$ and $u\in \R^{n_0(x)}$ under consideration,
or on the scale $\delta\in (0,1]^\nu$; we write $A\approx B$ for $A\lesssim B$ and $B\lesssim A$.  There exist $\eta_0,\xi_0\approx 1$ such that $\forall x\in \Compact$,
\begin{enumerate}[(i)]
\item
$\nu_x(B_{(X,d)}(x,\delta))\approx \mleft| \det_{n_0(x)\times n_0(x)} \delta^{d} X(x) \mright|_{\infty}$, $\forall \delta\in (0,1]^{\nu}$, $|\delta|\leq \xi_0$.
\item
$\nu_x(B_{(X,d)}(x,2\delta))\lesssim \nu_x(B_{(X,d)}(x,\delta))$, $\forall \delta\in (0,1]^{\nu}$, $|\delta|\leq \xi_0/2$.
\item
$\forall \delta\in (0,1]^\nu$, $\Phi_{x,\delta}(B^{n_0(x)}(\eta_1))\subseteq L_x$ is open and $\Phi_{x,\delta}:B^n(\eta_1)\rightarrow \Phi_{x,\delta}(B^n(\eta_1))$ is a real analytic diffeomorphism.
\item
For $\delta\in (0,1]^{\nu}$, define $h_{x,\delta}(t)$ on $B^{n_0(x)}(\eta_1)$ by $h_{x,\delta}\LebDensity = \Phi_{x,\delta}^{*} \nu_x$.  Then, $h_{x,\delta}(t)\approx \mleft| \det_{n_0(x)\times n_0(x)} \delta^{d} X(x) \mright|_{\infty}$,
$\forall t\in B^{n_0(x)}(\eta_1)$, and there exists $s\approx 1$ with $\ANorm{h_{x,\delta}}{n_0(x)}{s}\lesssim  \mleft| \det_{n_0(x)\times n_0(x)} \delta^{d} X(x) \mright|_{\infty}$.
\item
$B_{(X,d)}(x,\xi_0\delta)\subseteq \Phi_{x,\delta}(B^{n_0(x)}(\eta_1))\subseteq B_{(X,d)}(x,\delta)$, $\forall \delta\in (0,1]^{\nu}$.
\item
For $\delta\in (0,1]^{\nu}$, $x\in \Compact$, let $Y_j^{x,\delta}:=\Phi_{x,\delta}^{*} \delta^{d_j} X_j$, so that $Y_j^{x,\delta}$ is a real analytic vector field on $B^{n_0(x)}(\eta_1)$.  We have
\begin{equation*}
\ANorm{Y_j^{x,\delta}}{n_0(x)}{\eta_1}[\R^n]\lesssim 1, \quad 1\leq j\leq q.
\end{equation*}
Finally, $Y_1^{x,\delta}(u),\ldots, Y_q^{x,\delta}(u)$ span $T_u B^{n_0(x)}(\eta_1)$, uniformly in $x$, $\delta$, and $u$, in the sense that
\begin{equation*}
\max_{k_1,\ldots, k_{n_0(x)}\in \{1,\ldots, q\}} \inf_{u\in B^{n_0(x)}(\eta_1)} \mleft| \det\mleft( Y_{k_1}^{x,\delta}(u)| \cdots | Y_{k_{n_0(x)}}^{x,\delta}(u)  \mright) \mright|\approx 1.
\end{equation*}
\end{enumerate}
\end{thm}
\begin{proof}[Comments on the proof]
The proof is nearly identical to the proof of \cref{Thm::Scaling::BeyondHor}; anywhere in the proof where one writes $d\leq e$, where now $d,e\in \N^\nu$, the inequality means $d_\mu\leq e_\mu$, $\forall 1\leq \mu\leq \nu$.
A main change needed is that the set $\sS$ consists of vector fields paired with formal degrees in $\N^\nu$, instead of formal degrees in $\N$.
To deal with this one needs to generalize \cref{Lemma::ScalingRevis::NoetherianWithFormal} to deal with $\sS\subset \sA_n^n\times \N^\nu$; the same proof works by treating $t\in \R^\nu$ and using each degree as a multi-index.
With these modifications it is straightforward to adapt the proof to yield this more general result.  We leave the details to the interested reader.
\end{proof}



\bibliographystyle{amsalpha}

\bibliography{coords}

\center{\it{University of Wisconsin-Madison, Department of Mathematics, 480 Lincoln Dr., Madison, WI, 53706}}

\center{\it{street@math.wisc.edu}}

\center{MSC 2010:  58A30 (Primary), 32C05 and 53C17 (Secondary)}

\end{document}